\documentclass[12pt]{article}
\usepackage{color}
\usepackage{amsmath}
\usepackage[colorlinks=true]{hyperref}
\usepackage{microtype}
\usepackage{amscd,amsthm,amsfonts,amssymb,esint}
\usepackage{fullpage}
\usepackage[all]{xy}
\usepackage{mathtools}
\usepackage{xcolor}
\usepackage{mathrsfs}
\usepackage{amsmath}
\usepackage[all]{xy}
\usepackage{geometry}
\usepackage{array}

\usepackage{colonequals}
\usepackage{tikz-cd}
\usepackage{comment}
\usepackage{setspace}
\usepackage{enumerate}

\begin{document}

\title{%
  \makebox[\textwidth][c]{%
    \parbox{1.15\textwidth}{%
      \centering
       Exceptional maps and abelian points in backward orbits
    }%
  }%
}

\author{Zhuchao Ji, Jiarui Song, Junyi Xie}

\maketitle

\begin{abstract}
Consider a number field $K$, a polarized endomorphism $f:X\to X$ on a normal
projective variety, and a non-exceptional point $\alpha$ for $f$.
We show that the extension $K(f^{-\infty}(\alpha))/K$ generated by the backward orbit of $\alpha$ is virtually abelian if and only if $f$ is an exceptional map with CM and $\alpha$ is $f$-preperiodic.
This answers a conjecture of Andrews and Petsche and extends it to higher dimensions. To this end, we develop the theory of exceptional maps on normal
projective varieties over $\mathbb{C}$ and prove that commuting polarized
endomorphisms of multiplicatively independent degrees are exceptional.
\end{abstract}

\theoremstyle{plain}
\newtheorem{thm}{Theorem}[section]
\newtheorem{theorem}[thm]{Theorem}
\newtheorem{cor}[thm]{Corollary}
\newtheorem{corollary}[thm]{Corollary}
\newtheorem{lem}[thm]{Lemma}
\newtheorem{lemma}[thm]{Lemma}
\newtheorem{pro}[thm]{Proposition}
\newtheorem{proposition}[thm]{Proposition}
\newtheorem{prop}[thm]{Proposition}

\newtheorem{assumption}[thm]{Assumption}
\newtheorem{conjecture}[thm]{Conjecture}
\newtheorem{sublemma}[thm]{Sub-lemma}

\theoremstyle{definition}
\newtheorem{definition}[thm]{Definition}

\theoremstyle{remark}
\newtheorem{remark}[thm]{Remark}
\newtheorem{example}[thm]{Example}
\newtheorem{remarks}[thm]{Remarks}
\newtheorem{problem}[thm]{Problem}
\newtheorem{exercise}[thm]{Exercise}
\newtheorem{situation}[thm]{Situation}
\newtheorem{Question}[thm]{Question}

\newcommand{\Zhuchao}[1]{ \sf $\clubsuit\clubsuit\clubsuit$ Zhuchao : [#1]}
\newcommand{\SH}{\rm SH}
\newcommand{\Cov}{\rm Cov}
\newcommand{\Tan}{\rm Tan}
\newcommand{\res}{\rm res}
\newcommand{\Om}{\Omega}
\newcommand{\om}{\omega}
\newcommand{\La}{\Lambda}
\newcommand{\la}{\lambda}

\newcommand{\surj}{\twoheadrightarrow}
\newcommand{\inj}{\hookrightarrow}
\newcommand{\zar}{{\rm zar}}
\newcommand{\Exc}{{\rm Exc}}
\newcommand{\Mod}{{\rm Mod}}

\newcommand{\codim}{{\rm codim}}

\newcommand{\Leb}{{\rm Leb}}
\newcommand{\rank}{{\rm rank}}
\newcommand{\geom}{{\rm geom}}
\newcommand{\Ker}{{\rm Ker \ }}

\newcommand{\Der}{{\rm Der}}
\newcommand{\Div}{{\rm Div}}

\newcommand{\Corr}{{\rm Corr}}

\newcommand{\Spec}{{\rm Spec \,}}
\newcommand{\Nef}{{\rm Nef \,}}
\newcommand{\Frac}{{\rm Frac \,}}
\newcommand{\Sing}{{\rm Sing}}
\newcommand{\sing}{{\rm sing}}
\newcommand{\reg}{{\rm reg}}
\newcommand{\Char}{{\rm char\,}}
\newcommand{\Tr}{{\rm Tr}}

\newcommand{\bif}{{\rm bif}}
\newcommand{\AS}{{\rm AS}}
\newcommand{\loc}{{\rm loc}}
\newcommand{\FS}{{\rm FS}}
\newcommand{\CE}{{\rm CE}}
\newcommand{\PCE}{{\rm PCE}}
\newcommand{\WR}{{\rm WR}}
\newcommand{\PR}{{\rm PR}}
\newcommand{\TCE}{{\rm TCE}}
\newcommand{\diam}{{\rm diam\,}}
\newcommand{\id}{{\rm id}}
\newcommand{\NE}{{\rm NE}}

\newcommand{\Min}{{\rm Min \ }}
\newcommand{\Hol}{{\rm Hol \ }}
\newcommand{\Rat}{{\rm Rat}}
\newcommand{\Hdim}{{\rm Hdim}}
\newcommand{\dist}{{\rm dist}}
\newcommand{\FL}{{\rm FL}}
\newcommand{\fm}{{\rm fm}}
\newcommand{\vol}{{\rm vol}}

\newcommand{\Max}{{\rm Max \ }}
\newcommand{\Alb}{{\rm Alb}\,}
\newcommand{\Aff}{{\rm Aff}\,}
\newcommand{\GL}{{\rm GL}\,}        
\newcommand{\PGL}{{\rm PGL}\,}
\newcommand{\Bir}{{\rm Bir}}
\newcommand{\Bif}{{\rm Bif}}
\newcommand{\Aut}{{\rm Aut}}
\newcommand{\topo}{{\rm top}}
\newcommand{\End}{{\rm End}}
\newcommand{\Per}{{\rm Per}\,}
\newcommand{\Preper}{{\rm Preper}\,}
\newcommand{\Preim}{{\rm Preim}\,}
\newcommand{\ie}{{\it i.e.\/},\ }
\newcommand{\niso}{\not\cong}
\newcommand{\nin}{\not\in}
\newcommand{\soplus}[1]{\stackrel{#1}{\oplus}}
\newcommand{\by}[1]{\stackrel{#1}{\rightarrow}}
\newcommand{\longby}[1]{\stackrel{#1}{\longrightarrow}}
\newcommand{\vlongby}[1]{\stackrel{#1}{\mbox{\large{$\longrightarrow$}}}}
\newcommand{\ldownarrow}{\mbox{\Large{\Large{$\downarrow$}}}}
\newcommand{\lsearrow}{\mbox{\Large{$\searrow$}}}
\renewcommand{\d}{\stackrel{\mbox{\scriptsize{$\bullet$}}}{}}
\newcommand{\dlog}{{\rm dlog}\,}    
\newcommand{\longto}{\longrightarrow}
\newcommand{\vlongto}{\mbox{{\Large{$\longto$}}}}
\newcommand{\limdir}[1]{{\displaystyle{\mathop{\rm lim}_{\buildrel\longrightarrow\over{#1}}}}\,}
\newcommand{\liminv}[1]{{\displaystyle{\mathop{\rm lim}_{\buildrel\longleftarrow\over{#1}}}}\,}
\newcommand{\norm}[1]{\mbox{$\parallel{#1}\parallel$}}
\newcommand{\boxtensor}{{\Box\kern-9.03pt\raise1.42pt\hbox{$\times$}}}
\newcommand{\into}{\hookrightarrow}
\newcommand{\image}{{\rm image}\,}
\newcommand{\Lie}{{\rm Lie}\,}      
\newcommand{\CM}{\rm CM}
\newcommand{\Ma}{\mathbf{M}}
\newcommand{\Teich}{\rm Teich\;}
\newcommand{\genus}{{\rm genus}}
\newcommand{\gonality}{{\rm gonal}}
\newcommand{\sext}{\mbox{${\mathcal E}xt\,$}}  
\newcommand{\shom}{\mbox{${\mathcal H}om\,$}}  
\newcommand{\coker}{{\rm coker}\,}  
\newcommand{\sm}{{\rm sm}}
\newcommand{\pgcd}{\text{pgcd}}
\newcommand{\trd}{\text{tr.d.}}
\newcommand{\tensor}{\otimes}
\newcommand{\hotimes}{\hat{\otimes}}

\newcommand{\CH}{{\rm CH}}
\newcommand{\tr}{{\rm tr}}
\newcommand{\e}{\rm SH}

\renewcommand{\iff}{\mbox{ $\Longleftrightarrow$ }}
\newcommand{\supp}{{\rm supp}\,}
\newcommand{\esssup}{{\rm ess\,sup}}
\newcommand{\ext}[1]{\stackrel{#1}{\wedge}}
\newcommand{\onto}{\mbox{$\,\>>>\hspace{-.5cm}\to\hspace{.15cm}$}}
\newcommand{\propsubset}
{\mbox{$\textstyle{
			\subseteq_{\kern-5pt\raise-1pt\hbox{\mbox{\tiny{$/$}}}}}$}}
\newcommand{\sA}{{\mathcal A}}
\newcommand{\sB}{{\mathcal B}}
\newcommand{\sC}{{\mathcal C}}
\newcommand{\sD}{{\mathcal D}}
\newcommand{\sE}{{\mathcal E}}
\newcommand{\sF}{{\mathcal F}}
\newcommand{\sG}{{\mathcal G}}
\newcommand{\sH}{{\mathcal H}}
\newcommand{\sI}{{\mathcal I}}
\newcommand{\sJ}{{\mathcal J}}
\newcommand{\sK}{{\mathcal K}}
\newcommand{\sL}{{\mathcal L}}
\newcommand{\sM}{{\mathcal M}}
\newcommand{\sN}{{\mathcal N}}
\newcommand{\sO}{{\mathcal O}}
\newcommand{\sP}{{\mathcal P}}
\newcommand{\sQ}{{\mathcal Q}}
\newcommand{\sR}{{\mathcal R}}
\newcommand{\sS}{{\mathcal S}}
\newcommand{\sT}{{\mathcal T}}
\newcommand{\sU}{{\mathcal U}}
\newcommand{\sV}{{\mathcal V}}
\newcommand{\sW}{{\mathcal W}}
\newcommand{\sX}{{\mathcal X}}
\newcommand{\sY}{{\mathcal Y}}
\newcommand{\sZ}{{\mathcal Z}}
\newcommand{\A}{{\mathbb A}}
\newcommand{\B}{{\mathbb B}}
\newcommand{\C}{{\mathbb C}}
\newcommand{\D}{{\mathbb D}}
\newcommand{\E}{{\mathbb E}}
\newcommand{\F}{{\mathbb F}}
\newcommand{\G}{{\mathbb G}}
\newcommand{\HH}{{\mathbb H}}
\newcommand{\LL}{{\mathbb L}}
\newcommand{\J}{{\mathbb J}}
\newcommand{\M}{{\mathbb M}}
\newcommand{\N}{{\mathbb N}}
\renewcommand{\P}{{\mathbb P}}
\newcommand{\Q}{{\mathbb Q}}
\newcommand{\R}{{\mathbb R}}
\newcommand{\T}{{\mathbb T}}
\newcommand{\U}{{\mathbb U}}
\newcommand{\V}{{\mathbb V}}
\newcommand{\W}{{\mathbb W}}
\newcommand{\X}{{\mathbb X}}
\newcommand{\Y}{{\mathbb Y}}
\newcommand{\Z}{{\mathbb Z}}
\newcommand{\ch}{{\mathbbm {1}}}
\newcommand{\bk}{{\mathbf{k}}}

\newcommand{\bp}{{\mathbf{p}}}
\newcommand{\ep}{\varepsilon}
\newcommand{\bbk}{{\overline{\mathbf{k}}}}
\newcommand{\Fix}{\mathrm{Fix}}

\newcommand{\tor}{{\mathrm{tor}}}
\renewcommand{\div}{{\mathrm{div}}}

\newcommand{\trdeg}{{\mathrm{trdeg}}}
\newcommand{\Stab}{{\mathrm{Stab}}}

\newcommand{\OK}{{\overline{K}}}
\newcommand{\ok}{{\overline{k}}}

\newcommand{\cf}{[c.f. ?]}
\newcommand{\jy}{jy:}

\numberwithin{equation}{section}
\newcommand{\Gm}{\mathbb G_m}

\newcommand\mO{\mathcal{O}}
\newcommand{\mf}[1]{\mathfrak{#1}}
\newcommand{\ms}[1]{\mathscr{#1}}
\newcommand{\mb}[1]{\mathbb{#1}}
\newcommand{\mr}[1]{\mathrm{#1}}
\newcommand{\mc}[1]{\mathcal{#1}}
\newcommand{\ov}{\overline}
\newcommand{\Zar}[1]{\overline{#1}^{\mathrm{Zar}}}
\newcommand\ra{\rightarrow}
\newcommand{\ora}[1]{\stackrel{#1}\longrightarrow}
\newcommand{\cra}{\stackrel{\sim}\rightarrow}
\newcommand\mheight{\mathrm{height\,}}
\newcommand\wt{\widetilde}
\newcommand\im{\mathrm{im\,}}
\newcommand\trans{{}^{\top}}
\newcommand\Hom{\mathrm{Hom}}
\newcommand\Spe{\mathrm{Spec\,}}
\newcommand\di{\mathrm{div}}
\newcommand\red{\mathrm{red}}
\newcommand\Pro{\mathrm{Proj\,}}
\newcommand\sHom{\mathscr{H}om}
\newcommand\Supp{\mathrm{Supp\,}}
\newcommand\Ann{\mathrm{Ann\,}}
\newcommand\Ext{\mathrm{Ext}}
\newcommand\sExt{\mathscr{E}xt}
\newcommand\cdim{\mathrm{codim}}
\newcommand\Cl{\mathrm{Cl\,}}
\newcommand\Pic{\mathrm{Pic\,}}
\newcommand\ord{\mathrm{ord}}
\newcommand\Gal{\mathrm{Gal}}
\newcommand\ab{\mathrm{ab}}
\newcommand\an{\mathrm{an}}
\newcommand\dra{\dashrightarrow}
\newcommand\reldeg{\mathrm{reldeg}}
\newcommand\Crit{\mathrm{Crit}}
\newcommand{\mib}[1]{\textit{\textbf{#1}}}
\newcommand\bqed{\hfill $\blacksquare$}

\numberwithin{equation}{subsection}

\tableofcontents

\section{Introduction}\label{sec_intro}
Let $K$ be a number field, let $X$ be a normal projective variety over
$K$, and let $f\colon X\to X$ be a polarized endomorphism. Thus
$f^*L\simeq L^{\otimes d}$ for an ample line bundle $L$ and an integer
$d\geq 2$. Fix an algebraic closure $\overline K$ of $K$. For
$\alpha\in X(K)$ and $n\geq 1$, we set
\[f^{-n}(\alpha):=\left\{\beta\in X(\ov{K})\mid f^n(\beta)=\alpha\right\}\subset X(\ov{K}),\]
and $f^{-\infty}(\alpha)=\bigcup\limits_{n=1}^{\infty}f^{-n}(\alpha)$. Let $K_n(f,\alpha)$ be the field generated over $K$ by the points of $f^{-n}(\alpha)$. These are finite Galois extensions satisfying
\[
 K\subseteq K_1(f,\alpha)\subseteq K_2(f,\alpha)\subseteq\cdots.
\]
We write
\[
 K_\infty(f,\alpha):=\bigcup_{n\geq 1}K_n(f,\alpha),
\]
The abelian case asks for a classification of the pairs $(f,\alpha)$
for which $f^{-\infty}(\alpha)\subseteq X(K^{\mathrm{ab}})$, where $K^{\mathrm{ab}}$ is the maximal abelian extension of $K$ in $\overline K$. Equivalently, we ask when
$\Gal(K_\infty(f,\alpha)/K)$ is abelian. The study of these Galois groups goes back to Odoni's work on polynomial iterates \cite{Odo85} and is now part of the theory of arboreal Galois representations; see \cite{Jon13}.

Our main theorem gives a complete characterization of the pairs
$(f,\alpha)$ for which the extension $K_{\infty}(f,\alpha)/K$
is virtually abelian; see Theorem \ref{thm_vaeq1} for the precise
statement. Specializing to dimension one, we obtain the following
conjecture of Andrews and Petsche \cite{AP20} as a direct consequence;
see Theorem \ref{thm_onedim}. This conjecture characterizes the pairs
$(f,\alpha)$ for which $\Gal(K_{\infty}(f,\alpha)/K)$ is abelian
when $f$ is a polynomial map on the projective line.

To state the conjecture, recall that a point is \emph{exceptional}
for a rational map on $\mathbb P^1$ if its backward orbit is finite.
For an extension $F/K$, two pairs $(f,\alpha)$ and $(g,\beta)$
are \emph{$F$-conjugate} if there exists
$\varphi\in\operatorname{PGL}_2(F)$ such that
$g=\varphi^{-1}\circ f\circ\varphi$ and $\alpha=\varphi(\beta)$.

\begin{conjecture}[Andrews--Petsche; {\cite{AP20,FOZ24}}]\label{conj_P1}
Let $K$ be a number field and $f\in K[x]$ be a polynomial map of degree $d\geq 2$. Let $\alpha\in K$ be a non-exceptional point for $f$. Then $\Gal(K_{\infty}(f,\alpha)/K)$ is abelian if and only if the pair $(f,\alpha)$ is $K^{\mr{ab}}$-conjugate to a pair $(g,\beta)$ of the following two cases:
\begin{enumerate}
\renewcommand{\labelenumi}{(\theenumi)}
    \item $g(x)=x^{d}$ and $\beta$ is a root of unity in $\ov{K}$;
    \item $g(x)=\pm T_d(x)$ is the $d$-th Chebyshev polynomial and $\beta=\zeta+\zeta^{-1}$, where $\zeta$ is a root of unity in $\ov{K}$.
\end{enumerate}
\end{conjecture}
These examples arise from the multiplicative group. The signed Chebyshev maps
are obtained from $z\mapsto\pm z^d$ using $z\mapsto z+z^{-1}$. The backward orbits in these examples consist of roots of unity or points $\zeta+\zeta^{-1}$, where $\zeta$ is a root of unity. The following cases of Conjecture \ref{conj_P1} have been established:
\begin{enumerate}
\renewcommand{\labelenumi}{(\theenumi)}
    \item Andrews and Petsche \cite{AP20} proved the conjecture
    for maps geometrically conjugate to $x^d$ or $T_d$.

    \item For quadratic polynomials over $\mathbb Q$,
    Andrews and Petsche \cite{AP20} proved the conjecture
    under the stability assumption that $f^n(x)-\alpha$
    is irreducible for every $n\geq 1$.
    Ferraguti and Pagano \cite{FP20} subsequently removed
    this assumption.

    \item For arbitrary polynomials over $\mathbb Q$, the conjecture follows from Ostafe's work \cite{Ost17} and the Kronecker--Weber theorem, as established in
    \cite[Corollary~3.7]{FOZ24}.

    \item Over number fields, Ferraguti and Pagano \cite{FP23} proved further cases, including unicritical polynomials whose finite critical point is periodic.
\end{enumerate}
We refer to \cite{Fer22} for a survey of the earlier approaches.

For rational maps $f\in K(x)$ of degree at least $2$ over a
number field $K$, two results give necessary conditions for
$\Gal(K_\infty(f,\alpha)/K)$ to be abelian:
\begin{enumerate}
\renewcommand{\labelenumi}{(\theenumi)}
    \item 
    Ferraguti, Ostafe, and Zannier \cite{FOZ24} proved that,
    if $\alpha\in\mathbb P^1(K)$ is non-exceptional and
    $\Gal(K_\infty(f,\alpha)/K)$ is abelian, then $f$ is
    postcritically finite, meaning that every critical point
    is preperiodic.

    \item Leung and Petsche \cite{LP25b} proved that, if $f$ is
    postcritically finite and $\alpha\in\mathbb P^1(K)$ is
    not preperiodic, then $\Gal(K_\infty(f,\alpha)/K)$
    is nonabelian.
\end{enumerate}
Together, these results show that if $\alpha$ is non-exceptional and
$\Gal(K_\infty(f,\alpha)/K)$ is abelian, then $f$ is postcritically
finite and $\alpha$ is preperiodic. The rational setting also includes Latt\`es maps arising from
elliptic curves with complex multiplication. Their abelian
backward orbits are characterized in
\cite[Theorem~D]{FOZ24}. Related results include Leung's work
\cite{Leu25} on rational maps with nonreal Julia sets and
Leung--Petsche's work \cite{LP25a} on the Minkowski dimension
of arboreal Galois groups.

A related rigidity phenomenon concerns preperiodic points in the
cyclotomic closure $K^c:=K(\mu_\infty)$, where $\mu_\infty$ denotes
the group of roots of unity in $\overline K$. Dvornicich and Zannier
\cite[Theorem~2]{DZ07} proved that a polynomial of degree $d\geq 2$
with infinitely many preperiodic points in $K^c$ is geometrically
conjugate to $x^d$ or $\pm T_d$. In higher dimension, Ji, Xie, and
Zhang \cite{JXZ25} established analogous results for regular
polynomial endomorphisms of $\mathbb A^N$, that is, polynomial maps
extending to endomorphisms of $\mathbb P^N$. If either the cyclotomic
preperiodic points or the cyclotomic points in a single backward
orbit are Zariski dense, then some iterate is semiconjugate to a
surjective group endomorphism of $\mathbb G_m^N$ through a dominant
morphism $\mathbb G_m^N\to\mathbb A^N$ over $\overline K$.

\subsection{Main results}

Our main theorem classifies the pairs $(f,\alpha)$ for which
$\Gal(K_{\infty}(f,\alpha)/K)$ is virtually abelian, where $f$ is a
polarized endomorphism of a normal projective variety and $\alpha$
is non-exceptional. To state the result, we first introduce a notion
of exceptionality for such endomorphisms.

\subsubsection*{Exceptional maps}
A rational map of degree $d\geq 2$ on $\mb{P}^1_{\mb{C}}$ is called
\emph{exceptional} if it is conjugate over $\mb{C}$ to a power map
$x\mapsto x^{\pm d}$, a signed Chebyshev map $x\mapsto \pm T_d(x)$,
or a Latt\`es map. Such maps come from group structure.

\medskip

In higher dimensions, there is currently no generally accepted generalization of the notion of exceptional maps.
In \cite{DS02},Dinh and Sibony 
introduced a notion of exceptionality for endomorphisms of
$\mb{P}^n_{\mb{C}}$ and proved that commuting endomorphisms of
multiplicatively independent degrees are exceptional in this sense. 
However, their definition of exceptionality is analytic.
In this paper,
we develop a theory of exceptional maps on arbitrary normal
projective varieties.
Our definition is purely algebraic, and it coincides with the definition of Dinh–Sibony in $\mb{P}^n_{\mb{C}}$.

Our definition uses torsors under tori over
abelian varieties.

\begin{definition}
Let $\mathbf{k}$ be a field, let $A$ be an abelian variety over
$\mathbf{k}$, and let $T$ be a torus over $\mathbf{k}$. If
$\pi:Y\to A$ is a $T$-torsor, we call the triple $(Y,A,\pi)$ an
\emph{AT-variety}. Here, we allow $A$ to be a point.
\end{definition}

\begin{definition}\label{introdefidlift}
	Let $d\geq 2$ be an integer, and let $(Y,A,\pi)$ be an AT-variety with associated torus $T$. A self-isogeny $\psi : Y\to Y$ is called a \emph{$d$-lift} if the induced morphism $\psi_A:A\to A$ is a $d$-polarized isogeny and $\psi_T=[d]_T$, where
	$$
	[d]_T :T\longrightarrow T,\qquad t\longmapsto t^d
	$$
	is the multiplication-by-$d$ isogeny of $T$.
\end{definition}

We define exceptional maps to be the maps come from $d$-lifts of AT-varieties.
\begin{definition}\label{def:exceptional-intro}
Let $\mathbf{k}$ be a field of characteristic zero, let $X$ be a
normal projective variety over $\mathbf{k}$, and let $f:X\to X$ be
a $d$-polarized endomorphism. We call $f$ a \emph{$\mathbf{k}$-exceptional map} if there exist an AT-variety $(Y,A,\pi)$ over $\mathbf{k}$, a dominant morphism $h:Y\to X$, and
a $d$-lift $\psi:Y\to Y$ such that $h$ is finite onto a Zariski open dense subset $X^{\circ}$ of $X$ and
\[
h\circ\psi=f^l\circ h
\]
for some $l\geq 1$. If these data can be chosen so that $A$ has
complex multiplication over $\mathbf{k}$, we call $f$ a
\emph{$\mathbf{k}$-exceptional map with CM}. The CM condition is automatically satisfied when $A$ is a point. 

We call $f$ an \emph{exceptional map}
(resp.\ an \emph{exceptional map with CM}) if its base change $f_{\ov{\mathbf{k}}}:X_{\ov{\mathbf{k}}}\to X_{\ov{\mathbf{k}}}$ is a $\ov{\mathbf{k}}$-exceptional map
(resp.\ a $\ov{\mathbf{k}}$-exceptional map with CM).
\end{definition}

Thus, an iterate of an exceptional map descends from a $d$-lift
of an AT-variety. In general, an AT-variety of dimensional at least $2$ does \emph{not} carry any algebraic group structure. This is different from the one dimensional case.
We establish several basic properties of exceptional maps, including
stability under semiconjugacy.

\medskip

In \cite{DS02}, Dinh and Sibony proved the following result:
For commuting endomorphisms $f,g$ of $\P^k$ with multiplicatively
independent degrees, Dinh and Sibony used Poincar\'e maps to obtain
a strong affine rigidity theorem: the maps admit a common
Poincar\'e map semiconjugating them to affine maps, and a discrete
affine group acts transitively on its fibers; see
Proposition~\ref{prop:gq-transitive} for the form used here.
In particular, they proved that $f$ and $g$ are post-critically
finite (PCF).
Our Theorem~\ref{thm_commuting1} extends the theorem of Dinh and Sibony
\cite{DS02} from endomorphisms of $\P^k$ to polarized endomorphisms
of normal projective varieties and strengthens its conclusion.

Building on their work, we prove that $f$ and $g$ are exceptional
in the sense of Definition~\ref{def:exceptional-intro}.
Using the Poincar\'e map and Dinh--Sibony's rigidity theorem, we first
construct a holomorphic principal bundle under a complex algebraic
torus over a compact complex torus, together with a holomorphic
endomorphism and a holomorphic map semiconjugating it to an iterate
of $f$. We then prove that the base is an abelian variety and that
the principal bundle, the endomorphism, and the semiconjugacy are
algebraic. In Section~\ref{sec:Dinh-Sibony}, we follow Dinh and
Sibony's method to extend their rigidity theorem from endomorphisms
of $\P^k$ to polarized endomorphisms of normal projective varieties.
In Section~\ref{sec:commuting maps}, we carry out the construction
and algebraization to prove Theorem~\ref{thm_commuting1}.

\begin{theorem}\label{thm_commuting1}
Let $\mathbf{k}$ be an algebraically closed field of characteristic zero. Let $X$ be a normal projective variety over $\mathbf{k}$, and let
$f,g:X\to X$ be polarized endomorphisms with polarization degrees
$d_f$ and $d_g$, respectively. Assume that $f\circ g=g\circ f$ and $d_f^n\neq d_g^m$
for all positive integers $n,m$. Then $f$ and $g$ are exceptional maps.
\end{theorem}

\subsubsection*{Abelian backward orbits}

Let $f:X\to X$ be a polarized endomorphism of a normal projective
variety. Dinh and Sibony \cite{DS03} proved that there is a proper
Zariski closed subset $\mc{E}_f\subset X$ satisfying
\[
f^{-1}(\mc{E}_f)=f(\mc{E}_f)=\mc{E}_f
\]
such that, for every $x\in X\setminus\mc{E}_f$,
\[
\Zar{\bigcup_{n\geq 0}f^{-n}(x)}=X.
\]
We call $\mc{E}_f$ the \emph{exceptional set} of $f$, and a point
$x\in X$ is called \emph{exceptional} if $x\in\mc{E}_f$. Thus,
a point is non-exceptional precisely when its backward orbit is
Zariski dense in $X$.

A group is \emph{virtually abelian} if it contains an abelian
subgroup of finite index. We call a Galois extension $L/K$
\emph{virtually abelian} if $\Gal(L/K)$ is virtually abelian.
Our main result gives the following characterization.

\begin{theorem}[Theorem \ref{thm_vaeq}]\label{thm_vaeq1}
Let $K$ be a number field, let $X$ be a normal projective variety
over $K$, and let $f:X\to X$ be a polarized endomorphism.
Assume that $\alpha\in X(K)$ is non-exceptional. Then
$K_{\infty}(f,\alpha)/K$ is virtually abelian if and only if
$f$ is an exceptional map with CM and $\alpha\in\mr{Prep}(f)$.
\end{theorem}

Here $\mr{Prep}(f)$ denotes the set of preperiodic points of $f$.
In particular, Conjecture~\ref{conj_P1} holds. An intermediate step in the proof of Theorem~\ref{thm_vaeq1} is the following result for $K$-exceptional maps, which
generalizes \cite[Theorem~D]{FOZ24}.

\begin{theorem}[Theorem \ref{thm_Kexp}]\label{thm_Kexp1}
Let $K$ be a number field, let $X$ be a normal projective variety
over $K$, and let $f:X\to X$ be a $K$-exceptional map.
Suppose that $\alpha\in X(K)$ is preperiodic and non-exceptional
for $f$. Then $K_{\infty}(f,\alpha)/K$ is virtually abelian
if and only if $f$ is an exceptional map with CM.
\end{theorem}

Our theorems also generalize the following characterization of abelian varieties with complex multiplication, which was previously deduced from Faltings' theorem \cite[Theorem 4]{Fal83}.
\begin{theorem}\label{thm_CM}
Let $A$ be an abelian variety over a number field $K$, and let $\ell$ be a prime. Consider the associated $\ell$-adic representation
\[
\rho_{\ell}\colon \Gal(\ov{K}/K)\longrightarrow \mr{GL}(V_\ell(A)),
\]
where $V_{\ell}(A):=T_\ell(A)\otimes_{\mb{Z}_{\ell}}\mb{Q}_{\ell}$.  Then $\rho_{\ell}\big(\Gal(\ov{K}/K)\big)$ is virtually abelian if and only if $A$ has complex multiplication over $\ov{K}$.
\end{theorem}
Indeed, Theorem \ref{thm_CM} follows by applying Theorem \ref{thm_Kexp1} with $X=A$, $f=[\ell]$, and $\alpha=0_A$.

\medskip

\subsubsection*{Rigidity of abelian sequences}

The same methods yield a rigidity theorem for generic sequences
of small points over abelian extensions. To formulate the
arithmetic hypotheses, we introduce the following terminology.

\begin{definition}
Let $(x_m)_{m\geq 1}$ be an infinite sequence in $X(\ov{K})$.
We say that $(x_m)_{m\geq 1}$ is \emph{abelian} if
$\Gal\bigl(K(x_m)/K\bigr)$ is abelian for every $m\geq 1$.
For a finite place $v$ of $K$, we say that $(x_m)_{m\geq 1}$ is
\emph{almost unramified at $v$} if the set of ramification indices
\[
\left\{
e(w/v)\;\middle|\;
m\geq 1,\;
w \text{ is a place of } K(x_m) \text{ above } v
\right\}
\]
is bounded.
\end{definition}

The ramification bound is uniform along the sequence at each fixed
place, but may depend on the place. Under this condition outside
a finite set of places, the existence of a generic sequence of
small abelian points has the following geometric consequence.

\begin{theorem}[Theorem \ref{thm_main}]\label{thm_main1}
Let $K$ be a number field, and let $X$ be a normal projective
variety over $K$. Let $\overline{L}$ be a nef adelic line bundle
on $X$ with $L$ ample. Assume that there exist a finite set $S$
of places of $K$ and a generic sequence $(x_m)_{m\geq 1}$ in
$X(\overline{K})$ such that $(x_m)_{m\geq 1}$ is abelian,
almost unramified at every $v\notin S$, and satisfies
\[
\lim_{m\to\infty}h_{\overline{L}}(x_m)=0.
\]
Then there exist a normal projective variety $Y$ over $K$, a finite
surjective morphism $\pi:X\to Y$, a polarized endomorphism
$f:Y\to Y$, and an $f$-admissible adelic line bundle
$\overline{M}_f\in\widehat{\mathrm{Pic}}(Y)_{\mathbb{Q}}$ such that
\[
\overline{L}\leq\pi^*\overline{M}_f.
\]
Moreover, $f$ is an exceptional map.
\end{theorem}

Here $\overline{L}\leq\pi^*\overline{M}_f$ means that
$\pi^*\overline{M}_f-\overline{L}$ is a nef adelic line bundle.
A sequence $(x_m)_{m\geq 1}$ is called \emph{generic} if every
infinite subsequence is Zariski dense in $X$.

\subsection{Sketch of the proof}

Our approach to Theorems~\ref{thm_vaeq1} and Theorem \ref{thm_main1}
has three main ideas.

First, let $(x_m)_{m\geq 1}$ be a generic small sequence of
abelian points that is almost unramified at a suitable place $v$. We choose a suitable Frobenius lift $\sigma\in\Gal(\ov{K}/K)$ and use the Zariski closure of the pairs $(x_m,\sigma(x_m))$ to construct a correspondence $\mf{C}_X^{(v)}\subset X\times X$. Yuan's non-archimedean equidistribution theorem shows that this correspondence is bi-finite and its reduction is the graph of Frobenius.

Second, by iterating $\mf{C}_X^{(v)}$ and its transpose,
we construct a bi-finite equivalence relation
$\mf{R}_X^{(v)}$.
On the geometric quotient $Y=X/\mf{R}_X^{(v)}$, the
correspondence $\mf{C}_X^{(v)}$ descends to an endomorphism $g:Y\to Y$.
The proof of Theorem~\ref{thm_main1} is based on these two constructions.

Third, we develop a theory of exceptional endomorphisms in
higher dimension and establish their basic properties,
including the commuting criterion in
Theorem~\ref{thm_commuting1} and stability under semiconjugacy.
These results allow us to deduce the exceptionality of the
original map from the endomorphisms constructed on the quotient.

\medskip

We now explain how these ideas are used in the proof of
Theorem~\ref{thm_vaeq1}.
\medskip

\noindent\textbf{Step 1.} We first construct a suitable small sequence. Let $f:X\to X$ be a $d$-polarized endomorphism, and let $\alpha\in X(K)$ be non-exceptional.
Assume that $K_\infty(f,\alpha)/K$ is virtually abelian.
After a finite extension of $K$, we may assume that this
extension is abelian. Replacing $\alpha$ by a point in its backward orbit if
necessary, we show that there is a suitable finite place
$v$ of $K$ for which
\[
\mc{P}^{(v)}
:=
\left\{
\beta\in f^{-\infty}(\alpha)
\;\middle|\;
K(\beta)/K \text{ is unramified above } v
\right\}
\]
is Zariski dense in $X$. Using this set, we construct a generic sequence of small
abelian points that is almost unramified at $v$.

\medskip
\noindent\textbf{Step 2. }
Applying the first two constructions to this sequence, we
obtain a normal projective variety $Y$, a finite surjective
morphism $\pi:X\to Y$, and a polarized endomorphism
$g:Y\to Y$ such that $\gcd(\deg(g),d)=1$.
The map $f$ also descends to an endomorphism $f_Y:Y\to Y$
such that
\[
\pi\circ f=f_Y\circ\pi
\qquad\text{and}\qquad
f_Y\circ g=g\circ f_Y.
\]
By Theorem~\ref{thm_commuting1}, $f_Y$ is exceptional.
Since exceptionality is stable under semiconjugacy,
$f$ is also exceptional. Moreover, Theorem~\ref{thm_main1} shows that
$\alpha\in\mr{Prep}(f)$.

\medskip
\noindent\textbf{Step 3. }
Finally, we treat the exceptional case and prove
Theorem~\ref{thm_Kexp1}.
We first prove the corresponding statement for
$d$-lifts of AT-varieties. For preperiodic base points, the implication from CM to
virtual abelianity follows from Frobenius lifts and the
Chebotarev density theorem. We use Zarhin's result \cite[Theorem~1]{Zar87} to prove
that CM is necessary. This completes the proof of Theorem~\ref{thm_vaeq1}.

\begin{remark}
The first idea, constructing a correspondence in $X\times X$
using Yuan's non-archimedean equidistribution theorem, is
also used by the third author and Ziquan Yang in their
recent work \cite{XY26}. In that work, they give new proofs of Faltings' isogeny
theorem and Mordell's conjecture.
\end{remark}

\subsection*{Notations and Terminology}
By a \emph{variety}, we mean a separated, integral scheme of finite type over a base field. A \emph{curve} is defined as a variety of dimension one.

By a \emph{line bundle} on a scheme, we mean an invertible sheaf. For any line bundles $L, M$ and integers $a, b$, the notation $aL - bM$ denotes the tensor product $L^{\otimes a} \otimes M^{\otimes (-b)}$. We use both additive and multiplicative notation, as appropriate
to the context.

For line bundles $L\in \Pic(X)$, we write $L\equiv 0$ if $L$ is numerically trivial. Denote $N^1(X):=\mr{Pic}(X)/\equiv$ and $N^1(X)_{\mb{R}}:=N^1(X)\otimes_{\mb{Z}}\mb{R}$. 

For the theory of \emph{adelic line bundles}, we refer to \cite{YZ26}. Denote by $\widehat{\Pic}(X)$ the group of adelic line bundles on $X$, and set
\[
\widehat{\Pic}(X)_{\mathbb{Q}} := \widehat{\Pic}(X)\otimes_{\mathbb{Z}} \mathbb{Q}.
\]
For adelic line bundles $\overline{L}, \overline{M} \in \widehat{\Pic}(X)_{\mathbb{Q}}$, we write $\overline{L} \geq \overline{M}$ if $\overline{L} - \overline{M}$ is a nef adelic line bundle.

For a field $k$, we denote by
\[
\mathrm{PGL}_2(k) := \left\{\frac{ax+b}{cx+d} \in \mathrm{Aut}(\mathbb{P}^1_k) \;\middle|\; a,b,c,d \in k,\ ad-bc \neq 0 \right\}
\]
the group of M\"obius transformations.

For an integrable adelic line bundle $\overline{L}$ on a projective variety $X$ over a number field $K$, we define the \emph{essential minimum} by
\[
\mathrm{ess}(X,\overline{L}) := \sup_{U \subset X} \inf_{x \in U(\overline{K})} h_{\overline{L}}(x),
\]
where the supremum is taken over all Zariski open subschemes $U$ of $X$.

\subsection*{Organization of the paper}
Section~\ref{sec_pre} reviews Berkovich spaces, adelic line bundles, and bi-finite correspondences. Section~\ref{sec_AT} recalls the theory of AT-varieties and establishes the basic properties of exceptional maps. Sections~\ref{sec:Dinh-Sibony} and~\ref{sec:commuting maps} are devoted to the proof of Theorem~\ref{thm_commuting1}. We prove Theorem~\ref{thm_main1} in Section~\ref{sec_main} and complete the proof of Theorem~\ref{thm_vaeq1} in Section~\ref{sec_backward_orbit}. At the end of Section~\ref{sec_backward_orbit}, we will explain how to prove Conjecture \ref{conj_P1} using our method.

\subsection*{Acknowledgements}
{\emergencystretch=2em
The main results of this paper were obtained in December 2025 and subsequently presented by Junyi Xie on several occasions, including the \emph{Jingrun Star Awards and Young Researchers’ Forum on Number Theory} at Xiamen University in May 2026. Following an extended period of writing and revision, the manuscript was completed in September 2026. We thank the organizers of these events and Xiamen University for their hospitality.

Zhuchao Ji is supported by the National Key R\&D Program of China (No.\ 2025YFA1018300), NSFC Grant (No.12401106), and ZPNSF grant (No.XHD24A0201). Junyi Xie is supported by NSFC Grant (No.12271007) and by the Xplorer Prize from the New Cornerstone Science Foundation.\par}

\medskip

\noindent\textbf{AI usage.} The authors used ChatGPT 5.6 Sol to assist with the study of AT-varieties and to polish the language. Specifically, AI suggested the ideas for Propositions \ref{prop_etaleAT} and Proposition \ref{prop_semiconj1} on the basic properties of AT-varieties, and Lemma \ref{lem_Chebo}, Lemma \ref{lem_CMfrob}, Lemma \ref{lem_Frob_linebundle} on AT-varieties with CM. All mathematical arguments and the final text were written by the authors, who are responsible for their correctness.

\section{Preliminary}\label{sec_pre}

\subsection{Berkovich spaces and the Chambert-Loir measure}

\subsubsection*{Berkovich analytification}

Let $\mathbf{k}$ be a complete non-archimedean field with absolute value $|\cdot|_{\mathbf{k}}$, and let $X$ be a variety over $k$. The \emph{Berkovich analytification} $X^{\an}$ is defined as follows.

Assume first that $X=\Spe A$ is affine of finite type over $\mathbf{k}$, where $A$ is a finitely generated $\mathbf{k}$-algebra. Then $X^{\an}$ is the set of all multiplicative seminorms $\rho:A\to \mathbb{R}_{\ge 0}$ such that
\begin{enumerate}
\renewcommand{\labelenumi}{(\theenumi)}
    \item $\rho|_{\mathbf{k}}=|\cdot|_{\mathbf{k}}$,
    \item $\rho(a+b)\leq \rho(a)+\rho(b)$ for all $a,b\in A$,
    \item $\rho(ab)=\rho(a)\rho(b)$ for all $a,b\in A$.
\end{enumerate}
For each $f\in A$, define a function
\[
|f|:X^{\an}\to \mathbb{R}_{\ge 0},\qquad \rho\mapsto \rho(f).
\]
We endow $X^{\an}$ with the weakest topology for which all functions $|f|$ are continuous.

For a general variety $X$, choose an affine open cover and define $X^{\an}$ by gluing the analytifications of the affine pieces.

\subsubsection*{Shilov boundary}

Let $X$ be a projective variety over $\mathbf{k}$, and let $\mc{X}$ be a normal integral model of $X$ over $\Spe {\mathbf{k}}^\circ$, i.e. $\mc{X}$ is an integral scheme, projective and flat over $\Spe {\mathbf{k}}^\circ$, with generic fiber $\mc{X}_{\mathbf{k}}\simeq X$. Let $\mathbf{k}^{\circ\circ}$ denote the maximal ideal of $\mathbf{k}^\circ$, and let $\wt{\mathbf{k}}=\mathbf{k}^\circ/\mathbf{k}^{\circ\circ}$ be the residue field. Denote by
\[
\wt{X}:=\mc{X}\times_{\Spe \mathbf{k}^\circ}\Spe \wt{\mathbf{k}}
\]
the special fiber.

There is a natural reduction map
\[
\red_{\mc{X}}:X^{\an}\to \wt{X}.
\]

Let $\xi\in X^{\an}$, and let $x:=\ker(\xi)\in X$. By the valuative criterion of properness, the morphism $\Spe \mathbf{k}(x)\to X$ extends uniquely to a morphism $\Spe \mathbf{k}(x)^\circ\to \mc{X}$, where $\mathbf{k}(x)^\circ$ is the valuation ring of $\mathbf{k}(x)$.
We have a diagram,
\[
\begin{tikzcd}
\Spe \mathbf{k}(x) \arrow[d] \arrow[r]      & \mc{X} \arrow[d] \\
\Spe \mathbf{k}(x)^{\circ} \arrow[r] \arrow[ru] & \Spe \mathbf{k}^{\circ}                                                  
\end{tikzcd}
\]
The reduction $\red_{\mc{X}}(\xi)$ is defined as the image of the closed point of $\Spe \mathbf{k}(x)^\circ$ in $\wt{X}$.

Let $\eta$ be an irreducible component of $\wt{X}$. Since $\eta$ defines a prime divisor on $\mc{X}$, it induces a discrete valuation
\[
v_\eta:\mathbf{k}(X)^\times=\mathbf{k}(\mc{X})^{\times}\to \mathbb{Z}.
\]
Define a seminorm
\[
\rho_\eta(f):=\exp(-a\, v_\eta(f)), \qquad f\in \mathbf{k}(X)^\times,
\]
where $a>0$ is the unique constant such that $\rho_\eta|_{\mathbf{k}}=|\cdot|_{\mathbf{k}}$. This uniquely determine a point $\rho_\eta\in X^{\an}$, called the \emph{Shilov point} associated to $\eta$. Moreover, $\rho_\eta$ is the unique point in $X^{\an}$ mapping to the generic point of $\eta$ under $\red_{\mc{X}}$.

\subsubsection*{The Chambert--Loir measure}
Let $X$ be a projective variety over $\mathbf{k}$ of dimension $n$, and let $(L,\|\cdot\|)$ be a line bundle endowed with a semipositive metric. The associated \emph{Chambert--Loir measure} on $X^{\an}$ is defined as follows.

First assume that $\|\cdot\|$ is induced by a normal integral model $(\mc{X},\mathcal{M})$ of $(X,L^{\otimes e})$ over $\Spe \mathbf{k}^\circ$. Let $|\wt{X}|_g$ denote the set of irreducible components of $\wt{X}$. Then one defines
\[
c_1(L,\|\cdot\|)^n
:=\frac{1}{e^n}\sum_{\eta\in |\wt{X}|_g}
\deg_{\mathcal{M}|_\eta}(\eta)\,\delta_{\rho_\eta},
\]
where $\delta_{\rho_\eta}$ is the Dirac measure at $\rho_\eta$.

In general, any semipositive metric can be approximated uniformly by semipositive model metrics $\|\cdot\|_m$. One defines
\[
c_1(L,\|\cdot\|)^n
:=\lim_{m\to\infty} c_1(L,\|\cdot\|_m)^n.
\]
Here, the sequence converges weakly due to \cite{CL06,Gub10}.

\subsection{Adelic line bundles}

Let $K$ be a number field and $X$ be a projective variety over $K$ of dimension $n$. We refer \cite{YZ26} for the theory of adelic line bundles on $X$.

\subsubsection*{Admissible adelic line bundles}
Let $f: X\to X$ be a polarized endomorphism of $X$. Namely, there exists an ample line bundle $L\in \Pic(X)$ and an integer $q>1$ such that $f^*L\simeq qL$.  Yuan and Zhang show the following.
\begin{theorem}[{\cite[Theorem 6.4.2]{YZ26}}]\label{thm_admiss}
Let $(X,f, L)$ be a polarized algebraic dynamical system over a number field $K$. Assume that X is normal. The projection
$\widehat{\Pic}(X)_{\mb{Q}}\to \Pic(X)_{\mb{Q}}$ has a unique section
\[M\mapsto \ov{M}_f\]
as $f^*$-modules. The image ${\ov{M}_f}$ is always integrable. If $M\in\Pic(X)_{\mb{Q}}$ is ample, then $\ov{M}_f$ is nef.
\end{theorem}
We call $\ov{M}_f$ the $f$-admissible extension of $M$ in $\widehat{\Pic}(X)_{\mb{Q}}$. An adelic line
bundle in $\widehat{\Pic}(X)_{\mb{Q}}$ which is isomorphic to some $\ov{M}_f$ is called \emph{$f$-admissible}. 

\begin{corollary}[{\cite[Corollary 6.4.3]{YZ26}}]\label{cor_admiss}
For $M\in \Pic(X)_{\mb{Q}}$, the following are true:
\begin{enumerate}
\renewcommand{\labelenumi}{(\theenumi)}
    \item If $f^*M=\lambda M$ for some $\lambda \in \mb{Q}$, then $f^*\ov{M}_f=\lambda \ov{M}_f$ in $\widehat{\Pic}(X)_{\mb{Q}}$.
    \item For any $x\in \mr{Prep}(f)$, one has $\ov{M}_f|_{x'}=0$ in $\widehat{\Pic}(x')_{\mb{Q}}$. Here $x'$ is the closed point of $X$ corresponding to $x$.
\end{enumerate}
\end{corollary}
By Corollary \ref{cor_admiss}, $\ov{L}_f$ coincides with the invariant adelic line bundle constructed in \cite[Section 6.1]{YZ26}.

\subsubsection*{Model places}
Recall, an adelic line bundle $\overline{L}\in \widehat{\Pic}(X)$ can be represented by a model $(\mc{X}_{\mc{V}},\mc{L}_{\mc{V}})$ over an open subscheme of $\mc{V}\subset \Spe O_K$. A finite place $v$ is said to be a \emph{model place} with respect to $\overline{L}$ if 
there is a model $(\mc{X},\mathcal{L})$ of $(X,L)$ over $\Spe O_K$, and a neighborhood $\mc{V}$ of $v$, such that $\ov{L}$ is induced by a single model $(\mc{X}_{\mc{V}},\mathcal{L}_{\mc{V}})$ over $\mc{V}$, and $\mathcal{L}|_{\mc{X}_v}$ is ample. There are some basic facts about model places.

\begin{itemize}
\item If $v$ is a model place with respect to both $\ov{L}$ and $\ov{M}\in \widehat{\Pic}(X)$ , then $v$ is a model place with respect to $\ov{L}+\ov{M}$.
\item If $Y\subset X$ is a subvariety and $v$ is a model place with respect to $\ov{L}\in \widehat{\Pic}(X)$, then $v$ is a model place with respect to $\ov{L}|_Y$.
\item If $L$ is ample, then all but finitely many finite places of $K$ are model places with respect to $\overline{L}$.
\end{itemize}

At model places, the Chambert-Loir measure is a finite sum of Dirac measures.
\begin{lemma}\label{lem_measure}
Let $\overline{L}$ be a nef adelic line bundle on $X$ with $L$ ample. If $v$ is a model place, then the Chambert-Loir measure satisfies
\[
c_1(\overline{L})_v^n
=\sum_{\eta}
a_{v,\eta}\,\delta_{\rho_\eta},
\]
where $\eta$ runs over the irreducible components of the special fiber $\mc{X}_v$, $\rho_\eta$ is the associated Shilov point, and $a_{v,\eta}>0$.
\end{lemma}

\begin{proof}
By assumption, over a neighburhood $\mc{V}\subset \Spe O_K$ of $v$, the adelic line bundle $\overline{L}$ is induced by a single model $(\mc{X}_{\mc{V}},\mathcal{L}_{\mc{V}})$. The description of the Chambert-Loir measure for model metrics yields
\[
c_1(\overline{L})_v^n
=\sum_{\eta}
\deg_{\mathcal{L}|_\eta}(\eta)\,\delta_{\rho_\eta}.
\]
Since $\mathcal{L}|_{\mc{X}_v}$ is ample, all coefficients $a_{v,\eta}=\deg_{\mathcal{L}|_\eta}(\eta)$ are positive.
\end{proof}

\subsubsection*{Lemmas about adelic line bundles}
Here are some technical lemmas about adelic line bundles, which will be used in Section \ref{sec_main}. The following lemma is a variant of \cite[Theorem 5.3.7]{YZ26}
\begin{lemma}\label{lem_htcomp1}
Let $\overline{L}_1,\overline{L}_2$ be nef adelic line bundles on $X$ with $L_1$ and $L_2$ ample. Assume that $h_{\overline{L}_1+\overline{L}_2}(X)=0$. Then there exists a constant $C>0$, depending only on $L_1$ and $L_2$, such that for any $\varepsilon>0$, there exists a nonempty open subscheme $U_{\varepsilon}\subset X$ satisfying
\[
h_{\overline{L}_1}(x)\le C\,h_{\overline{L}_2}(x)+\varepsilon
\quad \forall x\in U_{\varepsilon}(\overline{K}).
\]
\end{lemma}
\begin{proof}
Choose a constant $C>0$ such that
\[
C\,L_2^n > n\,(L_2^{n-1}\cdot L_1)+L_2^n.
\]
Since $h_{\overline{L}_1+\overline{L}_2}(X)=0$, we have $\overline{L}_1^j\cdot \overline{L}_2^{n+1-j}=0$ for all $0\le j\le n+1$.

Let $\overline{N}\in \widehat{\Pic}(O_K)$ be a hermitian line bundle with $\widehat{\deg}(\overline{N})=1$, and set
\[
\overline{H}_i:=\overline{L}_i+\varepsilon\,\pi^*\overline{N},
\quad i=1,2,
\]
where $\pi:X\to \Spe K$ is the structure morphism. Then $\overline{H}_i$ is nef. A direct computation shows that
\[
\overline{H}_2^{n+1}
=(n+1)\varepsilon\, L_2^n>0,
\]
and
\[
\overline{H}_2^n\cdot \overline{H}_1
=\varepsilon\big(n\,L_2^{n-1}\cdot L_1+L_2^n\big)>0.
\]

By Yuan's inequality \cite[Theorem 5.2.2]{YZ26}, we obtain
\[
\widehat{\mathrm{vol}}(C\overline{H}_2-\overline{H}_1)
\ge C^{n+1}\overline{H}_2^{n+1}
-(n+1)C^n\,\overline{H}_2^n\cdot \overline{H}_1>0,
\]
Hence $C\overline{H}_2-\overline{H}_1$ is arithmetically big. It follows that there exists a positive integer $r$, and an effective section $s$ of $r(C\ov{H}_2-\ov{H}_1)$ such that 
\[h_{C\ov{H}_2-\ov{H}_1}(x)\geq 0,\quad \forall x\in X(\ov{K}), s(x)\neq 0.\]

Let $U_{(C-1)\varepsilon}:=X\setminus \Supp(s)$. Then for all $x\in U_{(C-1)\varepsilon}(\overline{K})$, we have
\[
0\le h_{C\overline{H}_2-\overline{H}_1}(x)
= C\,h_{\overline{H}_2}(x)-h_{\overline{H}_1}(x)=C\,h_{\overline{L}_2}(x)-h_{\overline{L}_1}(x)+(C-1)\varepsilon.
\]
Hence
\[
h_{\overline{L}_1}(x)\le C\,h_{\overline{L}_2}(x)+(C-1)\varepsilon
\]
for all $x\in U_{(C-1)\varepsilon}(\overline{K})$, as desired.
\end{proof}

Recall, for a finite surjective morphism $f:X\to Y $ between normal projective varieties over $K$, the Deligne pairing $\langle\,\cdot\,\rangle_{X/Y}: \widehat{\Pic}(X)\to \widehat{\Pic}(Y)$ is well-defined. See \cite[Section 4]{YZ26}, \cite{Li24} for details. We have the following lemma, which will be used in the proof of Theorem \ref{thm_main1}.

\begin{lemma}\label{lem_dpnef}
Let $X, Y$ be normal projective varieties over a number field $K$ and $\pi: X\to Y$ is a finite surjective morphism. Let $\ov{L}\in \widehat{\Pic}(X)$ be a nef adelic line bundle on $X$, then $\pi^*\langle \ov{L}\rangle_{X/Y}-\ov{L}$ is nef.
\end{lemma}
\begin{proof}
Let $Z$ be the Galois closure of $\pi: X\to Y$. In other words, $Z$ is a normal projective variety over $K$, and there is a finite surjective morphism
$\varphi: Z\to X$ such that $\psi=\pi\circ \varphi: Z\to Y$ is a Galois cover. Let $G:=\mr{Aut}_Y(Z)$ and $H:=\mr{Aut}_X(Z)\subset G$, then $Y\simeq Z/G$ and $X\simeq Z/H$.

Since $Y\simeq Z/G$, for any adelic line bundle $\ov{M}\in \widehat{\Pic}(Z)$, we have 
\[\psi^*\langle\ov{M}\rangle_{Z/Y}=\sum_{\sigma\in G}\sigma^*\ov{M}.\]
Note that $\langle \varphi^*\ov{L}\rangle_{Z/X}=\deg(\varphi)\, \ov{L}$. Choose $\ov{M}:= \varphi^*\ov{L}$, we obtain
\[\deg(\varphi)\,\varphi^*\pi^*\langle\ov{L}\rangle_{X/Y}=\deg(\varphi)\,\psi^*\langle \ov{L}\rangle_{X/Y}=\psi^*\langle \varphi^*\ov{L}\rangle_{Z/Y}=\sum_{\sigma\in G}\sigma^*\varphi^*\ov{L}.\]
For $\sigma\in H$, $\varphi\circ\sigma=\varphi$. So 
\[\varphi^*\pi^*\langle\ov{L}\rangle_{X/Y}-\varphi^*\ov{L}=\frac{1}{\deg(\varphi)}\sum_{\sigma\in G}\sigma^*\varphi^*\ov{L}-\ov{L}=\frac{1}{\deg(\varphi)}\sum_{\sigma\in G\setminus H}\sigma^*\varphi^*\ov{L}.\]
The right-hand side is nef. Hence
\[\pi^*\langle\ov{L}\rangle_{X/Y}-\ov{L}=\frac{1}{\deg(\varphi)}\langle \varphi^*(\pi^*\langle\ov{L}\rangle_{X/Y}-\ov{L})\rangle_{Z/X}\]
is nef by \cite[Theorem 4.1.3]{YZ26}.
\end{proof}

\begin{lemma}\label{lem_dphtcomp}
Let $X, Y$ be normal projective varieties over a number field $K$ and $\pi: X\to Y$ is a finite surjective morphism. Let $\ov{L}\in \widehat{\Pic}(X)$ be a nef adelic line bundle on $X$, then
\[h_{\langle \ov{L}\rangle_{X/Y}}(y)\leq \deg(\pi)\max_{x\in \pi^{-1}(y)}h_{\ov{L}}(x).\]
\end{lemma}
\begin{proof}
If $\pi:X\to Y$ is a Galois cover, suppose $G:=\mr{Aut}_Y(X)$ and $Y\simeq X/G$. Then
\[\pi^*\langle\ov{L}\rangle_{X/Y}=\sum_{\sigma\in G}\sigma^*\ov{L}.\]
For $x\in X(\ov{K})$ with $\pi(x)=y\in Y(\ov{K})$
\[h_{\langle\ov{L}\rangle_{X/Y}}(y)=h_{\pi^*\langle\ov{L}\rangle_{X/Y}}(x)=\sum_{\sigma\in G}h_{\sigma^*\ov{L}}(x)=\sum_{\sigma\in G}h_{\ov{L}}(\sigma(x))\leq \deg(\pi)\max_{x'\in\pi^{-1}(y)}h_{\ov{L}}(x').\]

For the general case, let $Z$ be the Galois closure of $\pi: X\to Y$. Then $\varphi: Z\to X$ and $\psi=\pi\circ \varphi:Z\to Y$ are Galois covers. Let $G:=\mr{Aut}_Y(Z)$ and $H:=\mr{Aut}_X(Z)\subset G$, then $Y\simeq Z/G$ and $X\simeq Z/H$. Note that
 $\langle \varphi^*\ov{L}\rangle_{Z/X}= \deg(\varphi)\, \ov{L}$. For $y\in Y(\ov{K})$,
\[
\begin{aligned}
h_{\langle \ov{L}\rangle_{X/Y}}(y)
&=\frac{1}{\deg(\varphi)}h_{\langle \varphi^*\ov{L}\rangle_{Z/Y}}(y)\\
&\leq \frac{\deg(\pi\circ \varphi)}{\deg(\varphi)}
\max_{z\in (\pi\circ \varphi)^{-1}(y)}h_{\varphi^*\ov{L}}(z)\\
&=\deg(\pi)\max_{x\in \pi^{-1}(y)}h_{\ov{L}}(x).
\end{aligned}
\]
\end{proof}

\subsection{Bi-finite correspondences}\label{sec_bifin}

Throughout this subsection, let $X$ be a normal projective variety
of dimension $n$ over a field $\mathbf{k}$ of characteristic zero.

For an integer $r$, an $r$-cycle on $X$ is a finite formal sum $\alpha = \sum n_i [\Gamma_i]$, where the $\Gamma_i$ are distinct irreducible subvarieties of $X$ of dimension $r$, and $n_i \in \mathbb{Z}$. We say that $\alpha$ is \emph{effective} if $n_i \ge 0$ for all $i$.

A \emph{correspondence} $c : X \vdash X$ is an effective $n$-cycle $c = \sum n_i [\Gamma_i]$
on $X \times X$. We denote by $\mathcal{C}(X,X)$ the group of correspondences from $X$ to itself, and by
\[
\mathrm{Supp}(c) := \bigcup \Gamma_i \subset X \times X
\]
the support of $c$. We say that $c$ is \emph{bi-finite} if, for every irreducible component $\Gamma_i$ of $c$, the projections
\[
\pi_{1,\Gamma_i}, \pi_{2,\Gamma_i} : \Gamma_i \to X
\]
are finite morphisms.

Bi-finite correspondences satisfy the following basic properties:
\begin{enumerate}
\renewcommand{\labelenumi}{(\theenumi)}
\item If $c \in \mathcal{C}(X,X)$ is bi-finite, then its transpose $\trans c$ is also bi-finite.
\item If $c_1, c_2 \in \mathcal{C}(X,X)$ are bi-finite, then their composition $c_1 \circ c_2$ is bi-finite.
\item If $c_1, c_2 \in \mathcal{C}(X,X)$ are bi-finite and $r_1, r_2 \in \mathbb{Z}_{\ge 0}$, then $r_1 c_1 + r_2 c_2$ is bi-finite.
\item Let $f : X \to X$ be a surjective endomorphism, and identify $f$ with its graph $\Gamma_f \in \mathcal{C}(X,X)$. Then $f$ is bi-finite. Moreover,
\[
f \circ \trans f = \deg(f)\, \Delta,
\]
where $\Delta \subset X \times X$ denotes the diagonal.
\end{enumerate}

\begin{definition}
A bi-finite correspondence $c$ is \emph{numerically $d$-polarized},
where $d>1$ is an integer, if there is an ample class
$D\in N^1(X)_{\mathbb R}$ such that
\[
\pi_{2,\Gamma}^*D\equiv d\,\pi_{1,\Gamma}^*D
\]
for every irreducible component $\Gamma$ of $c$.
\end{definition}

\begin{remark}
For a surjective endomorphism $f:X\to X$, numerical
$d$-polarization is equivalent to the existence of an ample line
bundle $L$ with $f^*L\simeq L^{\otimes d}$, by
\cite[Proposition~1.1]{MZ18}.
\end{remark}

\begin{lemma}\label{lem_polarizediter}
If $c$ is numerically $d$-polarized, then $c^m$ is numerically
$d^m$-polarized for every $m\geq1$.
\end{lemma}

\begin{proof}
Let $D\in N^1(X)_{\mathbb R}$ be an ample class polarizing $c$.
We prove by induction on $m$ that the same class $D$ numerically
$d^m$-polarizes $c^m$.

The case $m=1$ is immediate. Suppose that the assertion holds
for $m-1$, and let $\Gamma$ be an irreducible component of $c^m$.
By the definition of composition, there exist irreducible
components $\Gamma_1$ of $c$ and $\Gamma_2$ of $c^{m-1}$,
and an irreducible component $Z\subseteq\Gamma_2\times_X\Gamma_1$, whose image under the endpoint map is $\Gamma$. Here the fiber product is taken with respect to
$\pi_{2,\Gamma_2}$ and $\pi_{1,\Gamma_1}$.

Write $\psi:Z\to\Gamma_2,\varphi:Z\to\Gamma_1$ and $r:Z\to\Gamma$ for the induced morphisms. Then $r$ is finite and surjective, and we have the commutative diagram

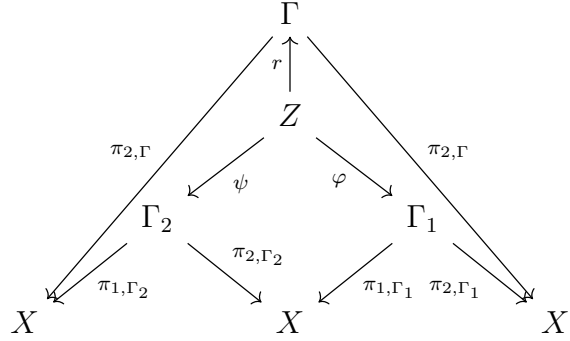
\[
\begin{tikzcd}
  &                                                                            & \Gamma \arrow[llddd, "{\pi_{2,\Gamma}}"'] \arrow[rrddd, "{\pi_{2,\Gamma}}"] &                                                                             &   \\
  &                                                                            & Z \arrow[u, "r"] \arrow[ld, "\psi"] \arrow[rd, "\varphi"']                  &                                                                             &   \\
  & \Gamma_2 \arrow[rd, "{\pi_{2,\Gamma_2}}"] \arrow[ld, "{\pi_{1,\Gamma_2}}"] &                                                                             & \Gamma_1 \arrow[rd, "{\pi_{2,\Gamma_1}}"'] \arrow[ld, "{\pi_{1,\Gamma_1}}"] &   \\
X &                                                                            & X                                                                           &                                                                             & X
\end{tikzcd}\]

Using the polarization condition on $\Gamma_1$ and the induction
hypothesis on $\Gamma_2$, we obtain
\[
r^*\pi_{2,\Gamma}^*D=\varphi^*\pi_{2,\Gamma_1}^*D\equiv d\,\varphi^*\pi_{1,\Gamma_1}^*D=d\,\psi^*\pi_{2,\Gamma_2}^*D\equiv d^m\psi^*\pi_{1,\Gamma_2}^*D=d^m r^*\pi_{1,\Gamma}^*D.
\]
Since the map $r^*:N^1(\Gamma)_{\mb{R}}\to N^1(Z)_{\mb{R}}$ is injective, it follows that $\pi_{2,\Gamma}^*D\equiv d^m\pi_{1,\Gamma}^*D$. As $\Gamma$ was arbitrary, this completes the induction.
\end{proof}

\begin{lemma}\label{lem_polardes}
Let $\pi:X\to Y$ be a finite surjective morphism of normal projective
varieties, let $c$ be a bi-finite correspondence on $X$, and let
$f:Y\to Y$ be a surjective endomorphism. Suppose that
\[
(\pi\times\pi)\bigl(\operatorname{Supp}(c)\bigr)=\Gamma_f.
\]
If $c$ is numerically $d$-polarized, then $f$ is $d$-polarized.
\end{lemma}

\begin{proof}
Choose an ample class $L$ polarizing $c$. Since $\pi^*:N^1(Y)_{\mb{R}}\to N^1(X)_{\mb{R}}$ is
injective, the formula
\[
\|D\|_L
:=
\inf\bigl\{
t\geq0:tL+\pi^*D,\ tL-\pi^*D\in\operatorname{Nef}(X)
\bigr\}
\]
defines a norm on $N^1(Y)_{\mathbb R}$.

Let $\pi_{1,\Gamma},\pi_{2,\Gamma}:\Gamma\to X$ be the projections of a component $\Gamma$ of $c$. Since $(\pi \times \pi)(\Supp(c))=\Gamma_f$, we have $(\pi \times \pi)(\Gamma)=\Gamma_{f}$, and hence
\[
f \circ \pi \circ \pi_{1,\Gamma} = \pi \circ \pi_{2,\Gamma}.
\]
Therefore,
\[
\begin{aligned}
(\pi_{2,\Gamma})^*(tL \pm  \pi^*D)
&\equiv dt \pi_{1,\Gamma}^*L \pm  \pi_{1,\Gamma}^*\pi^*f^*D \\
&=  \pi_{1,\Gamma}^*\big(dtL \pm  \pi^*f^*D\big).
\end{aligned}
\]
Finite surjective pullback preserves nefness, so $\|f^*D\|_L=d\|D\|_L$. Since $f^*:N^1(Y)_{\mb{R}}\to N^1(Y)_{\mb{R}}$ is invertible, this yields
\[
\|(f^*)^mD\|_L=d^m\|D\|_L
\]
for all $m\in \mb{Z}$. Moreover, $(f^*)^{\pm1}$ preserves $\operatorname{Nef}(Y)$.
By \cite[Proposition~2.9]{MZ18}, there is an ample class
$H\in N^1(Y)_{\mathbb R}$ satisfying $f^*H\equiv dH$.
Hence \cite[Proposition 1.1]{MZ18} implies that $f$ is $d$-polarized.
\end{proof}

\begin{definition}\label{def_relation}
A bi-finite correspondence $R$ is an \emph{equivalence relation}
if it satisfies:
\begin{enumerate}
\renewcommand{\labelenumi}{(\theenumi)}
    \item every component of $R$ has multiplicity one;
    \item $\Delta\subset \Supp(R)$
    \item $\trans R = R$;
    \item $\mathrm{Supp}(R \circ R) = \mathrm{Supp}(R)$;
\end{enumerate}
Equivalently, its support defines an equivalence relation on $X(\overline{\mathbf{k}})$.
\end{definition}

With its reduced structure, $\operatorname{Supp}(R)$ is a
pure-dimensional finite set-theoretic equivalence relation. By \cite[Lemma~21]{Kol12}, its geometric quotient $X/R$ exists and is a normal projetive variety. Moreover, the quotient map $\pi:X\to X/R$ is finite surjective, and
\[
\pi(x)=\pi(y)\quad\Longleftrightarrow\quad(x,y)\in\operatorname{Supp}(R)
\]
for all $x,y\in X(\overline{\mathbf{k}})$.

\begin{lemma}\label{lem_descent}
Let $R$ be an equivalence relation on $X$, and let
$\pi:X\to Y:=X/R$ be its geometric quotient.
\begin{enumerate}
\renewcommand{\labelenumi}{(\theenumi)}
\item
Let $c$ be a bi-finite correspondence satisfying
\[
\operatorname{Supp}(c\circ R\circ\trans c)\subseteq\operatorname{Supp}(R).
\]
Then there is a unique finite surjective endomorphism $f:Y\to Y$
such that
\[
(\pi\times\pi)\bigl(\operatorname{Supp}(c)\bigr)=\Gamma_f.
\]

\item If $c$ in \textup{(1)} is numerically $d$-polarized, then $f$
is $d$-polarized.
\end{enumerate}
\end{lemma}

\begin{proof}
Let $\Gamma:=(\pi\times\pi)\bigl(\operatorname{Supp}(c)\bigr)$ with its reduced structure. Both projections $\Gamma\to Y$
are finite and surjective. If $(x,x'),(y,y')\in\operatorname{Supp}(c)$ and
$\pi(x)=\pi(y)$, then $(x,y)\in\operatorname{Supp}(R)$.
The hypothesis gives $(x',y')\in\operatorname{Supp}(R)$, so
$\pi(x')=\pi(y')$.
Thus the first projection $\Gamma\to Y$ is one-to-one on
geometric points. In characteristic zero it is birational,
hence an isomorphism because $Y$ is normal.
Therefore $\Gamma$ is the graph of the required endomorphism and (2) follows directly from Lemma~\ref{lem_polardes}.
\end{proof}

\begin{remark}\label{rmk_bifin_corr}
Let $f:X\to X$ be a surjective endomorphism.
For every bi-finite correspondence $c$, one has
\[
\operatorname{Supp}
\bigl(\Gamma_f\circ c\circ\trans\Gamma_f\bigr)=(f\times f)\bigl(\operatorname{Supp}(c)\bigr).
\]
In particular, if
\[
(f\times f)\bigl(\operatorname{Supp}(R)\bigr)
=\operatorname{Supp}(R),
\]
then $f$ descends to a surjective endomorphism $g:Y\to Y$. Suppose further that $c$ satisfies Lemma~\ref{lem_descent}\textup{(1)}, inducing $h:Y\to Y$.
If $(f\times f)\bigl(\operatorname{Supp}(c)\bigr)=\operatorname{Supp}(c)$, then $g$ and $h$ commute. Indeed, applying $\pi\times\pi$ gives $(g\times g)(\Gamma_{h})=\Gamma_{h}$, which implies $g\circ h=h\circ g$.    
\end{remark}

\section{AT-varieties and exceptional maps}

\subsection{AT-varieties}\label{sec_AT}
We first recall the notion of AT-varieties. 

\begin{definition}
Let $\mathbf{k}$ be a field. Let $A$ be an abelian variety over $\mathbf{k}$, $T$ be a torus, and $\pi:Y\to A$ be a $T$-torsor over $A$. We call the triple $(Y,A,\pi)$ an {\em AT-variety}. Here, we allow the abelian variety $A$ to be a point.
\end{definition}
Here AT stands for abelian-by-torus. For $\dim T=0$, we have $Y=A$. Any semiabelian variety is an AT-variety. However, not every AT-variety admits a group structure. Since an AT-variety $Y$ is a $T$-torsor, there is an action $\rho:T\times Y\to Y$. For simplicity, we denote by $t\cdot x:=\rho(t,x)$ for $t\in T(\mathbf{k})$ and $x\in Y(\mathbf{k})$. This action induces an isomorphism 
\[(\rho,p_2): T\times Y\to Y\times_AY,\quad (t,x)\mapsto (t\cdot x,x).\]
For morphisms between AT-varieties, we have the following lemma.

\begin{lemma}[{\cite[Lemma 2.1]{Bet25}}]\label{lem_ATmap}
Let $(Y,A,\pi)$ and $(Y',A',\pi')$ be AT-varieties over $\mathbf{k}$ and $\psi_Y:Y\to Y'$ be a morphism of $\mathbf{k}$-varieties. Then there exists a unique homomorphism $\psi_T:T\to T'$ and a unique morphism $\psi_A:A\to A'$ such that the squares
\[
\begin{tikzcd}
Y \arrow[d, "\pi"] \arrow[r, "\psi_Y"] & Y' \arrow[d, "\pi'"] \\
A \arrow[r, "\psi_A"]                & A'                  
\end{tikzcd} 
\qquad
\begin{tikzcd}
T\times Y \arrow[d, "\rho"] \arrow[r, "\psi_T\times \psi_Y"] & T'\times Y' \arrow[d, "\rho'"] \\
Y \arrow[r, "\psi_Y"]                                        & Y'                             
\end{tikzcd}
\]
commutes, where $\rho$ and $\rho'$ are the action maps.
\end{lemma}

\begin{lemma}\label{lem_ATcomp}
Let $(Y,A,\pi)$ and $(Y',A',\pi')$ be  AT-varieties over a field $\mathbf{k}$, with associated tori $T$ and $T'$. Assume $\psi,\varphi:Y\to Y'$ are morphisms with $\psi_A=\varphi_A$ and $\psi_T=\varphi_T$. Then there exists a unique element $t_0\in T'(\mathbf{k})$ such that $\psi(x)=t_0\cdot\varphi(x)$ for all $x\in Y$.
\end{lemma}

\begin{proof}
Since the maps $\psi$ and $\varphi$ lie over the same morphism $A\to A'$. The $T'$-torsor structure of $Y'\to A'$ therefore yields a unique morphism $\delta:Y\to T'$ such that
\[
\psi(x)=\delta(x)\cdot\varphi(x).
\]
For $t\in T$, equivariance gives
\[\psi_T(t)\cdot\psi(x)=\psi(t\cdot x)=\delta(t\cdot x)\cdot \varphi(t\cdot x)=\big(\delta(t\cdot x)\varphi(t)\big)\cdot \varphi(x),\]
Since $\psi_T=\varphi_T$ and the $T'$-action is free, we obtain $\delta(t\cdot x)=\delta(x)$. Thus $\delta$ is $T$-invariant and descends to a morphism $A\to T'$, which must be a constant map, with value $t_0\in T'(\mathbf{k})$. Hence $\psi(x)=t_0\cdot\varphi(x)$ for all $x\in Y$. Uniqueness follows from the freeness of the $T'$-action.
\end{proof}

For simplicity, we say that the morphism $\psi_Y:Y\to Y'$ lies over the morphism $\psi_A:A\to A'$ and under the homomorphism $\psi_T:T\to T'$. We say a morphism $\psi_Y:Y\to Y'$ of AT-varieties an \emph{isogeny} if $\phi_T:T\to T'$ is an isogeny and $\phi_A:A\to A'$ is a finite and surjective morphism. Note that, when $\mr{char\,}\mathbf{k}=0$, any isogeny of AT-varieties over $\mathbf{k}$ is finite étale. 

\medskip
We now recall the $d$-lifts of AT-varieties.
\begin{definition}\label{defidlift}
Let $d\geq 2$ be an integer, and let $(Y,A,\pi)$ be an AT-variety with associated torus $T$. A self-isogeny $\psi : Y\to Y$ is called a \emph{$d$-lift} if the induced morphism $\psi_A:A\to A$ is a $d$-polarized isogeny and $\psi_T=[d]_T$, where
$$
[d]_T :T\longrightarrow T,\qquad t\longmapsto t^d
$$
is the multiplication-by-$d$ isogeny of $T$.
\end{definition}

\begin{remark}
For an isogeny $\psi:Y\to Y'$ of AT-varieties, the induced morphism $\psi_A:A\to A'$ need not be an isogeny of abelian varieties. However, in the definition of a $d$-lift $\psi:Y\to Y$, we require that $\psi_A:A\to A$ be an isogeny of abelian varieties.
\end{remark}

Recall that if $Y\to A$ is an AT-variety and $\psi_A:A'\to A$ is a morphism from an abelian variety $A'$, then the pullback $\psi_A^*Y:=Y\times_A A'$ is naturally a $T$-torsor over $A'$ and is equipped with a canonical morphism $\psi_A^*Y\to Y$. Similarly, given a homomorphism of tori $\psi_T:T\to T'$, the pushout $\psi_{T*}Y$ is the quotient of $T'\times Y$ by the $T$-action
\[
g\cdot(g',\widetilde{x})=\bigl(g'\psi_T(g)^{-1},g\widetilde{x}\bigr).
\]
It is naturally a $T'$-torsor over $A$ and is equipped with a canonical morphism $Y\to\psi_{T*}Y$.

Now let $Y$ and $Y'$ be AT-varieties with associated tori $T$ and $T'$ and abelian varieties $A$ and $A'$, respectively. Let $\psi_A:A\to A'$ be a morphism and $\psi_T:T\to T'$ an isogeny. By \cite[Lemma 2.3]{Bet25}, there exists a morphism $\psi:Y\to Y'$ inducing both $\psi_A$ and $\psi_T$ if and only if
\[
\psi_{T*}Y\simeq\psi_A^*Y'
\]
as $T'$-torsors over $A$.

Let $T$ be a split algebraic torus over $\mathbf{k}$, and fix an isomorphism $T\simeq\mb{G}_{m,\mathbf{k}}^r$. An AT-variety $Y\to A$ with associated torus $T$ can then be written as
\[
Y=L_1^\times\times_A\cdots\times_A L_r^\times,
\]
where each $L_i$ is a line bundle on $A$ and $L_i^\times$ denotes the complement of its zero section. Thus $Y\to A$ is determined, up to isomorphism as a $T$-torsor, by the tuple $(A,L_1,\dots,L_r)$, with $L_i\in\operatorname{Pic}(A)$. The following lemma is proved by the same argument as \cite[Theorem 2.8]{Bet25}.

\begin{lemma}\label{lem_ATlift}
Let $(Y,A,\pi)$ be an AT-variety over a field $\mathbf{k}$, with associated torus $T\simeq \mb{G}_{m,\mathbf{k}}^r$. Suppose that the $T$-torsor $Y$ corresponds to a tuple $(A,L_1,\dots,L_r)$, where $L_i\in\operatorname{Pic}(A)$ for $1\leq i\leq r$. Let $\psi_A:A\to A$ be a $d$-polarized isogeny satisfying $\psi_A^*L_i\simeq L_i^{\otimes d}$ for all $1\leq i\leq r$. Then there exists a $d$-polarized isogeny $\psi:Y\to Y$ such that $\pi\circ\psi=\psi_A\circ\pi$.
\end{lemma}

\begin{proof}
Write $Y$ as a $T$-torsor over $A$, and consider the isogeny $\psi_T=[d]_T:T\to T$. The pushout $\psi_{T*}Y$ corresponds to the tuple
\[
(A,L_1^{\otimes d},\dots,L_r^{\otimes d}),
\]
whereas the pullback $\psi_A^*Y$ corresponds to
\[
(A,\psi_A^*L_1,\dots,\psi_A^*L_r).
\]
The isomorphisms $\psi_A^*L_i\simeq L_i^{\otimes d}$ therefore yield an isomorphism $\psi_{T*}Y\simeq\psi_A^*Y$ of $T$-torsors over $A$. By \cite[Lemma 2.3]{Bet25}, there exists a morphism $\psi:Y\to Y$ lying over $\psi_A$ and inducing $\psi_T=[d]_T$ on the associated torus. This is the desired $d$-polarized isogeny.
\end{proof}

We prove the following property of isogenies, which confirms a conjecture in \cite[Remark 2.5]{Bet25}.

\begin{proposition}\label{prop_etaleAT}
Let $\mathbf{k}$ be an algebraically closed field of characteristic zero, and let $(Y,A,\pi)$ be an AT-variety. Let $\psi\colon P\to Y$ be a connected finite \'etale cover. Then there exist a torus $T'$, an abelian variety $B$, and a morphism $p:P\to B$ such that $P$ is a $T'$-torsor over $B$. Moreover, $\psi:P\to Y$ is an isogeny of AT-varieties.    
\end{proposition}

\begin{proof}
By a standard spreading-out argument and the Lefschetz principle, it suffices to prove the statement over $\mathbf{k}=\mathbb{C}$.

\textbf{Step 1.} We first construct $B$, $T'$, and $p$. Let $\rho_Y:T\times Y\to Y$ denote the given $T$-action. Since the universal cover $\mathbb{C}^{\dim A}$ of $A^{\an}$ is contractible, we have $\pi_2(A^{\an})=0$. The homotopy exact sequence associated with the fiber bundle $Y^{\an}\to A^{\an}$ therefore gives an exact sequence
\[
1\longrightarrow \pi_1(T^{\an})
 \longrightarrow \pi_1(Y^{\an})
 \xrightarrow{\pi_*}\pi_1(A^{\an})
 \longrightarrow 1.
\]
Set
\[
N:=\pi_1(T^{\an})\simeq\mathbb{Z}^r,\qquad
\Gamma:=\pi_1(Y^{\an}),\qquad
\Lambda:=\pi_1(A^{\an})\simeq\mathbb{Z}^{2\dim A}.
\]
Furthermore, let
\[
H:=\psi_*\pi_1(P^{\an})\subseteq\Gamma,\qquad
N':=H\cap N,\qquad
\Lambda':=\pi_*(H)\subseteq\Lambda.
\]
Because $\psi^{\an}:P^{\an}\to Y^{\an}$ is a connected finite covering, the subgroup $H$ has finite index in $\Gamma$. Consequently, $N'$ and $\Lambda'$ have finite index in $N$ and $\Lambda$, respectively.

The subgroup $\Lambda'\subseteq\Lambda$ determines a connected finite topological covering
\[
\psi_B^{\an}:B^{\an}\longrightarrow A^{\an}.
\]
By the Riemann existence theorem \cite[Expos\'e XII, Th\'eor\`eme 5.1]{SGA1}, this covering algebraizes to a connected finite \'etale morphism $\psi_B:B\longrightarrow A$. By \cite[Th\'eor\`eme 2]{LS57}, after choosing a point of $B$ lying over $0\in A$, the variety $B$ admits a unique structure of an abelian variety for which $\psi_B$ is an isogeny.

Let $\widetilde{T}\simeq\mathbb{C}^r$ be the universal cover of $T^{\an}$. The finite-index inclusion $N'\subseteq N$ determines a connected finite covering
\[
\widetilde{T}/N'\longrightarrow \widetilde{T}/N=T^{\an}.
\]
By the Riemann existence theorem, this covering algebraizes to a morphism $\psi_{T'}:T'\longrightarrow T$, where $(T')^{\an}\simeq \wt{T}/N'$. By \cite[Proposition 1.1]{BS13}, $T'$ admits a structure of an algebraic tours for which $\psi_{T'}$ is an isogeny.

Now consider $f^{\an}:=\pi^{\an}\circ\psi^{\an}:P^{\an}\longrightarrow A^{\an}$. By construction,
\[
f_*\pi_1(P^{\an})=\pi_*(H)=\Lambda'.
\]
The covering-space lifting criterion therefore gives a lift $p^{\an}:P^{\an}\longrightarrow B^{\an}$ such that $\psi_B^{\an}\circ p^{\an}=\pi^{\an}\circ\psi^{\an}$. By the Riemann existence theorem, this analytic lift is algebraic. We thus obtain a morphism $p:P\longrightarrow B$ satisfying $\psi_B\circ p=\pi\circ\psi$.

\textbf{Step 2.} We next construct the $T'$-action on $P$. Consider the algebraic morphism
\[
\beta
 :=\rho_Y\circ(\psi_{T'}\times\psi):
 T'\times P\longrightarrow T\times Y\longrightarrow Y.
\]
Note that
\[
\pi_1\bigl((T')^{\an}\times P^{\an}\bigr)
 \simeq \pi_1((T')^{\an})\times\pi_1(P^{\an})
 \simeq N'\times H.
\]
It follows from \cite[Lemma 2.11]{Bet25} that
\[
\beta_*^{\an}\pi_1\bigl((T')^{\an}\times P^{\an}\bigr)=N'H\subseteq\Gamma.
\]
Since $N'\subseteq H$, we have $N'H=H$. Hence, after choosing a base point $(e,y_0)\in (T')^{\an}\times P^{\an}$, the covering-space lifting criterion gives a unique lift
\[
\rho_P^{\an}:(T')^{\an}\times P^{\an}\longrightarrow P^{\an}
\]
satisfying
\[
\psi^{\an}\bigl(\rho_P^{\an}(t,y)\bigr)
 =\rho_Y^{\an}\bigl(\psi_{T'}^{\an}(t),\psi^{\an}(y)\bigr),
\]
and sending the chosen point $(e,y_0)$ to $y_0$. By the uniqueness of lifts, $\rho_P^{\an}$ defines an analytic $(T')^{\an}$-action on $P^{\an}$. By the Riemann existence theorem once again, $\rho_P^{\an}$ comes from an algebraic action $\rho_P:T'\times P\longrightarrow P$. Moreover, the uniqueness of lifts shows that the action $\rho_P$ preserves the fibers of $p:P\to B$.

\textbf{Step 3.} It remains to prove that $p:P\to B$ is a $T'$-torsor. We first claim that the fibers of $p^{\an}$ are connected. Let $Y_B:=Y\times_A B$, and let $q:Y_B\to B$ be the second projection. The morphism $\psi:P\to Y$ induces a morphism
\[
\varphi:=(\psi,p):P\longrightarrow Y_B.
\]
The covering $Y_B^{\an}\to Y^{\an}$ corresponds to the subgroup
\[
\Gamma_B:=\pi_1(Y_B^{\an})
 =\pi_*^{-1}(\Lambda')\subseteq\Gamma.
\]
The morphism $\varphi^{\an}$ is the finite covering associated with the inclusion $H\subseteq\Gamma_B$.

Fix $b\in B^{\an}$, and set $P_b^{\an}:=(p^{\an})^{-1}(b),(Y_B^{\an})_b:=(q^{\an})^{-1}(b)$. Then $P_b^{\an}=(\varphi^{\an})^{-1}\bigl((Y_B^{\an})_b\bigr)$. The fiber $(Y_B^{\an})_b$ is isomorphic to $T^{\an}$, and the image of its fundamental group in $\Gamma_B$ is $N$. Standard covering-space theory (see \cite[Theorem 3.4.10]{Geo08}) gives a bijection between the path-connected components of $P_b^{\an}$ and the double cosets $H\backslash\Gamma_B/N$. By the construction of $\Gamma_B$, we have $\Gamma_B=HN$. It follows that $H\backslash\Gamma_B/N$ consists of a single element. Thus $P_b^{\an}$ is connected.

Set $a:=\psi_B^{\an}(b)\in A^{\an}$. The restriction $\theta:P_b^{\an}\longrightarrow Y_a^{\an}$ of $\psi^{\an}$ is therefore a connected finite covering. Under the identification $\pi_1(Y_a^{\an})\simeq N$, the subgroup associated with this covering is
\[
\theta_*\pi_1(P_b^{\an})=H\cap N=N'.
\]
Choose $y\in P_b^{\an}$. The point $\psi^{\an}(y)\in Y_a^{\an}$ gives an isomorphism
\[
T^{\an}\xrightarrow{\sim}Y_a^{\an},
\qquad
t\longmapsto t\cdot\psi^{\an}(y).
\]
Under this identification, both $\psi_{T'}^{\an}:(T')^{\an}\longrightarrow T^{\an}$ and $\theta:P_b^{\an}\longrightarrow Y_a^{\an}$ are connected coverings corresponding to the same subgroup $N'\subseteq N$. Thus, by the uniqueness of the connected covering corresponding to $N'$, the orbit map
\[
(T')^{\an}\longrightarrow P_b^{\an},
\qquad
t\longmapsto t\cdot y,
\]
is an isomorphism. Hence the $(T')^{\an}$-action is free and transitive on every fiber of $p^{\an}$. It follows that the map
\[
(T')^{\an}\times P^{\an}
\longrightarrow
P^{\an}\times_{B^{\an}}P^{\an},
\qquad
(t,y)\longmapsto(t\cdot y,y),
\]
is biholomorphic. Hence the algebraic morphism
\[
T'\times P\longrightarrow P\times_B P
\]
induced by $\rho_P$ is an isomorphism. Since $p:P\to B$ is smooth and surjective, it is a $T'$-torsor. Finally, since $\psi:P\to Y$ is finite \'etale and surjective, it is an isogeny of AT-varieties.
\end{proof}

\begin{lemma}\label{lem_isolift}
Let $\mathbf{k}$ be an algebraically closed field of characteristic zero. Let $(Y,A,\pi)$ and $(Y',A',\pi')$ be AT-varieties over $\mathbf{k}$, and let $\varphi:Y\to Y,\psi:Y'\to Y'$ and $f:Y\to Y'$ be isogenies of AT-varieties satisfying $\psi\circ f=f\circ\varphi$. Assume $\psi$ is a $d$-lift of $Y'$, then, up to changing the base point of $A$, $\varphi$ is a $d$-lift of $Y$.
\end{lemma}
\begin{proof}
By Lemma~\ref{lem_ATmap}, the isogenies $\varphi$, $\psi$, and $f$
induce morphisms
\[
\begin{aligned}
\varphi_A&:A\to A, & \psi_{A'}&:A'\to A', & f_A&:A\to A',\\
\varphi_T&:T\to T, & \psi_{T'}&:T'\to T', & f_T&:T\to T'.
\end{aligned}
\]
Moreover, the uniqueness part in Lemma~\ref{lem_ATmap} gives
\[
f_A\circ\varphi_A=\psi_{A'}\circ f_A
\quad\text{and}\quad
f_T\circ\varphi_T=\psi_{T'}\circ f_T.
\]
Up to replacing the base of $A$ by a fixed point of $\varphi_A$, $\varphi_A$ is an isogeny of $A$. Since $\psi_{A'}$ is a $d$-polarized, $\varphi_A$ is also $d$-polarized.
Let $[d]_T:T\to T$ and $[d]_{T'}:T'\to T'$ denote the power maps given by $(t_1,\dots,t_r)\mapsto (t_1^d,\dots,t_r^d)$. By definition, $\psi_{T'}=[d]_{T'}$. It remains to prove that $\varphi_T=[d]_T$. By the uniqueness of $\varphi_T$, it suffices to show that $\varphi(t\cdot x)=[d]_T(t)\cdot\varphi(x)$ for all $t\in T$ and $x\in Y$.
Indeed,
\[
\pi\bigl(\varphi(t\cdot x)\bigr)
=\varphi_A(\pi(x))
=\pi\bigl([d]_T(t)\cdot\varphi(x)\bigr).
\]
Hence $\varphi(t\cdot x)$ and $[d]_T(t)\cdot\varphi(x)$ lie in the same fiber of $\pi$. Let
\[
\delta:Y\times_A Y\to T\times Y\to T
\]
be the difference morphism characterized by $x_1=\delta(x_1,x_2)\cdot x_2$ for every $(x_1,x_2)\in Y\times_A Y$. Define
\[
\theta:T\times Y\to Y\times Y\to T,\qquad (t,x)
\mapsto\delta\bigl(\varphi(t\cdot x),[d]_T(t)\cdot\varphi(x)\bigr).
\]
Then
\[
\varphi(t\cdot x)
=\theta(t,x)\cdot\bigl([d]_T(t)\cdot\varphi(x)\bigr).
\]
Applying $f$, we obtain
\[
f\bigl(\varphi(t\cdot x)\bigr)
=f_T\bigl(\theta(t,x)\bigr)\cdot
f\bigl([d]_T(t)\cdot\varphi(x)\bigr).
\]
On the other hand, using the equivariance of $f$ and $\psi$, we have
\[
\begin{aligned}
f\bigl(\varphi(t\cdot x)\bigr)&=\psi\bigl(f(t\cdot x)\bigr)=\psi_{T'}\bigl(f_T(t)\bigr)\cdot\psi(f(x))=[d]_{T'}\bigl(f_T(t)\bigr)\cdot f(\varphi(x))\\
&=f_T\bigl([d]_T(t)\bigr)\cdot f(\varphi(x))=f\bigl([d]_T(t)\cdot\varphi(x)\bigr).
\end{aligned}
\]
Since the $T'$-action on $Y'$ is free, comparison of the preceding two equalities gives $f_T\bigl(\theta(t,x)\bigr)=e_{T'}$. Thus the image of $\theta$ is contained in the finite set
$\ker(f_T)$. Since $T\times Y$ is connected, $\theta$ is constant. Moreover, $\theta(e_T,x)=e_T$ for every $x\in Y$. Therefore, $\theta(t,x)=e_T$ for all $t\in T$ and $x\in Y$. Hence $\varphi(t\cdot x)=[d]_T(t)\cdot\varphi(x)$, which completes the proof.
\end{proof}

\subsection{Basic properties of exceptional maps}\label{sec_exp}
In this subsection, we assume that $\mathbf{k}$ is an algebraically closed field of characteristic zero. Our main goal is to prove the following theorem.

\begin{theorem}\label{thm_semiconj}
Let $\pi: X\to Y$ be a finite surjective morphism between normal projective varieties over $\mathbf{k}$. Let $f:X\to X$ and $g: Y\to Y$ be $d$-polarized endomorphisms satisfying $\pi\circ f=g\circ \pi$. Then $f$ is exceptional if and only if $g$ is exceptional.
\end{theorem}

We first do some preparations.
\begin{lemma}\label{lem_invar_empty}
Let $X$ be a normal projective variety over $\mathbf{k}$ and $f:X\to X$ be a polarized endomorphism. Let $Z\subset X$ be a non-empty closed subset satisfying $f^{-1}(Z)\subset Z$. Assume $f$ is \'etale at the generic point of $Z$, then $Z=X$.
\end{lemma}
\begin{proof}
We may first reduce to the case that $Z$ is irreducible and $f^{-1}(Z)=Z$. In fact, we have
\[Z\supset f^{-1}(Z)\supset f^{-2}(Z)\supset\cdots.\]
Since $X$ is noetherian, $f^{-i}(Z)=f^{-i-1}(Z)$ for some $i$. Thus $f^{-1}(Z)=Z$. Note that $f$ permutes the irreducible components of $Z$. Replacing $f$ by an iteration, we may assume $f$ fixes every irreducible component of $Z$. It follows that every irreducible component is totally invariant under $f$. 

Now, we assume $Z$ is irreducible and $f^{-1}(Z)=Z$. Suppose $\dim X=n$ and $\dim Z=r$. Let $L$ be an ample line bundle on $X$ such that $f^*L\simeq dL$. By the projection formula,
\[f_*f^*Z=\deg(f)Z=d^{n}Z.\]
So we have
\[d^r(f^*Z\cdot L^r)=(f^*Z\cdot (f^*L)^r)=(f_*f^*Z\cdot L^r)=d^n(Z\cdot L^r).\]
On the other hand, since $f$ is \'etale at the generic point of $Z$, the cycle $f^*Z$ is reduced. Since $f^{-1}(Z)=Z$, we obtain $f^*Z\leq Z$. Thus
\[(f^*Z\cdot L^r)\leq (Z\cdot L^r).\]
Hence $n=r$ and $Z=X$.
\end{proof}

\begin{lemma}\label{lem_etale}
Let $X$ be a normal projective variety, let $f:X\to X$ be a
polarized endomorphism, and let $X^{\circ}\subset X$ be a nonempty
open subset satisfying $f^{-1}(X^{\circ})=X^{\circ}$. Suppose that there are a smooth variety $P$, a finite surjective morphism $\pi:X^{\circ}\to P$, and a finite \'etale endomorphism $\psi:P\to P$ such that
\[
\pi\circ f_{X^{\circ}}=\psi\circ \pi,
\]
where $f_{X^{\circ}}:=f|_{X^{\circ}}:X^{\circ}\to X^{\circ}$. Then both $\pi$ and $f_{X^{\circ}}$ are \'etale.
\end{lemma}
\begin{proof}
Let $R_\pi$ and $R_{f_{X^{\circ}}}$ denote the ramification divisors
of $\pi$ and $f_{X^{\circ}}$, respectively.
Since $\psi$ is \'etale, the ramification formula gives
\[
R_\pi=R_{f_{X^{\circ}}}+f_{X^{\circ}}^*R_\pi.
\]
Consequently,
\[
R_\pi\geq f_{X^{\circ}}^*R_\pi
\geq(f_{X^{\circ}}^2)^*R_\pi
\geq\cdots.
\]
We obtain $f_{X^{\circ}}^*R_\pi=R_\pi,$ and $R_{f_{X^{\circ}}}=0$. In particular, $f_{X^{\circ}}$ is \'etale in codimension one.

Set $D:=\Zar{\operatorname{Supp}R_\pi}\subset X$. Since $f^{-1}(X^{\circ})=X^{\circ}$, we have
\[
f^{-1}(D)=\overline{f_{X^{\circ}}^{-1}(\operatorname{Supp}R_\pi)}=D.
\]
If $D$ were nonempty, it would be a proper closed subset of $X$.
Moreover, the generic points of its irreducible components lie
in $X^{\circ}$ and have codimension one, so $f$ is \'etale at these
points. Lemma~\ref{lem_invar_empty} would then imply $D=X$,
a contradiction. Thus $R_\pi=0$.

Since $X^{\circ}$ is normal and $P$ is smooth, purity of the branch
locus implies that $\pi$ is \'etale
\cite[Tag~0BMB]{stacks-project}.
The identity $\pi\circ f_{X^{\circ}}=\psi\circ\pi$ then implies that $f_{X^{\circ}}$ is also \'etale.
\end{proof}

\begin{proposition}\label{prop_semiconj1}
Let $\pi:X\to Y$ be a finite surjective morphism between normal
projective varieties over $\mathbf{k}$. Let $f:X\to X$ and
$g:Y\to Y$ be $d$-polarized endomorphisms, where $d\geq 2$, satisfying
$\pi\circ f=g\circ\pi$. If $f$ is exceptional, then $g$ is exceptional.
\end{proposition}
\begin{proof}
After replacing $f$ and $g$ by compatible iterates, choose an
exceptional presentation
\[
h:P\to X^{\circ}\subset X,
\qquad
h\circ\psi=f\circ h,
\]
where $P$ is an AT-variety, and $X^{\circ}\subset X$ is an open dense subset. Then $f^{-1}(X^{\circ})=X^{\circ}$ by noetherianity and surjectivity.

Let $p:W\to X$ be the normalization of $X$ in $\mathbf{k}(P)$.
Then $p^{-1}(X^{\circ})=P$, and $\psi$ extends to a polarized
endomorphism $\wt{\psi}:W\to W$ satisfying
\[
p\circ\wt{\psi}=f\circ p,
\qquad
\wt{\psi}^{-1}(P)=P.
\]
Set
\[
q:=\pi\circ p,\qquad E:=W\setminus P,\qquad V:=q(P).
\]
The morphism $q$ is finite and universally open, so $V$ is open. We show that $q^{-1}(V)=P$.

Suppose otherwise, so $P\times_Y E\neq\varnothing$.
Consider
\[
Z:=(W\times_Y W)_{\mathrm{red}},
\qquad
\alpha:=(\wt{\psi},\wt{\psi}):Z\to Z.
\]
Both projections $Z\to W$ are finite and open, so every
irreducible component dominates both factors.
Thus $\alpha$ sends components to components.
Since $\alpha$ preserves $P\times_Y E$, some periodic component
$C$ meets this locus.

Let $\wt{C}$ be the normalization of $C$.
For some $m\geq1$, $\alpha^m(C)=C$ and the restriction $\alpha^m|_C:C\to C$ lifts to an
endomorphism $\varphi:\wt{C}\to\wt{C}$.
The two projections $\pi_1,\pi_2:\wt{C}\to W$ are finite
surjective and satisfy
\[
\pi_i\circ\varphi=\wt{\psi}^{\,m}\circ\pi_i,
\qquad i=1,2.
\]
In particular, $\varphi$ is an polarized endomorphism

Set $P':=\pi_1^{-1}(P)$.
Then $\varphi^{-1}(P')=P'$. Lemma~\ref{lem_etale} therefore implies that $\varphi|_{P'}$
is \'etale.

Now, set
\[
D:=\Zar{P'\cap\pi_2^{-1}(E)}
\subseteq\pi_2^{-1}(E)\subsetneq\wt{C}.
\]
By the construction of $\wt{C}$, we have $D\neq \varnothing$. Since both $P'$ and $\pi_2^{-1}(E)$ are totally invariant, we have $\varphi^{-1}(D)=D$.
Every irreducible component of $D$ meets $P'$, so $\varphi$
is \'etale at its generic point.
Lemma~\ref{lem_invar_empty} gives $D=\wt{C}$, a contradiction.

Hence $q^{-1}(V)=P$, and $q|_P:P\to V$ is finite surjective.
Together with
\[
q|_P\circ\psi=g\circ q|_P,
\]
we obtain the exceptionality of $g$.
\end{proof}

\begin{proposition}\label{prop_semiconj2}
Let $\pi:X\to Y$ be a finite surjective morphism between normal
projective varieties over $\mathbf{k}$. Let $f:X\to X$ and
$g:Y\to Y$ be $d$-polarized endomorphisms, where $d\geq 2$, satisfying
$\pi\circ f=g\circ\pi$. If $g$ is exceptional, then $f$ is exceptional.    
\end{proposition}
\begin{proof}
After replacing $f$ and $g$ by compatible iterates, choose an
exceptional presentation
\[
h:P\to Y^{\circ}\subset Y,
\qquad
h\circ\psi=g\circ h.
\]
where $P$ is an AT-variety, and $Y^{\circ}\subset Y$ is an open dense subset. Then $g^{-1}(Y^{\circ})=Y^{\circ}$. Set $X^{\circ}:=\pi^{-1}(Y^{\circ})$; thus $f^{-1}(X^{\circ})=X^{\circ}$.

Let $p:W\to Y$ be the normalization of $Y$ in $\mathbf{k}(P)$.
Then $p^{-1}(Y^{\circ})=P$, and $\psi$ extends to a polarized
endomorphism $\wt{\psi}:W\to W$ satisfying
\[
p\circ\wt{\psi}=g\circ p,
\qquad
\wt{\psi}^{-1}(P)=P.
\]
Consider
\[
Z:=(W\times_Y X)_{\mathrm{red}},
\qquad
\alpha:=(\wt{\psi},f):Z\to Z.
\]
Every irreducible component of $Z$ dominates both
factors, and the finite morphism $\alpha$ sends components to
components.

Choose a periodic component $C$, and let $\wt{C}$ be its
normalization. For some $m\geq1$, $\alpha^m(C)=C$ and the restriction $\alpha^m|_C:C\to C$ lifts to an
endomorphism $\varphi:\wt{C}\to\wt{C}$.
The projections
\[
\pi_1:\wt{C}\to W,
\qquad
\pi_2:\wt{C}\to X
\]
are finite surjective and satisfy
\[
\pi_1\circ\varphi=\wt{\psi}^{\,m}\circ\pi_1,
\qquad
\pi_2\circ\varphi=f^m\circ\pi_2.
\]
In particular, $\varphi$ is a polarized endomorphism.

Set
\[
P':=\pi_1^{-1}(P)=\pi_2^{-1}(X^{\circ}).
\]
Then $\varphi^{-1}(P')=P'$, and
$\pi_1|_{P'}\circ\varphi|_{P'}=\psi^m\circ\pi_1|_{P'}$.
Lemma~\ref{lem_etale} implies that $\pi_1|_{P'}$ is finite
\'etale. By Proposition~\ref{prop_etaleAT} and
Lemma~\ref{lem_isolift}, $P'$ admits an AT-variety structure
for which $\varphi|_{P'}$ is a $d$-lift.

Finally, $\pi_2|_{P'}:P'\to X^{\circ}$ is finite surjective and satisfies
\[
\pi_2|_{P'}\circ\varphi|_{P'}=f^m\circ\pi_2|_{P'}.
\]
Hence $f$ is exceptional.    
\end{proof}

Combine Proposition \ref{prop_semiconj1} and Proposition \ref{prop_semiconj2}, we obtain Theorem \ref{thm_semiconj}.

	\section{Dinh-Sibony's theorem for commuting polarized endomorphisms}\label{sec:Dinh-Sibony}
	\label{sec:dinh-sibony}
	
	The aim of Sections~\ref{sec:Dinh-Sibony} and~\ref{sec:commuting maps}
	is to prove Theorem~\ref{thm_commuting}. In Section~\ref{sec:Dinh-Sibony} , we recall
	Dinh--Sibony's rigidity theorem for commuting endomorphisms of $\P^k$~\cite[Theorem~1.1]{DS02} and extend it to polarized endomorphisms of normal projective varieties.
	The main result is Proposition~\ref{prop:input-data}.
	Their method extends naturally to our setting. For completeness,
	we give the full proof.

	\medskip

	In Sections~\ref{sec:Dinh-Sibony} and~\ref{sec:commuting maps}, all varieties and morphisms are defined over $\C$. A variety is assumed to be irreducible and reduced.
	Let $X$ be a normal projective variety, and let $f_1,f_2:X\to X$
	be commuting polarized endomorphisms with polarized degrees $d_1,d_2>1$.
	We assume that $d_1^m\neq d_2^n$ for all integers $m,n\geq1$. We write $\Gm$ for the multiplicative algebraic group. For $r\geq0$, write $T_r:=\Gm^r$ and
	\[
	[d]_{T_r}(t_1,\ldots,t_r):=(t_1^d,\ldots,t_r^d).
	\]


	\subsection{A common polarization}\label{sec:common-polarization}
	Put $k:=\dim X$ and write $q:=d_1$, $r:=d_2$. The case $k=0$ is immediate;
	we assume $k\geq1$.

	\medskip

	\begin{thm}\label{thm:commonL}
		Let $f_1,f_2$ commute and be individually polarized with polarized degrees $q,r>1$. Then there exists an ample line bundle $L$ such that 
		\begin{equation}\label{eq:commonL}
			f_1^*L\simeq L^{\otimes q}\quad\text{and}\quad f_2^*L\simeq L^{\otimes r}.
		\end{equation}
	\end{thm}

	\begin{proof}
		Let $V:=N^1(X)_\R$ be the real vector space of numerical classes of line
		bundles on $X$. We write $L_1\equiv L_2$ if they are numerically equivalent. 
		
		
		Choose ample line bundles $M_1,M_2$ polarizing $f_1,f_2$, and let
		$a:=[M_1]$, $b:=[M_2]$. By \cite[Proposition~2.9]{MZ18},
		the normalized averages
		\[
		\mathcal P_2:=\lim_{m\to\infty}\frac1m
		\sum_{j=0}^{m-1}r^{-j}(f_2^j)^*:V\longrightarrow V
		\]
		converge to a projection onto the $r$-eigenspace of $f_2^*$. In particular,
		$f_2^*\circ\mathcal P_2=r\mathcal P_2$ and $\mathcal P_2b=b$.
		Since pullback preserves nef classes, so does $\mathcal P_2$.
		Take $\epsilon>0$ such that $a-\epsilon b$ is nef. Then
		$\mathcal P_2a-\epsilon b$ is nef, and hence $\mathcal P_2a$ is ample.
		The commutation relation gives
		$(f_1^*\circ\mathcal P_2)(a)=(\mathcal P_2\circ f_1^*)(a)=q\mathcal P_2a$.
		Since the simultaneous eigenspace is defined over $\Q$, we may choose a
		rational ample class in it. After clearing denominators, we obtain an
		ample line bundle $H$ with $f_1^*H\equiv H^{\otimes q}$ and
		$f_2^*H\equiv H^{\otimes r}$.
		
		By \cite[Lemma~3.4]{MZ18}, a numerical polarization of
		integer degree greater than one admits a line-bundle polarization
		in the same numerical class. Thus there is an ample line
		bundle $L_0\equiv H$ such that $f_1^*L_0\simeq L_0^{\otimes q}$.
		Set $D:=f_2^*L_0\otimes L_0^{\otimes(-r)}$. Then $D$ is numerically
		trivial, and
		\[
		f_1^*D\simeq f_2^*(L_0^{\otimes q})
		\otimes L_0^{\otimes(-qr)}
		\simeq D^{\otimes q}.
		\]
		
		Let $\operatorname{Pic}^0(X)$ be the compact complex torus of line
		bundles algebraically equivalent to $\mathcal O_X$. There is an integer
		$a_0>0$ such that $D^{\otimes a_0}\in \operatorname{Pic}^0(X)$. By
		\cite[Lemma~2.3(1)]{NZ} and the identification
		$T_0\operatorname{Pic}^0(X)=H^1(X,\mathcal O_X)$, the eigenvalues of the
		differential of $f_1^*:\operatorname{Pic}^0(X)\to \operatorname{Pic}^0(X)$ have modulus $\sqrt q$. Therefore
		\[
		\operatorname{Pic}^0(X)\longrightarrow \operatorname{Pic}^0(X),\qquad
		P\longmapsto f_1^*P\otimes P^{\otimes(-q)}
		\]
		has invertible differential and finite kernel. This kernel contains
		$D^{\otimes a_0}$, so $D^{\otimes t}\simeq\mathcal O_X$ for some $t>0$.
		The line bundle $L:=L_0^{\otimes t}$ then satisfies~\eqref{eq:commonL}.
	\end{proof}
	
	\subsection{Projective extensions and homogeneous polynomial lifts}\label{sec:ambient}

	\begin{lemma}\label{lem:ambient-extension}
		Let $L$ be an ample line bundle on a normal projective variety $X$, and suppose $f_i^*L\simeq L^{\otimes q_i}$ with integers $q_i\geq2$, for $i=1,2$. A positive power of $L$ defines a projectively normal embedding $\iota:X\hookrightarrow\P^N$ such that there are endomorphisms $\widetilde f_i:\P^N\to\P^N$ of algebraic degrees $q_i$ satisfying $\widetilde f_i\circ\iota=\iota\circ f_i$ for $i=1,2$.
	\end{lemma}

	\begin{proof}
		Let $\delta:=\min(q_1,q_2)$. By the construction in the proof of
		\cite[Corollary~2.2]{Fak03}, applied to the map of degree
		$\delta$, there is a very ample power $B:=L^{\otimes a}$ such that the
		associated embedding $X\hookrightarrow\P^N$ satisfies:
		\begin{enumerate}
        \renewcommand{\labelenumi}{(\theenumi)}
			\item for every $m\geq0$, the restriction map
			\[
			H^0(\P^N,\mathcal O_{\P^N}(m))\longrightarrow H^0(X,B^{\otimes m})
			\]
			is surjective;
			\item $X$ is cut out set-theoretically by homogeneous polynomials of
			degrees at most $\delta$.
		\end{enumerate}
		For such an embedding, \cite[Proposition~2.1]{Fak03}
		shows that any morphism $u:X\to X$ with $u^*B\simeq B^{\otimes e}$,
		$e\geq\delta$, can be extended to an endomorphism of $\P^N$ of degree $e$.  Since $f_i^*B\simeq B^{\otimes q_i}$ and
		$q_i\geq\delta$, it applies to both $f_i$ in the same embedding.	By the first property,  the homogeneous coordinate ring is
		$\bigoplus_{m\geq0}H^0(X,B^{\otimes m})$. This ring is integrally
		closed because $X$ is normal. Thus the embedding is projectively normal.
	\end{proof}

	\medskip

	Apply Lemma~\ref{lem:ambient-extension} to the line bundle in
	Theorem~\ref{thm:commonL}. 
	Replacing $L$ by a power, we identify $X$ with its image in $\P^N$ and we identify  $\widetilde f_i$ as $f_i$. Let $p:\C^{N+1}\setminus\{0\}\to\P^N$ be the projectivization map $p(z):=[z]$.
	Set
	\begin{equation}\label{eq:geometric-lift}
		Z:=\{0\}\cup p^{-1}(X)\subset\C^{N+1},\qquad
		Z^\times:=p^{-1}(X)\quad\text{and}\quad p_X:=p|_{Z^\times}.
	\end{equation}
	The reduced cone $Z$ is defined by the homogeneous equations of $X$.
	It is normal because the embedding is projectively normal.

	\medskip
	
	\begin{prop}\label{prop:cone}
		There exist homogeneous polynomial lifts $F_i:\C^{N+1}\to\C^{N+1}$
		of degrees $q_1=q$ and $q_2=r$ with $p\circ F_i=f_i\circ p$.
		They are finite and surjective, satisfy $F_i^{-1}(0)=\{0\}$ and preserve $Z$.
		They can be normalized so that $F_1\circ F_2=F_2\circ F_1$ on $Z$.
		Consequently the restrictions $u_i:=F_i|_Z:Z\to Z$ satisfy
		$u_1\circ u_2=u_2\circ u_1$ and $p_X\circ u_i=f_i\circ p_X$.
		There are constants $0<c<C$ such that
		\begin{equation}\label{eq:cone-growth}
			c\|z\|^{q_i}\leq\|F_i(z)\|\leq C\|z\|^{q_i}, \;\;
			z\in\C^{N+1}
		\end{equation}
	\end{prop}

	\begin{proof}
	Let $F_i$ be  a homogeneous polynomial lift of $f_i$.
		Since $f_i(X)=X$, we have $F_i(Z)\subset Z$. The tuples $F_1\circ F_2$ and $F_2\circ F_1$ have the same degree and the same
		projectivization on $Z^\times$. Their ratio is a nowhere-zero regular
		function, homogeneous of degree zero. It descends to an invertible
		regular function on $X$ and is therefore constant. Write
		$F_1\circ F_2=cF_2\circ F_1$ on $Z$, where $c\in\C^*$. Choose
		$\lambda^{r-1}=c$. Replacing $F_1$ by $\lambda F_1$ makes the lifts commute on $Z$.
		
		Taking the minimum and maximum of $\|F_i\|$ on the unit sphere
		gives~\eqref{eq:cone-growth} by homogeneity. The restrictions $u_i$
		are finite and surjective on $Z$.
	\end{proof}

	\medskip

	\begin{prop}\label{prop:escape}
		For $i=1,2$ the locally uniform limit
		\[
		H_i(z):=\lim_{m\to\infty}q_i^{-m}\log\|F_i^m(z)\|
		\]
		exists on $\C^{N+1}\setminus\{0\}$. It is continuous and plurisubharmonic, satisfies $H_i(tz)=\log|t|+H_i(z)$ and $H_i\circ F_i=q_iH_i$, and differs from $\log\|z\|$ by a bounded function. Its restrictions satisfy $H_1|_{Z^\times}=H_2|_{Z^\times}$. Write $H$ for this common restriction. We have
		\begin{gather}\label{eq:H}
			H(tz)=\log|t|+H(z),\quad H\circ u_1=qH,\quad H\circ u_2=rH,\\
			\intertext{and}
			H(z)-\log\|z\|=O(1)\qquad(z\in Z^\times).
		\end{gather}
		The function $G^+:=\max\{H,0\}$, extended by $G^+(0):=0$, is continuous, nonnegative and plurisubharmonic on $Z$, satisfies $G^+\circ u_i=q_iG^+$, and has compact zero set $K\subset Z$.
	\end{prop}

	\begin{proof}
		By the homogeneous Green-function construction in
		\cite[Section~3, homogeneous Green-function construction before Theorem~3.1]{DS02},
		the functions $H_i$ are continuous and plurisubharmonic, logarithmically
		homogeneous, and satisfy $H_i\circ F_i=q_iH_i$. The estimate
		\eqref{eq:cone-growth} gives local uniform convergence and
		$H_i(z)-\log\|z\|=O(1)$.
		
		By Proposition~\ref{prop:cone}, the restrictions of $F_1,F_2$ preserve
		$Z$ and commute there. Since $H_1-H_2$ is bounded on $Z^\times$, the
		argument of \cite[Proposition~3.3]{DS02} gives
		$H_1|_{Z^\times}=H_2|_{Z^\times}$ and the two scaling identities for
		$H$. Finally, $H(z)=\log\|z\|+O(1)$ implies that $G^+$ vanishes near $0$ and that its zero set is compact. Thus $G^+$ extends
		continuously and plurisubharmonically across $0$.
	\end{proof}
	
	\subsection{Simultaneous Poincar\'e maps and continuous real symmetry}\label{sec:poincare}\label{sec:flow}

We write $Z_{\reg}$ for the smooth locus of $Z$. 	Suppose that $a\in Z^\times_{\reg}$ and $s\geq1$ satisfy
	$u_1^s(a)=u_2^s(a)=a$ and that $(u_1\circ u_2)^s$ is repelling at $a$.
	Fix such a pair $(a,s)$ and put $R_1:=u_1^s$, $R_2:=u_2^s$ and $d:=k+1$.
	Thus $R_1(a)=R_2(a)=a$, and $H_0:=R_1\circ R_2$ is repelling at the smooth point $a$.

	\medskip

	Fix $\rho:=(\rho_1,\ldots,\rho_d)$ with $|\rho_j|>1$.
	A polynomial automorphism $h=(h_1,\ldots,h_d)$ of $\C^d$ fixing $0$
	is called \emph{$\rho$-resonant} if every monomial $z^\alpha$ occurring
	with nonzero coefficient in $h_j$ satisfies
	\[
	\rho^\alpha:=\rho_1^{\alpha_1}\cdots\rho_d^{\alpha_d}
	=\rho_j,\qquad \alpha\in\Z_{\geq0}^d,\quad |\alpha|\geq1.
	\]
	Equivalently, $h$ commutes with the diagonal map
	$D_\rho(z):=(\rho_1z_1,\ldots,\rho_dz_d)$.
	We denote the group of these automorphisms by $\mathcal G_\rho$.
	Since $|\rho_j|>1$, only finitely many monomials satisfy these
	relations, and $\mathcal G_\rho$ is a finite-dimensional complex
	Lie group with uniformly bounded polynomial degrees
	\cite[Section~2, before Proposition~2.1]{DS02}.
	In the next proposition, $\rho_1,\ldots,\rho_d$ are the
	eigenvalues of $DH_0(a)$. Resonance in that proposition is taken with respect to
	this fixed tuple $\rho$.

	\medskip

	\begin{prop}\label{prop:poincare}
	The following hold: 
		\begin{enumerate}
		\renewcommand{\labelenumi}{(\theenumi)}
			
		\item There are commuting polynomial automorphisms $A,B$ of $\C^d$ fixing $0$, a repelling polynomial automorphism $Q:=A\circ B$, and a holomorphic map $\Psi:\C^d\to Z^\times$ with $\Psi(0)=a$ such that
		\[
		\Psi\circ A=R_1\circ\Psi
		\quad\text{and}\quad \Psi\circ B=R_2\circ\Psi.
		\]
		The map $\Psi$ is locally biholomorphic at $0$, locally finite everywhere, and open. 
		
		\item The maps $A^m\circ B^n$, for all integers $m,n$, belong to $\mathcal G_\rho$ and have a common upper bound on their polynomial degrees. Moreover $Q^{-n}\to0$ locally uniformly on $\C^d$.
		
	\end{enumerate}
	\end{prop}

	\begin{proof}
		(1) follows from  \cite[Proposition~2.2]{DS02}. Since $Q$ is repelling, $Q^{-j}\to0$ locally uniformly, as in the
		proof of \cite[Proposition~3.5]{DS02}.   Since $A,B\in\mathcal G_\rho$, and the degree bound follows from
		the definition above, hence (2) holds.
			\end{proof}


	\medskip
	
	The simultaneous Poincar\'e map also yields a continuous real symmetry of the pulled-back escape function.
	Set $U:=G^+\circ\Psi:\C^d\to\R_{\geq0}$.
	The function $U$ is continuous and plurisubharmonic, satisfies $U(0)=0$, and
	\begin{equation}\label{eq:U}
		U\circ A=q^sU\quad\text{and}\quad U\circ B=r^sU.
	\end{equation}
	Since $a$ is periodic for $u_1$ and $u_2$,  we have $G^+(a)=0$.

	\medskip

	\begin{lemma}\label{lem:no-entire}
		If $v:\C\to\C^d$ is an entire curve with $U\circ v=0$, then $v$ is constant. 
	\end{lemma}
	\begin{proof}
		Each coordinate of $\Psi\circ v$ is a bounded entire function,
		since its image lies in the compact set $K\subset Z\subset\C^{N+1}$. 
		Thus $\Psi\circ v$ is constant. By Proposition~\ref{prop:poincare},
		the fibers of $\Psi$ are discrete, so the connected image of $v$
		is a point. 
	\end{proof}
	
	\medskip
	
	\begin{thm}\label{thm:flow}
        In the setting of Section~\ref{sec:poincare}.
        Assume that the degrees are multiplicatively independent:
        \begin{equation}\label{eq:independence}
        q^m\neq r^n\qquad\text{for all integers }m,n\geq1.
        \end{equation}
        Then the following hold:
		
		\begin{enumerate}
        \renewcommand{\labelenumi}{(\theenumi)}
			\item $0$ is a repelling fixed point of $A$ and $B$.  
			
			\item The closure $\Gamma$ of $\langle A,B\rangle$ in $\mathcal G_\rho$ contains a real one-parameter subgroup $(A_t)_{t\in\R}$ satisfying
			\begin{equation}\label{eq:flow}
				U\circ A_t=e^tU
				\quad\text{and}\quad A_t\longrightarrow0\quad(t\to-\infty)
			\end{equation}
			locally uniformly on $\C^d$.
		\end{enumerate} 
	\end{thm}

	\begin{proof}
		The argument of \cite[Lemma~4.2]{DS02} shows that
		the coefficients of $A^m\circ B^n$ are bounded whenever $q^{sm}r^{sn}$
		is bounded above. It uses the finite-dimensional resonance group,
		the scaling identities~\eqref{eq:U}, and the fact that $\{U=0\}$
		contains no nonconstant entire curve, which follows from
		Lemma~\ref{lem:no-entire}. In particular, $A^{-j}$ and $B^{-j}$
		have bounded coefficients. Every subsequential limit has image
		in $\{U=0\}$ and fixes $0$, so it is the constant zero map. Hence
		$A^{-j},B^{-j}\to0$ locally uniformly. Taking derivatives at $0$,
		we conclude that both $A$ and $B$ are repelling.
		
		Applying the argument of \cite[Corollary~4.3]{DS02} to $A,B,U$
		in their resonance group, we obtain a real one-parameter subgroup
		$(A_t)$ in $\Gamma$ such that $U\circ A_t=e^tU$. The same
		corollary gives $A_t\to0$ locally uniformly as $t\to-\infty$.
	\end{proof}

	\medskip
	
	\subsection{The equilibrium measure and common repelling periodic points}
We write $(L^{\dim W})$
	for the top self-intersection of a line bundle $L$ on a projective variety $W$.

	\medskip

	\begin{prop}\label{prop:normal-equilibrium}
		Let $W$ be a normal complex projective variety of dimension $d\geq1$,
		and let $f:W\to W$ be a polarized endomorphism of polarized degree $q>1$.
		Choose an ample line bundle $L$ with $f^*L\simeq L^{\otimes q}$,
		and let $\omega$ be a smooth positive representative of
		$(L^d)^{-1/d}c_1(L)$. Then its Green current is
		$T:=\lim_{j\to\infty}q^{-j}(f^j)^*\omega$ and has continuous local potentials.
		Put $\nu:=T^d$, where products on $W$ are defined on a
		resolution and pushed down. Then the following hold.
		\begin{enumerate}
        \renewcommand{\labelenumi}{(\theenumi)}
			\item $\nu$ is a probability measure, gives no mass to pluripolar sets,
			and satisfies $f^*\nu=q^d\nu$.
			\item There is a countable union $B:=\bigcup_{j\geq0}B_j$ of proper
			algebraic subsets of $W$ such that $q^{-dn}(f^n)^*\delta_z\to\nu$
			weakly for every $z\notin B$. The inverse images here are counted with their local multiplicities.
			\item The smooth repelling periodic points belonging to $\operatorname{supp}\nu$
			are dense in $\operatorname{supp}\nu$ for the complex topology.
		\end{enumerate}
	\end{prop}

	\begin{proof}
		By \cite[Theorem~2.4(1)]{GV}, the normalized
		pullbacks of a curvature form under a polarized endomorphism of a
		complex projective variety converge to a positive closed current
		$T=\omega+dd^c\phi$, where $\phi$ is continuous and $f^*T=qT$.
		
		Let $p:W_0\to W$ be a projective resolution which is an isomorphism
		over $W_{\mathrm{reg}}$, and let $g:=p^{-1}\circ f\circ p:W_0\dashrightarrow W_0$ be the induced
		dominant rational map. Set
		$\widetilde\nu:=(p^*T)^d$.
		The $j$-th dynamical degree is $\lambda_j(f)=q^j$. By birational invariance of dynamical
		degrees \cite[Theorem~1.1(2), with the base a point]{Tru20},
		$\lambda_j(g)=q^j$ as well. Thus $g$ is a dominant
		meromorphic map on the compact K\"ahler manifold $W_0$, whose
		topological degree $q^d$ is larger than all its other dynamical degrees.
		
		The property (1) holds since $p^*T$ has bounded local potentials.
		Thus $\widetilde\nu$  also gives no mass to pluripolar sets. Moreover,
		$\int_{W_0}(p^*T)^d=\int_W\omega^d=1$,
		so $\nu=p_*\widetilde\nu$ is a probability measure.
		Moreover, $f^*(T^d)=(f^*T)^d=q^dT^d$ away from
		the singular locus. Both measures give no mass to the omitted proper analytic sets:
		for $f^*\nu$, this follows because $f$ is finite and their images
		are proper analytic sets of $\nu$-mass zero. Hence
		$f^*\nu=q^d\nu$ on $W$.


			By \cite[Theorem~1.2]{DNT},   the inverse-image convergence to
		$\widetilde\nu$ outside a countable union of proper subvarieties.
		Let $B$ contain the images under $p$ of these exceptional sets,
		together with all forward images of $W_{\mathrm{sing}}$.
	This is a countable union of proper algebraic subsets.
		Outside $B$, the inverse fibers of $f$ and $g$ correspond with
		multiplicities, so pushing the convergence down by $p$ proves (2).
		
		For (3), let $R_n$ be the set of repelling fixed points of $g^n$
		in the support of its equilibrium measure. By
		\cite[Theorem~1.1]{DNT}, these points equidistribute:
		$q^{-dn}\sum_{a\in R_n}\delta_a\to\widetilde\nu$.
		Since the closed set $p^{-1}(W_{\mathrm{sing}})$ has zero
		$\widetilde\nu$-mass, we can discard the repelling periodic points it contains
		without changing the limit. The remaining points descend through
		local biholomorphisms to smooth repelling periodic points of $f$, and their images
		are dense in $\operatorname{supp}\nu$.
	\end{proof}

	\medskip

	\begin{prop}\label{prop:cone-equilibrium}
		Let $H,G^+$ and the maps $u_i,F_i$ with degrees $q_i$ be as in
		Proposition~\ref{prop:escape}, and put $d=\dim Z=k+1$. Then the
		following hold.
		\begin{enumerate}
        \renewcommand{\labelenumi}{(\theenumi)}
			\item The measure $\nu_Z:=(L^k)^{-1}(dd^cG^+)^d$
			is a probability measure supported on the compact set $\{H=0\}$.
			It gives no mass to pluripolar sets and satisfies
			$u_i^*\nu_Z=q_i^d\nu_Z$ for $i=1,2$.
			For $u_1$, the measure $\nu_Z$ has the backward-orbit
			equidistribution property of Proposition~\ref{prop:normal-equilibrium}(2).
			\item The smooth periodic points in $J:=\operatorname{supp}\nu_Z$
			which are repelling for both $u_1$ and $u_2$ are dense in $J$.
		\end{enumerate}
	\end{prop}

	\begin{proof}
	For the extensions of $u_i$ to the projective closure of $Z$,
	the measure $\nu$ in Proposition~\ref{prop:normal-equilibrium} is
	$\nu_Z$. Its conclusion (1) gives $\nu_Z(Z)=1$,
	$\nu_Z(E)=0$ for every pluripolar set $E\subset Z$, and
	$u_i^*\nu_Z=q_i^d\nu_Z$.
	By Proposition~\ref{prop:normal-equilibrium}(2),
	$q_1^{-dn}(u_1^n)^*\delta_z\to\nu_Z$ for $z\in Z\setminus B$,
	where $B$ is a countable union of proper algebraic subsets of $Z$.
	This proves (1).
	
	For (2), put $u:=u_1\circ u_2$.
	Proposition~\ref{prop:escape} gives $G^+\circ u=(q_1q_2)G^+$,
	so the projective extension of $u$ has equilibrium measure $\nu_Z$.
	By Proposition~\ref{prop:normal-equilibrium}(3),
	$\overline{\operatorname{Per}^*(u)\cap J}=J$, where
	$\operatorname{Per}^*$ denotes smooth repelling periodic points.
	Fix $n\geq1$. The finite set
	$S:=\Fix(u^n)$ satisfies $u_i(S)\subset S$, since $u^n\circ u_i=u_i\circ u^n$.
	As $u^n|_S=\mathrm{id}_S$, each $u_i|_S$ is a permutation with inverse
	$(u_{3-i}\circ u^{n-1})|_S$. Taking $s:=n(\#S)!$, we have
	$n\mid s$ and $u_1^s|_S=u_2^s|_S=\mathrm{id}_S$.
	Thus every point of $S$ is periodic for both $u_1$ and $u_2$.
	
	Let $x\in S\cap J$ be a smooth repelling fixed point of $u^n$.
	Put $R_i:=u_i^s$. Then $R_1(x)=R_2(x)=x$ and
	$R_1\circ R_2=u^s=(u^n)^{s/n}$ is repelling at $x$.
	By Proposition~\ref{prop:poincare}, there are $A,B,\Psi$ with
	$\Psi(0)=x$ and $D\Psi(0)$ invertible such that
	\[
	\Psi\circ A=R_1\circ\Psi\quad\text{and}\quad
	\Psi\circ B=R_2\circ\Psi.
	\]
	By Theorem~\ref{thm:flow}, $A$ and $B$ are repelling at $0$.
	Hence $R_1,R_2$ are repelling at $x$ by local conjugacy, and
	$x\in\operatorname{Per}^*(u_1)\cap\operatorname{Per}^*(u_2)$.
	Varying $n$, we obtain
	\[
	\operatorname{Per}^*(u)\cap J
	\subset\operatorname{Per}^*(u_1)\cap\operatorname{Per}^*(u_2)\cap J.
	\]
	Taking closures proves (2).
	\end{proof}

\subsection{Simultaneous analytic lamination}
	
	\begin{definition}[Simultaneous analytic lamination]\label{def:simultaneous-lamination}
		Let $M$ be a real analytic manifold of dimension $D$, let $G\geq0$ be
		continuous on $M$, and let $\mu$ be a Radon measure on $M$.
		An \emph{$m$-lamination chart for $(G,\mu)$} is a real analytic
		coordinate box $\Omega:=P\times Q$, where $P\subset\R^m$ and
		$Q\subset\R^{D-m}$ are open, such that
		$\mu|_\Omega=(dp_1\otimes\cdots\otimes dp_m)\otimes\sigma$.
		Here $dp_1\otimes\cdots\otimes dp_m$ is the standard Lebesgue measure
		on $P$, and $\sigma$ is a Radon measure on $Q$.
		For every fixed $q\in Q$, the function $p\mapsto G(p,q)$ is real analytic.
		We additionally require a function $\delta:[0,\epsilon)\to(0,\infty)$,
		with $\delta(t)\to1$ as $t\to0$, such that
		\[
		G(p,q)\leq\delta(\|p-p'\|)G(p',q)
		\quad\text{when }\|p-p'\|<\epsilon.
		\]
		This is the simultaneous lamination condition in
		\cite[Section~4, definitions preceding Proposition~4.4]{DS02}.
	\end{definition}

	\medskip

	For a fixed pair $(G,\mu)$, we say that $G$ is \emph{laminated} by the projection $\Omega\to Q$
	if $\Omega$ is a lamination chart for $(G,\mu)$ as in
	Definition~\ref{def:simultaneous-lamination}.
	
	\medskip
	
	\begin{prop}
		\label{prop:cone-lamination}
		With the notation of Propositions~\ref{prop:escape}
		and~\ref{prop:cone-equilibrium}, there exist a nonempty coordinate box
		$\Omega=P\times Q\subset Z_{\mathrm{reg}}$ and an integer
		$d\leq m<2d$ such that
		\[
		J\cap\Omega=P\times\{0\}
		\quad\text{and}\quad \nu_Z|_\Omega=(dp_1\otimes\cdots\otimes dp_m)\otimes\delta_0,
		\]
		and $G^+$ is laminated by the projection $\Omega\to Q$.
	\end{prop}

	\begin{proof}
		By Proposition~\ref{prop:cone-equilibrium}(1), the measure
		$\nu_Z$ gives no mass to pluripolar sets. By Proposition~\ref{prop:cone-equilibrium} (2), common smooth repelling periodic points are dense in its
		support. Choose such a point $a\in J\cap Z_{\reg}$, let $\Psi$ be
		the simultaneous Poincar\'e map based at $a$, and let
		$(A_t)_{t\in\R}$ be the real one-parameter subgroup of resonant polynomial
		automorphisms of $\C^d$ given by Theorem~\ref{thm:flow}.
		Thus $A_0=\mathrm{id}$ and $A_{t+t'}=A_t\circ A_{t'}$ for $t,t'\in\R$,
		and $A_t\to0$ locally uniformly as $t\to-\infty$.
		In a local chart, we have
		\[
		U\circ A_t=e^tU
		\quad\text{and}\quad A_t^*(\Psi^*\nu_Z)=e^{dt}\Psi^*\nu_Z.
		\]
		The proof of \cite[Proposition~4.4]{DS02}, applied in these charts
		and using inverse branches in $Z_{\mathrm{reg}}$, gives coordinates
		in which $J\cap\Omega=P\times\{0\}$, the measure $\nu_Z$ is a
		product, and $G^+$ satisfies the simultaneous lamination condition.
		
		The non-pluripolarity argument in the same proof gives $m\geq d$.
		Since $H(ty)=\log|t|+H(y)$, the set $\{H=0\}$ has empty interior.
		Thus $J\subset\{H=0\}$ implies $m<2d$.
	\end{proof}
	
	\subsection{Spectral bounds, the quadratic model, and local affine rigidity}\label{sec:affine-repair}

	A smooth real submanifold $M\subset\C^d$ is \emph{CR-generic} if
	$T_pM+iT_pM=T_p\C^d$ for every $p\in M$.
	In other words, the real tangent space spans the ambient complex
	tangent space. For a CR-generic $M$ of real dimension $d+n$, the complex tangent space
	$T_pM\cap iT_pM$ has complex dimension $n$, called the
	\emph{CR dimension}.

	\medskip

	Let $\Lambda$ be a polynomial automorphism of $\C^d$ fixing $0$,
	and put $A:=D\Lambda(0)$. Write $\lambda_1,\ldots,\lambda_d$
	for the eigenvalues of $A$.
	We call $\Lambda$ \emph{expanding} if $|\lambda_j|>1$ for all $j$.
	Choose coordinates in which the semisimple part of $A$ is
	$D_\lambda=\operatorname{diag}(\lambda_1,\ldots,\lambda_d)$.
	The condition $\Lambda\circ D_\lambda=D_\lambda\circ\Lambda$ means
	that $\Lambda$ belongs to the group $\mathcal G_\lambda$ defined as in
	Subsection~\ref{sec:poincare}, with $\rho$ replaced by
	$\lambda=(\lambda_1,\ldots,\lambda_d)$.
	
	\medskip

	\begin{lemma}
		\label{lem:spectral-bounds}
		Let $\Lambda$ be an expanding polynomial automorphism of $\C^d$
		fixing $0$ and commuting with $D_\lambda$, with $A=D\Lambda(0)$
		and $D_\lambda$ as above.
		Let $G\geq0$ be continuous and
		plurisubharmonic on $\mathbb C^d$, with $G\circ\Lambda=qG$ for $q>1$,
		and suppose $\{G=0\}$ contains no image of a nonconstant entire curve.
		Suppose $J:=\operatorname{supp}(dd^cG)^d$ is invariant under $\Lambda$
		and near $0$ is a real analytic CR-generic submanifold $M$ of real
		dimension $d+n$, on which $(dd^cG)^d$ has a positive continuous
		density with respect to Lebesgue measure in local real coordinates
		on $M$. Suppose also $0\in J\subset\{G=0\}$.
		Put $H:=T_0M$ and $E_c:=H\cap iH$, of complex dimension $n$.
		Then $H$ has an $A$-invariant real direct sum decomposition
		$H=E_c\oplus E_r$, with $\dim_\R E_r=d-n$.
		The eigenvalues of $A|_{E_c}$ have modulus $\sqrt q$, and those of
		the complexification of $A|_{E_r}$ have modulus $q$.
	\end{lemma}

	\begin{proof}
		The proof of \cite[Lemma~5.3]{DS02} gives an upper bound $q$ for
		the modulus of every eigenvalue of $A$. On the complex tangent
		space $E_c$, the first paragraph of the proof of
		\cite[Proposition~5.4]{DS02} gives the upper bound $\sqrt q$. These
		arguments apply to the invariant spectral subspaces of $A$.
		
		Put $\mu:=(dd^cG)^d$. Since $G\circ\Lambda=qG$, we have
		$\Lambda^*\mu=q^d\mu$. The point $0$ is fixed, and the density of
		$\mu$ on $M$ is continuous and positive there. Hence
		$|\det(A|_H)|=q^d$, where the determinant is taken on the real vector space $H$, of
		dimension $d+n$. Let $\alpha_1,\ldots,\alpha_n$ be the complex
		eigenvalues of $A|_{E_c}$. The complexification of $A|_H$ has these
		eigenvalues, their conjugates, and $d-n$ further eigenvalues
		$\beta_1,\ldots,\beta_{d-n}$, all counted with algebraic multiplicity.
		Then
		\[
		q^d=\prod_{j=1}^n|\alpha_j|^2
		\prod_{j=1}^{d-n}|\beta_j|
		\leq q^nq^{d-n}=q^d.
		\]
		Thus $|\alpha_j|=\sqrt q$ and $|\beta_j|=q$.
		Let $E_r\subset H$ be the real spectral subspace whose complexification
		is the sum of the generalized eigenspaces of the complexification of
		$A|_H$ with eigenvalue modulus $q$.
		Since $q\ne\sqrt q$, we obtain $H=E_c\oplus E_r$ and
		$\dim_\R E_r=d-n$. CR-genericity gives
		$\C^d=E_c\oplus(E_r+iE_r)$ and $E_r\cap iE_r=\{0\}$.
	\end{proof}
	
	\medskip

	The complex eigenvalues of $A|_{E_r}$ need not equal $\pm q$.
	For a dynamical example, consider the endomorphism
	$f(x,y)=(y^{-2},x^2)$ of $(\P^1)^2$, polarized by
	$\mathcal O(1,1)$ with $q=2$. At its repelling fixed point $(1,1)$,
	the Poincar\'e map and the corresponding polynomial automorphism are
	\[
	\Psi(z_1,z_2)=(e^{iz_1},e^{iz_2}),\qquad
	\La(z_1,z_2)=(-2z_2,2z_1),
	\]
	so that $f\circ\Psi=\Psi\circ \La$.
	The equilibrium support $(S^1)^2$ pulls back to $\R^2$;
	thus $H=E_r=\R^2$, $E_c=0$, and
	\[
	A=D_0\La=2\begin{pmatrix}0&-1\\1&0\end{pmatrix},\qquad
	\operatorname{Spec}\bigl((A|_{E_r})_{\C}\bigr)=\{2i,-2i\}.
	\]

	\medskip

	We now use these spectral bounds to obtain a quadratic model.

	A function $\Phi\colon\R^r\to[0,\infty)$ is called
	\emph{positively one-homogeneous} if $\Phi(tv)=t\Phi(v)$ for
	every $v\in\R^r$ and every $t\geq0$.

	\medskip
	
	\begin{prop}\label{prop:quadratic-model}
		Let $U\colon\mathbb C^\kappa\to[0,\infty)$ be continuous and plurisubharmonic,
		with $U(0)=0$, and suppose that no nonconstant entire holomorphic curve is
		contained in $\{U=0\}$. Let $\La$ be a repelling $\lambda$-resonant
		polynomial automorphism such that $D_\lambda$ is the semisimple
		part of $D_0\La$, and suppose $U\circ \La=dU$, where $d>1$.
		Let $J\subset\{U=0\}$ be closed and satisfy $\La(J)=J$. Suppose that near $0$ it is a
		smooth real-analytic CR-generic submanifold of real dimension $\kappa+n$, with
		$0\leq n<\kappa$. Write $H:=T_0J$ and $E_c:=H\cap iH$.
		
		Assume that $H=E_c\oplus E_r$ is a $D_0\La$-invariant real direct sum,
		with $\dim_\R E_r=\kappa-n$. Suppose that every eigenvalue of
		$D_0\La|_{E_c}$ has modulus $\sqrt d$, and every eigenvalue of the complexification
		of $D_0\La|_{E_r}$ has modulus $d$.
		Assume also that an ambient real-analytic foliation near $0$ has $J$ as
		one plaque and satisfies the following uniform comparison: on each plaque,
		\[
		U(a)\leq\delta(\|a-b\|)U(b)\quad\text{and}\quad \delta(t)\longrightarrow1
		\quad(t\to0).
		\]
		Then after a polynomial change of coordinates, there are coordinates
		$(z,w)\in\mathbb C^n\times\mathbb C^r$, $r:=\kappa-n$, a real Hermitian
		vector-valued quadratic form $\Theta$, and a continuous convex positively
		one-homogeneous function $\Phi\colon\mathbb R^r\to[0,\infty)$ such that
		$\{v\in\C^n:\Theta(v,v)=0\}=\{0\}$ and
		\[
		J=\{\operatorname{Im}w=\Theta(z,z)\}
		\quad\text{and}\quad U(z,w)=\Phi(\operatorname{Im}w-\Theta(z,z)).
		\]
		Moreover
		\[
		\La(z,w)=(Pz,Bw),\quad \Theta(Pz,Pz)=B\Theta(z,z)\quad\text{and}\quad
		\Phi(Bv)=d\Phi(v),
		\]
		where $P\in\operatorname{GL}_n(\mathbb C)$ and
		$B\in\operatorname{GL}_r(\mathbb R)$ are matrices such that
		$P/\sqrt d$ is conjugate to a unitary matrix and
		$B/d$ is conjugate over $\mathbb R$ to an orthogonal matrix.
	\end{prop}

	\begin{proof}
		Choose complex linear coordinates with $E_c=\C^n\times\{0\}$
		and $E_r=\{0\}\times\R^r$. Taking absolute values in the resonance
		relations, a nonlinear monomial can only be quadratic in $z$ and
		occur in a $w$-component. Thus $\La(z,w)=(Pz,Bw+C(z))$,
		where $P\in\operatorname{GL}_n(\C)$
		and $B\in\operatorname{GL}_r(\R)$ are the matrices of the two
		blocks of $D_0\La$, and $C\colon\C^n\to\C^r$ is a holomorphic
		homogeneous quadratic polynomial map. Write $w=x+iy$, with
		$x,y\in\R^r$. The implicit function theorem gives, near $0$,
		$J=\{y=h(z,\bar z,x)\}$, where $h$ is a real analytic map to
		$\R^r$ with $h(0)=Dh(0)=0$.

		By the Taylor expansion argument of \cite[Proposition~5.4]{DS02},
		we can write $h(z,\bar z)=\Theta(z,z)+\operatorname{Im}p(z)$ as in
		\cite[Proposition~5.6]{DS02}.
		Here $\Theta\colon\C^n\times\C^n\to\C^r$ is complex linear in its
		first variable, conjugate linear in its second, and satisfies
		$\Theta(v,u)=\overline{\Theta(u,v)}$ componentwise; in particular,
		$\Theta(z,z)\in\R^r$. The map $p\colon\C^n\to\C^r$ is holomorphic
		and homogeneous of degree two. Comparing terms of bidegree
		$(1,1)$ in $z,\bar z$ in the graph equation gives
		$\Theta(Pz,Pz)=B\Theta(z,z)$. The remaining terms give
		$\operatorname{Im}(p(Pz)-Bp(z)-C(z))=0$.
		A holomorphic function taking only  real values is a constant
		hence constant, and  its value at $0$ is zero, therefore
		$C(z)=p(Pz)-Bp(z)$.
		The polynomial automorphism $\chi(z,w):=(z,w-p(z))$ gives
		$\chi\circ \La\circ\chi^{-1}(z,w)=(Pz,Bw)$.
		In these coordinates, $J$ agrees near $0$ with
		$\{\operatorname{Im}w=\Theta(z,z)\}$.
		Both sets are invariant under the transformed map, by
		$\Theta(Pz,Pz)=B\Theta(z,z)$, and $(P^{-m}z,B^{-m}w)\to0$
		locally uniformly as $m\to\infty$, since all eigenvalues of $P,B$
		have modulus greater than one. Applying a sufficiently large
		inverse iterate proves the equality globally.
		Finally, if $\Theta(v,v)=0$ for some $v\ne0$, then
		$t\mapsto(tv,0)$ is a nonconstant entire curve in
		$J\subset\{U=0\}$, a contradiction.

		By the argument in the proof of
		\cite[Proposition~5.7]{DS02}, the function $U$ is constant on each
		translated quadric $\{\operatorname{Im}w-\Theta(z,z)=c\}$. Thus
		$U=\Phi(\operatorname{Im}w-\Theta(z,z))$. The identity
		$U\circ \La=dU$ gives $\Phi(Bv)=d\Phi(v)$.
		
		The restriction $U(0,w)=\Phi(\operatorname{Im}w)$ is
		plurisubharmonic, so $\Phi$ is convex, and $\Phi(0)=0$.
		Put $S(v):=\Phi(v)+\Phi(-v)$. The symmetric convex set
		$\mathcal B:=\{v:S(v)\leq1\}$ is a neighborhood of $0$.
		It is bounded: otherwise, choose $v_j\in\mathcal B$ with
		$\|v_j\|\to\infty$ and $v_j/\|v_j\|\to v$, where $\|v\|=1$.
		For every $t\geq0$, convexity and continuity give
		\[
		S(tv)=\lim_{j\to\infty}S\left(\frac{t}{\|v_j\|}v_j\right)
		\leq\lim_{j\to\infty}\frac{t}{\|v_j\|}S(v_j)=0.
		\]
		Since $\Phi\geq0$, this implies $U(0,\zeta v)=0$ for every
		$\zeta\in\C$, contrary to the hypothesis on entire curves.

		Let $R:=B/d$. Convexity and $\Phi(Bv)=d\Phi(v)$ give
		$R\mathcal B\subset\mathcal B$. Every eigenvalue of $R$ has
		modulus one, so $|\det R|=1$; volume comparison therefore gives
		$R\mathcal B=\mathcal B$. Since $\mathcal B$ is a bounded
		neighborhood of zero, the closure of $\{R^m:m\in\Z\}$ is a
		compact group. Hence there are integers $m_j\to\infty$ such that
		$R^{m_j}\to I$.

		Fix $v\in\R^r$. For $0<a\leq b$, convexity gives
		\[
		\Phi(av)
		=\Phi\left(\frac ab\,bv+\left(1-\frac ab\right)0\right)
		\leq\frac ab\,\Phi(bv).
		\]
		Thus $t\mapsto\Phi(tv)/t$ is nondecreasing on $(0,\infty)$.
		Using continuity, we obtain
		\[
		d^{m_j}\Phi(d^{-m_j}v)
		=\Phi(B^{m_j}d^{-m_j}v)
		=\Phi(R^{m_j}v)\longrightarrow\Phi(v).
		\]
		For $0<t\leq1$ and all sufficiently large $j$, we have
		$d^{-m_j}\leq t\leq1$, and hence
		\[
		\Phi(R^{m_j}v)\leq\frac{\Phi(tv)}{t}\leq\Phi(v).
		\]
		Letting $j\to\infty$ gives $\Phi(tv)=t\Phi(v)$.
		For $t>1$, applying this identity to the vector $tv$ and the
		factor $1/t$ gives $\Phi(v)=t^{-1}\Phi(tv)$. The case $t=0$
		follows from $\Phi(0)=0$. Thus $\Phi$ is positively one-homogeneous.

		The equality $\Theta(P^mz,P^mz)=B^m\Theta(z,z)$ and the condition
		$\{v\in\C^n:\Theta(v,v)=0\}=\{0\}$ imply that all powers of $P/\sqrt d$ are bounded.
		By \cite[Lemma~5.10]{DS02}, a group of real (respectively complex)
		linear maps preserving a bounded spanning set preserves a
		positive-definite real (respectively Hermitian) inner product.
		The two normalized groups preserve bounded neighborhoods of zero,
		so this gives the asserted orthogonal and unitary forms.
	\end{proof}
	
	The quadratic model gives the following local affine rigidity statement.

	\begin{prop}
		\label{prop:quadric-affine-repair}
		Use the coordinates and notation of
		Proposition~\ref{prop:quadratic-model}, and suppose that
		$\operatorname{supp}(dd^cU)^\kappa=J$.
		Then there are a positive-definite Hermitian form $\zeta$ on $\mathbb C^n$
		and a positive-definite real inner product $\eta$ on $\mathbb R^r$,
		depending only on $\Phi$ and $\Theta$, with the following property.
		
		If $W\subset\mathbb C^\kappa$ is connected and open, $W\cap J\neq\varnothing$,
		and $\tau:W\to\mathbb C^\kappa$ is an open holomorphic map such that
		$U\circ\tau=U$, then $\tau$ is the restriction of an affine holomorphic
		automorphism of the form
		\begin{equation}
			\label{eq:repaired-affine-form}
			\tau(z,w)=
			\bigl(Az+a,\;Dw+2i \Theta(Az,a)+b+i \Theta(a,a)\bigr),
		\end{equation}
		where $a\in\mathbb C^n$, $b\in\mathbb R^r$,
		$A\in\operatorname{GL}_n(\mathbb C)$, and
		$D\in\operatorname{GL}_r(\mathbb R)$.
		We denote by $T_{a,b}$ the map in~\eqref{eq:repaired-affine-form}
		with $A=I_n$ and $D=I_r$. Moreover,
		$\Theta(Au,Av)=D\Theta(u,v)$, and
		\[
		\zeta(Au,Av)=\zeta(u,v)
		\quad\text{and}\quad \eta(Dc,De)=\eta(c,e).
		\]
			\end{prop}

	\begin{proof}
		This is \cite[Proposition~5.12]{DS02}.
	\end{proof}

	\subsection[The affine deck transformation group and the quotient map]{The affine deck transformation group\\and the quotient map}\label{sec:gq-global}
	We retain the cone $Z$ and the maps $u_1,u_2$ of
	Subsection~\ref{sec:ambient}, and put $\kappa:=\dim Z=k+1$.
	There are a common smooth repelling periodic point $x_0$ of $u_1,u_2$
	and an integer $s\geq1$ such that $u_1^s$ fixes this point,
	admits the quadratic model of Proposition~\ref{prop:quadratic-model},
	and satisfies the conclusions of Proposition~\ref{prop:gq-transitive} below.

	\medskip

	To choose $x_0$ and $s$, take a plaque $M$ from
	Proposition~\ref{prop:cone-lamination}, with
	$\kappa\leq m:=\dim_{\R}M<2\kappa$.
	By Proposition~\ref{prop:cone-equilibrium}(1) and
	\cite[Section~5, discussion before Lemma~5.3]{DS02}, we may shrink
	$M$ so that it is CR-generic. Proposition~\ref{prop:cone-equilibrium}(2) gives
	\[
	x_0\in M\cap\operatorname{Per}^*(u_1)\cap\operatorname{Per}^*(u_2)
	\quad\text{and}\quad u_1^s(x_0)=u_2^s(x_0)=x_0
	\]
	for some $s\geq1$. Put $F:=u_1^s$ and $D:=d_1^s$.
	By Proposition~\ref{prop:poincare}, there are a holomorphic map
	$\Psi:\C^\kappa\to Z^\times$ and an expanding $\lambda$-resonant
	polynomial automorphism $\Lambda$ such that
	\[
	F\circ\Psi=\Psi\circ\Lambda\quad\text{and}\quad\Psi(0)=x_0.
	\]
	Here $D_\lambda$ is the semisimple part of $D\Lambda(0)$.
	The map $\Psi$ is locally biholomorphic at $0$, open and locally
	finite, and $\Psi(0)$ is a repelling fixed point of $F$.

	\medskip

	Retain $U:=G^+\circ\Psi$ as in Subsection~\ref{sec:flow}, and put
	$J^*:=\Psi^{-1}(J)\subset\{U=0\}$. Local biholomorphicity at $0$ gives
	$\operatorname{supp}(dd^cU)^\kappa=J^*$ near $0$.
	Since $U\circ\Lambda=DU$, $\Lambda^{-1}(J^*)=J^*$ and
	$\Lambda^{-j}\to0$, this equality holds on all of $\C^\kappa$.
	The local density and lamination conditions follow from
	Proposition~\ref{prop:cone-lamination}; the absence of nonconstant
	entire curves in $\{U=0\}$ follows from Lemma~\ref{lem:no-entire}.
	Set $n:=m-\kappa$ and $r:=\kappa-n\geq1$.
	Lemma~\ref{lem:spectral-bounds} supplies the spectral hypotheses of
	Proposition~\ref{prop:quadratic-model}. In the resulting coordinates,
	\[
	J^*=\{\operatorname{Im}w=\Theta(z,z)\}
	\quad\text{and}\quad \Lambda(z,w)=(Pz,Bw),
	\]
	where $P/\sqrt D$ is unitary and $B/D$ is real orthogonal.
	We use $\zeta,\eta$ and $T_{a,b}$ from
	Proposition~\ref{prop:quadric-affine-repair}, applied with $J^*$ in place of $J$.

	\medskip

	Let $K$ be the group of linear maps of the form
	\eqref{eq:repaired-affine-form} with $a=b=0$ satisfying the identities
	in Proposition~\ref{prop:quadric-affine-repair}.
	It is a closed subgroup of the product of the isometry groups of
	$\zeta$ and $\eta$, hence compact.

	\medskip

	Define the affine deck transformation group and the image of $\Psi$ by
	\[
	\mathcal A:=\{\gamma\in\Aff(\C^\kappa):\Psi\circ\gamma=\Psi\}
	\quad\text{and}\quad \Omega:=\Psi(\C^\kappa).
	\]

	\medskip

	\begin{prop}\label{prop:gq-transitive}\label{prop:input-data}
		The group $\mathcal A$ acts properly discontinuously on $\C^\kappa$
		and cocompactly on $J^*$. Every fiber of $\Psi$ is an
		$\mathcal A$-orbit, $\Omega$ is open, and the induced map
		$\C^\kappa/\mathcal A\to\Omega$ is a biholomorphism.
		Moreover, every $\gamma\in\mathcal A$ has a factorization
		$\gamma=T_{a,b}\circ S$ with $S\in K$, and
		$\Lambda\mathcal A\Lambda^{-1}\subset\mathcal A$.
	\end{prop}
	\begin{proof}
		Let $\tau:W\to\C^\kappa$ be an open holomorphic local deck map,
		where $W$ is connected and open and $W\cap J^*\neq\varnothing$.
		Since $\Psi\circ\tau=\Psi$ on $W$, we have
		\[
		U\circ\tau=G^+\circ\Psi\circ\tau=G^+\circ\Psi=U.
		\]
		Together with $\operatorname{supp}(dd^cU)^\kappa=J^*$, this verifies
		the hypotheses of Proposition~\ref{prop:quadric-affine-repair},
		with $J^*$ in place of $J$.
		Hence $\tau=\gamma|_W$ for some $\gamma\in\Aff(\C^\kappa)$.
		Thus $\Psi\circ\gamma=\Psi$ on $\C^\kappa$, so $\gamma\in\mathcal A$.
		For every $\gamma\in\mathcal A$, the same proposition gives
		$\gamma=T_{a,b}\circ S$ with $S\in K$, since $U\circ\gamma=U$.

		The group $\mathcal A$ is closed in $\Aff(\C^\kappa)$;
		local injectivity of $\Psi$ at $0$ makes it discrete.
		For compact $C\subset\C^\kappa$, the affine formula and compactness
		of $K$ make
		$\{\gamma\in\mathcal A:\gamma(C)\cap C\neq\varnothing\}$
		compact, hence finite. Thus the action is properly discontinuous.

		Use the proof of \cite[Proposition~5.15]{DS02} with
		$(\varphi,f,J_k,d,k)=(\Psi,F,J,D,\kappa)$, using the local
		finiteness in Proposition~\ref{prop:poincare}, the backward
		equidistribution in Proposition~\ref{prop:cone-equilibrium}(1), and
		the quadratic model and invariant inner products in
		Propositions~\ref{prop:quadratic-model} and~\ref{prop:quadric-affine-repair}.
		Its cocompactness conclusion makes $J^*/\mathcal A$ compact, and
		its fiber-transitivity conclusion gives
		\[
		\Psi^{-1}(\Psi(x))=\mathcal A\cdot x
		\qquad(x\in\C^\kappa).
		\]

		By Proposition~\ref{prop:poincare}, $\Omega$ is open.
		For $x\in\C^\kappa$, let $\mathcal A_x$ be its finite stabilizer
		and choose a finite surjective local representative $\Psi:W\to V$,
		where $V\subset\Omega$ is open, $W$ is invariant under $\mathcal A_x$, and
		$\gamma(W)\cap W=\varnothing$ for $\gamma\notin\mathcal A_x$.
		The fiber identity makes $W/\mathcal A_x\to V$ finite and bijective,
		hence biholomorphic since $V$ is normal. These local identifications
		give $\C^\kappa/\mathcal A\simeq\Omega$.

		Finally, for $\gamma\in\mathcal A$,
		\[
		\Psi\circ\Lambda\circ\gamma\circ\Lambda^{-1}
		=F\circ\Psi\circ\gamma\circ\Lambda^{-1}
		=F\circ\Psi\circ\Lambda^{-1}=\Psi.
		\]
		Thus $\Lambda\circ\gamma\circ\Lambda^{-1}\in\mathcal A$.
	\end{proof}
	
	\section[Commuting polarized endomorphisms and exceptional maps]{Commuting polarized endomorphisms\\and exceptional maps}\label{sec:commuting maps}
	\label{sec:exceptional}

In this section, we complete the proof of Theorem \ref{thm_commuting}. Throughout, all varieties and morphisms are defined over $\C$, and we use the notation of Section \ref{sec:Dinh-Sibony}.
	
\subsection{A holomorphic torus cover}
\label{sec:holomorphic-lift}
    
We retain the notation of Proposition~\ref{prop:input-data},
including the function $U$ of Subsection~\ref{sec:flow}.
We will construct a cover $Y_0$ of the cone and then take its quotient
by the radial action to obtain a cover $Y$ of an open subset of $X$.
Proposition~\ref{prop:escape} gives a constant $C_0$ such that
\begin{equation}\label{eq:new-growth}
	\bigl|\log^+\|\Psi(z,w)\|-U(z,w)\bigr|\leq C_0.
\end{equation}

	\medskip

	A discrete subgroup $\Gamma$ of a Lie group $H$ is a
	\emph{uniform lattice} if $H/\Gamma$ is compact.

	\medskip

	For a complex Lie group $G$, a \emph{principal holomorphic $G$-bundle}
	is a holomorphic map $p:E\to B$ with a holomorphic right $G$-action
	and $G$-equivariant local trivializations $p^{-1}(V)\simeq V\times G$
	over open subsets $V\subset B$, where $G$ acts by right multiplication
	on the second factor.

	\medskip

	A \emph{holomorphic isogeny} is a surjective holomorphic group
	homomorphism between compact complex tori with finite kernel.

	\medskip

	\begin{prop}
		\label{prop:holomorphic-cover}
		With the preceding notation, the following assertions hold.
		\begin{enumerate}[(1)]
			\item Let $\mathscr N$ be the group of translations $T_{a,b}$
			from Proposition~\ref{prop:quadric-affine-repair}. Let
			$\mathscr N_c:=\{T_{0,b}:b\in\R^r\}$ and 
			$\mathscr N_h:=\mathscr N/\mathscr N_c$ be the quotient group.  We have $\mathscr N_c \cong (\R^r,+)$ and  $\mathscr N_h \cong (\C^n,+)$.
			Then the subgroup $\Delta:=\mathcal A\cap\mathscr N$ is normal of finite
			index in $\mathcal A$ and is a uniform lattice in $\mathscr N$.
			Let
			\begin{align*}
			\Sigma&:=\{b\in\R^r:T_{0,b}\in\Delta\},\\
			\intertext{and}
			\Pi&:=\{a\in\C^n:T_{a,b}\in\Delta\text{ for some }b\in\R^r\}.
			\end{align*}
			Then the quotient
			\[
			Y_0:=(\C^n\times\C^{r})/\Delta\longrightarrow A:=\C^n/\Pi
			\]
			is a holomorphic principal $T_0$-bundle, where
			$T_0:=\C^{r}/\Sigma\simeq(\C^*)^{r}$, and $\Psi$ induces
			a finite surjective holomorphic map $h_0:Y_0\to\Omega$.
			Moreover $\Lambda\Delta\Lambda^{-1}\subset\Delta$.
			
			\item There is an integer $m\geq1$ such that $B^m=D^mI$.
			Set $\ell:=sm$ and $d:=d_1^\ell=D^m$. The map $\Lambda^m$ induces a
			holomorphic endomorphism $\psi_0:Y_0\to Y_0$ over the
			holomorphic isogeny $\alpha:A\to A$ induced by $P^m$, with
			$h_0\circ\psi_0=u_1^\ell\circ h_0$ and
			$\psi_0(t\cdot y)=t^d\cdot\psi_0(y)$ for $t\in T_0$.
			
		\end{enumerate}
	\end{prop}

	\begin{proof}
		To prove (1), identify $T_{a,b}$ with its parameters $(a,b)$.
		The group law is
		\begin{align*}
			(a,b)(a',b')&=(a+a',b+b'+2\operatorname{Im}\Theta(a,a')),\\
			\intertext{and}
			[(a,b),(a',b')]&=(0,4\operatorname{Im}\Theta(a,a')).
		\end{align*}
		Thus $\mathscr N$ is simply connected, nilpotent of step at most two,
		and acts simply transitively on $J^*$. Its center is $\mathscr N_c$.
		Indeed, if $T_{a,b}$ is central, then
		$\operatorname{Im}\Theta(a,a')=0$ for every $a'$. Taking $a'=ia$, we get
		$\Theta(a,a)=0$, hence $a=0$.
		
		The group $K$ acts on $\mathscr N$ by conjugation, and
		$\mathcal A\subset \mathscr N\rtimes K$ is discrete and cocompact on
		$\mathscr N$. Choose a maximal compact subgroup $K_{\max}$ of
		$\operatorname{Aut}(\mathscr N)$ containing $K$.
		Then $\mathcal A$  is a discrete
		subgroup of $\mathscr N\rtimes K_{\max}$ whose action on
		$\mathscr N$ is cocompact. By \cite[Theorem~3.4]{Dek},
		$\Delta=\mathcal A\cap \mathscr N$ is a uniform
		lattice and $\mathcal A/\Delta$ is finite. Since $\mathscr N$ is normal
		in $\mathscr N\rtimes K$, $\Delta$ is normal in $\mathcal A$.
		
		When $n=0$, we have $\mathscr N=\mathscr N_c\simeq(\R^r,+)$,
		$\Delta\simeq\Sigma$, and $\Pi=\{0\}$. Suppose that $n>0$. Choose a left invariant metric on
		$\mathscr N$ whose value at the identity makes the coordinate tangent
		subspaces $\C^n\times\{0\}$ and $\{0\}\times\R^r$ orthogonal.
		By \cite[Proposition~5.3(1)]{Ebe94}, the intersection of a uniform
		lattice in a simply connected two-step nilpotent Lie group with its
		center is a uniform lattice in the center. Thus
		$\Delta\cap\mathscr N_c$ is a uniform lattice in $\mathscr N_c$,
		and its coordinate image $\Sigma$ is a uniform lattice in $(\R^r,+)$.
		By part~(3) of the same proposition, the image of $\Delta$ in $\mathscr N_h$
		is a uniform lattice, identified with $\Pi\subset(\C^n,+)$.
		

		If $T_{a,b}$ has a fixed point in $\C^{n+r}$, then $a=b=0$. Hence the action of $\Delta$ is free. It is also proper,
		since the action of $\mathcal A$ is proper. Taking the quotient of the
		central fibers by $\Sigma$ and then of the base by $\Pi$, we
		obtain the holomorphic principal $T_0$-bundle $Y_0\to A$.
		By Proposition~\ref{prop:gq-transitive}, the quotient of $Y_0$ by
		the finite group $\mathcal A/\Delta$ is $\Omega$. Thus $h_0$ is finite
		and surjective.
		
		Since $\Theta(Pz,Pz)=B\Theta(z,z)$, we have $\Lambda\mathscr N\Lambda^{-1}=\mathscr N$.
		It follows that
		\[
		\Lambda\Delta\Lambda^{-1}
		= \Lambda\mathcal A\Lambda^{-1}\cap \mathscr N
		\subset\mathcal A\cap \mathscr N=\Delta.
		\]
		Conjugation by $\Lambda$ acts as $B$ on $\mathscr N_c\simeq(\R^r,+)$
		and as $P$ on $\mathscr N_h\simeq(\C^n,+)$.
		Hence $B\Sigma\subset\Sigma$ and $P\Pi\subset\Pi$.
		
		For (2), choose real linear coordinates such that
		$\Sigma=\Z^{r}$ and set $E(w):=\Psi(0,w)$. Since $E$ is
		$\Sigma$-periodic, we may write the Fourier series
		\[
		E(w)=\sum_{\nu\in\Z^{r}}c_\nu e^{2\pi i\langle\nu,w\rangle},
		\qquad c_\nu\in\C^{N+1}.
		\]
		Since $U(0,u+iv)=\Phi(v)$ and $\Phi(v)\leq C_1\|v\|$,
		equation~\eqref{eq:new-growth} gives
		\[
		\log\|E(u+iv)\|\leq\log^+\|E(u+iv)\|
		\leq C_0+\Phi(v)\leq C_0+C_1\|v\|
		\]
		for $u,v\in\R^r$. Hence $\|E(u+iv)\|\leq e^{C_0+C_1\|v\|}$.
		For $t>0$, $v\in\R^r$ and $\nu\in\Z^r$, the Fourier coefficient formula gives
		\[
		c_\nu=e^{2\pi t\langle\nu,v\rangle}
		\int_{[0,1]^r}E(u+itv)e^{-2\pi i\langle\nu,u\rangle}\,du.
		\]
	Therefore
		\[
		\begin{aligned}
		\|c_\nu\|
		&\leq e^{2\pi t\langle\nu,v\rangle}
		\int_{[0,1]^r}\|E(u+itv)\|\,du\\
		&\leq\exp\bigl(C_0+tC_1\|v\|+2\pi t\langle\nu,v\rangle\bigr).
		\end{aligned}
		\]
		For $\nu\ne0$, take $v:=-\nu/\|\nu\|$. Then
		$\|c_\nu\|\leq e^{C_0+t(C_1-2\pi\|\nu\|)}$ for every $t>0$.
		Letting $t\to\infty$, we obtain $c_\nu=0$ whenever
		$2\pi\|\nu\|>C_1$. Thus the set $S$ of
		frequencies with nonzero coefficients is finite. By Parseval's identity,
		\begin{equation}\label{eq:parseval-fourier}
		1+\int_{[0,1]^{r}}\|E(u+itv)\|^2\,du
		=1+\sum_{\nu\in S}\|c_\nu\|^2e^{-4\pi t\langle\nu,v\rangle}.
		\end{equation}
		Since $U(0,u+itv)=\Phi(tv)=t\Phi(v)$, equation~\eqref{eq:new-growth}
		gives $|\log^+\|E(u+itv)\|-t\Phi(v)|\leq C_0$, uniformly in
		$u\in[0,1]^r$. For $a\geq0$, we have
		$\max\{1,a^2\}\leq1+a^2\leq2\max\{1,a^2\}$.
		Applying this with $a:=\|E(u+itv)\|$ and integrating over the unit
		cube, whose Lebesgue measure is one, we obtain
		\[
		e^{2t\Phi(v)-2C_0}
		\leq1+\int_{[0,1]^r}\|E(u+itv)\|^2\,du
		\leq2e^{2t\Phi(v)+2C_0}.
		\]
		Taking logarithms and dividing by $2t$ gives
		\[
		\Phi(v)-\frac{C_0}{t}
		\leq\frac{1}{2t}\log\left(1+\int_{[0,1]^r}\|E(u+itv)\|^2\,du\right)
		\leq\Phi(v)+\frac{2C_0+\log2}{2t}.
		\]
		Thus the logarithm of the left-hand side of~\eqref{eq:parseval-fourier},
		divided by $2t$, tends to $\Phi(v)$ as $t\to+\infty$.
		The right-hand side therefore gives
		\[
		\Phi(v)=\max\bigl(0,\max_{\nu\in S}-2\pi\langle\nu,v\rangle\bigr).
		\]
		Put $\mathcal P:=\operatorname{conv}\bigl(\{0\}\cup\{-2\pi\nu:\nu\in S\}\bigr)$,
		where $\operatorname{conv}$ denotes the convex hull.
		The vertices of $\mathcal P$ span $\R^{r}$. Otherwise, there
		would be a nonzero real vector $v$ annihilating their span. We would
		then have $\Phi(v)=\Phi(-v)=0$, so the entire line $(0,\C v)$
		would lie in $\{U=0\}$, a contradiction.
		
		Set $R:=B/D$. Since $\Phi(Bv)=D\Phi(v)$, homogeneity gives
		$\Phi(Rv)=\Phi(v)$, or equivalently $R^{\mathsf T}\mathcal P=\mathcal P$,
		where $R^{\mathsf T}$ denotes the transpose of $R$. 
		Hence $R^{\mathsf T}$ permutes the vertices of $\mathcal P$.
		A positive power fixes every vertex and is therefore the identity.
		Thus $B^m=D^mI$ for some $m\geq1$.
		
		By (1), $\Lambda^m\Delta\Lambda^{-m}\subset\Delta$, so $\Lambda^m$ induces a
		holomorphic map $\psi_0$ on $Y_0$. Similarly,
		$P^m\Pi\subset\Pi$ gives a holomorphic map $\alpha$ on $A$.
		Since $P^m$ is invertible and $\Pi$ is a uniform lattice, $\alpha$
		is a holomorphic isogeny. The central block of $\Lambda^m$ is $dI$,
		so $\psi_0$ is equivariant for the $d$th-power map on $T_0$.
		The relation $\Psi\circ\Lambda^m=u_1^\ell\circ\Psi$ descends to
		$h_0\circ\psi_0=u_1^\ell\circ h_0$.
	\end{proof}

	\medskip

	In the next proposition, we draw a conclusion for $f_1$ instead of for $u_1$.
	For a one-dimensional subtorus $S_{\rm rad}\subset T_0$ and a holomorphic
	group isomorphism $\iota:\C^*\to S_{\rm rad}$, {\em equivariance} of $h_0$ means
	\[
	h_0(\iota(\lambda)\cdot y)=\lambda h_0(y)
	\qquad(\lambda\in\C^*,\ y\in Y_0).
	\]

	\medskip

	\begin{prop}\label{prop:radial-quotient}
		Use the notation of Proposition~\ref{prop:holomorphic-cover}.
		In particular, $\alpha:A\to A$ is the holomorphic isogeny from
		part~(2), and $d=d_1^\ell=D^m>1$ is the polarized degree of $f_1^\ell$.
		Then there is a one-dimensional subtorus $S_{\rm rad}\subset T_0$
		such that $h_0$ is equivariant.  Put $s:=r-1$ and
		$T:=T_0/S_{\rm rad}\simeq(\C^*)^s$. The quotient
		$\pi:Y:=Y_0/S_{\rm rad}\to A$ is a holomorphic principal $T$-bundle,
		of dimension $k$. There are a nonempty open subset $X^\circ\subset X$
		for the complex topology, a finite surjective holomorphic map $h:Y\to X^\circ$,
		and a holomorphic endomorphism $\psi:Y\to Y$ with
		\begin{equation}\label{eq:new-lift-identities}
			\pi\circ\psi=\alpha\circ\pi,\qquad
			\psi(t\cdot y)=t^d\cdot\psi(y)\quad\text{and}\quad
			h\circ\psi=f_1^\ell\circ h.
		\end{equation}
	\end{prop}

	\begin{proof}
		For $x$ in a small neighborhood of $0$, set
		$R_t(x):=\Psi^{-1}(e^{it}\Psi(x))$ for  $t>0$ small enough.
		Thus $U\circ R_t=U$.
		By Proposition~\ref{prop:quadric-affine-repair}, each $R_t$ extends affinely
		to $\C^{n+r}$. These extensions satisfy $\Psi\circ R_t=e^{it}\Psi$ and form a local
		real-analytic one-parameter group: $R_{t+u}=R_t\circ R_u$ for small
		real $t,u$. Recall that  $u_1^\ell$ is homogeneous
		of degree $d=d_1^\ell=D^m$, that is,
		\[
		u_1^\ell(\lambda x)=\lambda^d u_1^\ell(x)
		\qquad(\lambda\in\C,\ x\in Z).
		\]
		Together with $\Psi\circ\Lambda^m=u_1^\ell\circ\Psi$ from
		Proposition~\ref{prop:holomorphic-cover}(2), this gives
		\[
		\Psi\circ\Lambda^m\circ R_t=e^{idt}\Psi\circ\Lambda^m
		=\Psi\circ R_{dt}\circ\Lambda^m.
		\]
		Local injectivity of $\Psi$ gives $\Lambda^m\circ R_t\circ\Lambda^{-m}=R_{dt}$.
		For the generator $V(x):=\left.\frac{\partial}{\partial t}R_t(x)\right|_{t=0}
		=Cx+v$, differentiation yields
		\[
		\Lambda^m C\Lambda^{-m}=dC
		\quad\text{and}\quad \Lambda^mv=dv.
		\]
		The eigenvalues of $\Lambda^m$ have modulus $\sqrt d$ and $d$.
		Thus those of $M\mapsto\Lambda^m M\Lambda^{-m}$ have modulus in
		$\{d^{-1/2},1,d^{1/2}\}$, all smaller than $d$, so $C=0$.
		Since $R_t$ preserves $J^*$ and $B^m=dI$, we have
		$v\in\ker(\Lambda^m-dI)\cap T_0J^*=\{0\}\times\R^r$. 
		Write $v=(0,b)$, then $V(x)\equiv (0,b)$ gives 
		\begin{equation}\label{eq:new-radial-identity}
		\Psi(z,w+tb)=e^{it}\Psi(z,w)\qquad(t\in\C).
		\end{equation}

	Since $\Psi$ is non-constant, $b\neq 0$. 	By definition, for $c\in\R^r$, we have $c\in\Sigma$ if and only if
		$\Psi(z,w+c)=\Psi(z,w)$ for every $(z,w)$. Thus $\C b\cap\Sigma=2\pi\Z b$.
		
		Hence 
		\[
		S_{\rm rad}:=\C b/(2\pi\Z b)\subset T_0,
		\qquad\iota(\lambda):=[-ib\log\lambda]
		\]
		defines a subtorus with $h_0(\iota(\lambda)\cdot y)=\lambda h_0(y)$.
		In particular $\C^*\Omega=\Omega$.

		Extend $2\pi b$ to a basis of $\Sigma$. The coordinate characters
		give holomorphic line bundles $E_0,\ldots,E_s$ on $A$ such that
		\begin{align*}
		Y_0&=E_0^\times\times_A E_1^\times\times_A\cdots\times_A E_s^\times,\\
		\intertext{and}
		Y:=Y_0/S_{\rm rad}&=E_1^\times\times_A\cdots\times_A E_s^\times.
		\end{align*}
		Thus $\pi:Y\to A$ is a principal $T$-bundle of dimension
		$n+r-1=k$. Put $X^\circ:=p_X(\Omega)$. Since $p_X$ is open and
		$\C^*\Omega=\Omega$, the set $X^\circ$ is open for the complex
		topology and $\Omega=p_X^{-1}(X^\circ)$.
		Equivariance gives a holomorphic finite surjection
		$h:Y\to X^\circ$, $h([y]):=p_X(h_0(y))$.


		By Proposition~\ref{prop:holomorphic-cover}(2),
		$\psi_0(\iota(\lambda)\cdot y)
		=\iota(\lambda^d)\cdot\psi_0(y)$.
		Thus $\psi_0$ descends to a holomorphic endomorphism
		$\psi:Y\to Y$, and the identities in the same proposition descend
		to~\eqref{eq:new-lift-identities}.
	\end{proof}

	\medskip

In Proposition~\ref{prop:polarization-after-algebraization}, we will show
	that $\psi$ is algebraic and that $\alpha$ admits an ample
	polarization.
	\subsection{Finite area and meromorphic extension}\label{sec:new-extension}
	For $1\leq j\leq s$, write $t\cdot_j y:=(1,\ldots,1,t,1,\ldots,1)\cdot y$,
	with $t\in\C^*$ in the $j$th coordinate of $T$.
	We measure the area of a holomorphic curve in $X$ by integrating the
	pullback of the Fubini--Study form, so that multiplicities are counted.

	\medskip

	\begin{prop}\label{prop:torus-area}
		In the setting of Proposition~\ref{prop:radial-quotient}, let $\omega_{\rm FS}$ be the Fubini--Study form on $\P^N$, normalized by
		$\int_{\P^1}\omega_{\rm FS}=1$. For every $y\in Y$ and $1\leq j\leq s$, the holomorphic curve
		$c(t):=h(t\cdot_j y)$ has finite area
		$\int_{\C^*}c^*\omega_{\rm FS}<\infty$.
		Consequently $c$ extends holomorphically to $\P^1\to X$.
	\end{prop}

	\begin{proof}
		Let $v\in\Sigma$ be the basis vector corresponding to the $j$th
		coordinate of $T$ in the proof of Proposition~\ref{prop:radial-quotient}, and let
		$(z_0,w_0)$ be a lift of $y$ through $Y_0$ to $\C^n\times \C^r$.
		Define $E:\C^*\to Z^\times\subset\C^{N+1}\setminus\{0\}$ by
		\[
		E(t):=\Psi\left(z_0,w_0+\frac{\log t}{2\pi i}v\right).
		\]
		This is a single-valued, nowhere-vanishing holomorphic map, since
		$T_{0,v}\in\Delta$. We have $c(t)=[E(t)]$. Hence
		$u(t):=\log\|E(t)\|$ is smooth and subharmonic, with
		$dd^cu=c^*\omega_{\rm FS}$.
		Put $b:=\operatorname{Im}w_0-\Theta(z_0,z_0)$, where $\Theta$ is the
		form introduced in Proposition~\ref{prop:quadratic-model}.
		Proposition~\ref{prop:escape}, where $G^+$ was defined, gives
		$\log\|x\|\leq C+G^+(x)$ for $x\in Z^\times$.
		Since $v\in\R^r$ and $U=G^+\circ \Psi$,
		the formula $U(z,w)=\Phi(\operatorname{Im}w-\Theta(z,z))$ in
		Proposition~\ref{prop:quadratic-model} gives
		\[
		\begin{aligned}
		u(t)&\leq C+G^+(E(t))\\
		&=C+U\left(z_0,w_0+\frac{\log t}{2\pi i}v\right)\\
		&=C+\Phi\left(b-\frac{\log|t|}{2\pi}v\right).
		\end{aligned}
		\]
		The function $\Phi$ is subadditive by convexity and positive homogeneity.
		It follows that the circular mean
		$m(x):=\frac1{2\pi}\int_0^{2\pi}u(e^{x+i\theta})\,d\theta$
		is convex on $\R$ and satisfies
		\[
		m(x)\leq C+\Phi(b)+
		\begin{cases}
			x\Phi(-v)/(2\pi),&x\geq0,\\
			-x\Phi(v)/(2\pi),&x\leq0.
		\end{cases}
		\]
		By convexity, $-\Phi(v)/(2\pi)\leq m'(x)\leq\Phi(-v)/(2\pi)$ for every
		$x$: otherwise a supporting line would contradict one of these bounds
		as $x\to-\infty$ or $x\to+\infty$. By Stokes' formula, we get
		\[
		\int_{\C^*}c^*\omega_{\rm FS}
		=m'(+\infty)-m'(-\infty)
		\leq\frac{\Phi(v)+\Phi(-v)}{2\pi}<\infty.
		\]
		By \cite[Proposition~(2.4), p.~32]{Gri71},
		a holomorphic map from a punctured disc to a compact Hermitian manifold
		with finite area extends meromorphically across the puncture.
		Applying this to $c$ at $0$ and $\infty$, we obtain an extension
		$\P^1\to\P^N$, which is holomorphic since a meromorphic map from a
		nonsingular curve has no indeterminacy. Its image lies in $X$, since
		$X$ is closed in $\P^N$.
	\end{proof}

	\medskip

	We write $\mathbb D:=\{z\in\C:|z|<1\}$ and
	$\mathbb D^*:=\mathbb D\setminus\{0\}$.
	We use meromorphic maps in the graph sense: the irreducible analytic
	graph is proper and bimeromorphic over the source.

	\medskip

	Write $\sO_A$ for the sheaf of holomorphic functions on $A$, and also
	for the corresponding trivial holomorphic line bundle. With the line
	bundles $E_j$ constructed in the proof of
	Proposition~\ref{prop:radial-quotient}, set
	\[
	\overline Y:=\P(\sO_A\oplus E_1)\times_A\cdots
	\times_A\P(\sO_A\oplus E_s).
	\]
	Here $\P(\sO_A\oplus E_j)$ parametrizes lines, and the embeddings
	$E_j^\times\hookrightarrow\P(\sO_A\oplus E_j)$, $v\mapsto[1:v]$,
	identify $Y$ with an open subset of $\overline Y$.
	Thus $\overline Y$ is a compactification of $Y$, and
	$\overline Y\to A$ is a holomorphic $(\P^1)^s$-bundle.

	\medskip

	\begin{prop}
		\label{prop:slice-extension}
		The map $h$ has a meromorphic extension
		$\overline h:\overline Y\dashrightarrow X$.
	\end{prop}

	\begin{proof}
		The assertion is immediate if $s=0$, and follows from
		Proposition~\ref{prop:torus-area} if $k=1$. Assume $s>0$ and $k\geq2$.
		By \cite[Section~3.2, assertion~$(*)$, p.~442]{Siu75}, a meromorphic map
		$\mathbb D^m\times\mathbb D^*\dashrightarrow\P^N$ extends meromorphically to
		$\mathbb D^{m+1}$ if its punctured-disc slices extend holomorphically for
		a set of transverse parameters of positive Lebesgue measure.
		
		Let $x\in\overline Y\setminus Y$, and let $j$ be the number of fiber
		coordinates of $x$ equal to $0$ or $\infty$, so $1\leq j\leq s$.
		In bundle coordinates near $x$, the map $h$ is defined on
		$\mathbb D^{k-j}\times(\mathbb D^*)^j$, with one punctured factor for each
		such coordinate.
		By Proposition~\ref{prop:torus-area}, every coordinate slice extends
		holomorphically. Applying the slicing extension result in the first punctured factor,
		with the remaining punctured coordinates restricted to small discs away
		from zero, we obtain a meromorphic extension to
		$\mathbb D^{k-j+1}\times(\mathbb D^*)^{j-1}$. For the next factor, the original
		slice extensions are available whenever the coordinates just filled
		are nonzero. These parameters form an open set of positive measure,
		so the same result applies again. We obtain a meromorphic
		extension across all the boundary factors and their intersections.
		The local extensions agree on $Y$ and hence glue to a meromorphic map
		on $\overline Y$.
	\end{proof}
	
	\subsection{Algebraization}
	\label{sec:new-algebraization}
	A compact complex manifold of dimension $m$ is called \emph{Moishezon} if its field of
	meromorphic functions has transcendence degree $m$.

	\medskip

	\begin{prop}\label{prop:post-extension-algebraization}
		With the notation of Propositions~\ref{prop:radial-quotient}
		and~\ref{prop:slice-extension}, the torus $A$ is an abelian variety,
		$\overline Y$ is projective, $Y\to A$ is an algebraic torus torsor,
		$X^\circ$ is Zariski open in $X$, and
		$h:Y\to X^\circ$ is a finite surjective algebraic morphism.
	\end{prop}

	\begin{proof}
		Since $\overline h:\overline Y\dashrightarrow X$ is dominant and
		$\dim\overline Y=\dim X$, the manifold $\overline Y$ is Moishezon.
		The bundle $\overline Y\to A$ is K\"ahler by
		\cite[Th\'eor\`eme principal~II]{Bla56}: the complex torus $A$ is
		K\"ahler, its fiber $(\P^1)^s$ has
		vanishing $H^1(-,\R)$, and its structure group preserves the
		product K\"ahler class. By \cite[Theorems~1.1 and~1.2]{Ji93},
		a compact K\"ahler Moishezon manifold is projective. Thus
		$\overline Y$ is projective.

		The zero sections identify $A$ with a closed analytic submanifold
		of $\overline Y$. By the GAGA principle \cite[Proposition~13]{Ser56},
		$A$ is algebraic. By the GAGA principle \cite[Proposition~15]{Ser56},
		its group operations are algebraic; hence $A$ is an abelian variety.

		By the GAGA principle \cite[Proposition~18, with $n=1$]{Ser56}, each $E_j$ is
		algebraic, and therefore so is the torus torsor $Y\to A$.
		By the GAGA principle \cite[Proposition~15]{Ser56}, its projective
		bundle compactification is algebraically identified with $\overline Y$.
		In particular, $Y$ is Zariski open in $\overline Y$.
		By the GAGA principle \cite[Proposition~13]{Ser56}, the closed
		graph of $\overline h$ is algebraic. Its restriction to $Y\times X$
		is the graph of $h$. By the GAGA principle \cite[Proposition~8]{Ser56},
		$h$ is algebraic.
		Its image $X^\circ=h(Y)$ is constructible
		\cite[Th\'eor\`eme~1.8.4]{EGAIV1} and open for the complex topology,
		hence Zariski open by
		\cite[Expos\'e~XII, Corollaire~2.3]{SGA1}.
		Finally, $h:Y\to X^\circ$ is finite and surjective analytically,
		hence algebraically by
		\cite[Expos\'e~XII, Proposition~3.2(i), (vi)]{SGA1}.
	\end{proof}

	\medskip

	\begin{prop}
		\label{prop:polarization-after-algebraization}
		With the notation of
		Propositions~\ref{prop:holomorphic-cover}--\ref{prop:post-extension-algebraization},
		the isogeny $\alpha$ admits an ample line bundle $\mathcal H$ satisfying
		$\alpha^*\mathcal H\simeq\mathcal H^{\otimes d}$.
		The map $\psi$ is algebraic and is therefore a $d$-lift of $\alpha$
		in the algebraic sense.
	\end{prop}

	\begin{proof}
		By the GAGA principle \cite[Proposition~15]{Ser56},
		the holomorphic isogeny $\alpha$ is algebraic.
		Suppose that $\dim A>0$.
		
		Put $V:=N^1(A)_\R$ and $S:=d^{-1}\alpha^*|_V$.
		The unitary normalization of $P^m$ implies that
		$\{S^j:j\in\Z\}$ is bounded, and $S^{\pm1}$ preserve
		$\operatorname{Nef}(A)$, the closure of the ample cone.
		By \cite[Proposition~2.9]{MZ18}, an invertible real linear
		map preserving a closed convex pointed spanning cone in both directions has
		an interior fixed vector whenever its positive and negative powers
		are bounded. Thus $S$ fixes an ample real class.
		Since $d$ is an integer, \cite[Lemma~3.5]{MZ18} gives an ample
		line bundle $\mathcal H_1$ with
		$\alpha^*\mathcal H_1\equiv\mathcal H_1^{\otimes d}$.
		Finally, \cite[Lemma~3.4]{MZ18} gives an ample line bundle
		$\mathcal H\equiv\mathcal H_1$ satisfying
		$\alpha^*\mathcal H\simeq\mathcal H^{\otimes d}$.
		
		The identity
		$\psi(t\cdot y)=[d](t)\cdot\psi(y)$ implies that the $j$th component
		of $\psi$ is a nowhere-zero homogeneous map of degree $d$ on the
		fibers. It therefore defines a holomorphic line-bundle isomorphism
		$E_j^{\otimes d}\simeq\alpha^*E_j$.
		By the GAGA principle \cite[Theorem~2]{Ser56}, these
		isomorphisms and their inverses are algebraic. They define an algebraic endomorphism of $Y$
		whose analytification is $\psi$. The identities
		$\pi\circ\psi=\alpha\circ\pi$ and $\psi(t\cdot y)=[d](t)\cdot\psi(y)$ then
		hold algebraically, so $\psi$ is a $d$-lift.
	\end{proof}

	\medskip

\begin{theorem}\label{thm_commuting}
Let $\mathbf{k}$ be an algebraically closed field of characteristic zero. Let $X$ be a normal projective variety over $\mathbf{k}$, and let
$f,g:X\to X$ be polarized endomorphisms with polarization degrees
$d_f$ and $d_g$, respectively. Assume that $f\circ g=g\circ f$ and $d_f^n\neq d_g^m$
for all positive integers $n,m$. Then $f$ and $g$ are exceptional maps.
\end{theorem}
\begin{proof}
By a standard spreading-out argument and the Lefschetz principle, it suffices to prove the statement over $\mathbf{k}=\mathbb{C}$. The case $\dim X=0$ follows from Subsection~\ref{sec:common-polarization}. Assume $\dim X>0$. By Propositions~\ref{prop:holomorphic-cover} and~\ref{prop:radial-quotient}, there are $Y$, $h$, $\alpha$, $\psi$ and an integer $\ell\geq1$ satisfying~\eqref{eq:new-lift-identities}, with $d=d_1^\ell$. By Proposition~\ref{prop:torus-area}, $h$ extends holomorphically along each coordinate orbit. Proposition~\ref{prop:slice-extension} then gives a meromorphic extension to $\overline Y$. By Proposition~\ref{prop:post-extension-algebraization}, $A$ is an abelian variety, $Y\to A$ is an algebraic torus torsor, $X^\circ$ is a nonempty Zariski open subset of $X$, and $h:Y\to X^\circ$ is a finite surjective algebraic morphism. By Proposition~\ref{prop:polarization-after-algebraization}, the isogeny $\alpha$ has a polarization of polarized degree $d$, and $\psi$ is an algebraic $d$-lift. The identity $h\circ\psi=f_1^\ell\circ h$ proves the theorem.
	\end{proof}

\section{Endomorphisms comes from abelian sequence}\label{sec_main}

\subsection{Tubes in Berkovich spaces}\label{sec_tube}

Let $\mathbf{k}$ be a complete non-archimedean field equipped
with a nontrivial absolute value $|\cdot|$. Set
\[
\mathbf{k}^{\circ}:=\{x\in\mathbf{k}:|x|\leq 1\},
\qquad
\mathbf{k}^{\circ\circ}:=\{x\in\mathbf{k}:|x|<1\},
\qquad
\wt{\mathbf{k}}:=\mathbf{k}^{\circ}/\mathbf{k}^{\circ\circ}.
\]
Assume that $\mr{char}(\wt{\mathbf{k}})=p>0$, and let $q=p^a$
for some integer $a\geq 1$. We define tubes in
$(\mb{P}^N_{\mathbf{k}}\times\mb{P}^N_{\mathbf{k}})^{\an}$
as follows.

For $0\leq i\leq N$, consider the affine chart
\[
\begin{aligned}
U_i
&:=\left\{
([x_0:\dots:x_N],[y_0:\dots:y_N])
\mid x_i\neq 0,\ y_i\neq 0
\right\}\\
&=\Spe\mathbf{k}
[s_{j,i},t_{j,i}\mid 0\leq j\leq N,\ j\neq i],
\end{aligned}
\]
where $s_{j,i}=x_j/x_i$ and $t_{j,i}=y_j/y_i$.
For convenience, set $s_{i,i}=t_{i,i}=1$.
Let
\[
E_i:=
\left\{
\xi\in U_i^{\an}
\;\middle|\;
|s_{j,i}|_\xi\leq 1,\ |t_{j,i}|_\xi\leq 1
\text{ for all }j\neq i
\right\}
\]
be the closed unit polydisc in $U_i^{\an}$.
For $\delta\in[0,1)$ and $r\in\mb{Z}_{\geq 0}$, define
\[
T_i^{(r)}(\delta):=
\left\{
\xi\in E_i
\;\middle|\;
|t_{j,i}-s_{j,i}^{q^r}|_\xi\leq\delta
\text{ for all }j\neq i
\right\},
\]
and set
\[
T^{(r)}(\delta):=
\bigcup_{i=0}^N T_i^{(r)}(\delta)
\subset
(\mb{P}^N_{\mathbf{k}}\times\mb{P}^N_{\mathbf{k}})^{\an}.
\]

\begin{lemma}\label{lem_tube}
For every $r\in\mb{Z}_{\geq 0}$ and $\delta\in[0,1)$,
the tube $T^{(r)}(\delta)$ is compact and closed in
$(\mb{P}^N_{\mathbf{k}}\times\mb{P}^N_{\mathbf{k}})^{\an}$.
Moreover, $T^{(r)}(\delta)\cap E_i=T_i^{(r)}(\delta)$ for all $0\leq i\leq N$.
\end{lemma}

\begin{proof}
Each $E_i$ is compact by \cite[Theorem~1.2.1]{Ber90},
and the defining inequalities show that $T_i^{(r)}(\delta)$
is closed in $E_i$. Thus $T^{(r)}(\delta)$ is compact,
and hence closed in the Hausdorff space
$(\mb{P}^N_{\mathbf{k}}\times\mb{P}^N_{\mathbf{k}})^{\an}$.

It remains to check that the local definitions agree on overlaps.
Let $\xi\in T_i^{(r)}(\delta)\cap E_j$, with $i\neq j$. Since $\xi\in E_i\cap E_j$, we have $|s_{j,i}|_\xi=|t_{j,i}|_\xi=1$. Writing $e_\ell=t_{\ell,i}-s_{\ell,i}^{q^r}$, then $e_i=0$ and $|e_{\ell}|_{\xi}\leq 1$ for all $\ell\neq i$. Thus, we obtain, for every $\ell\neq j$,
\[
t_{\ell,j}-s_{\ell,j}^{q^r}
=
\frac{s_{j,i}^{q^r}e_\ell-s_{\ell,i}^{q^r}e_j}
     {t_{j,i}s_{j,i}^{q^r}}.
\]
The denominator has absolute value $1$ at $\xi$, whereas
the numerator has absolute value at most $\delta$.
Thus $\xi\in T_j^{(r)}(\delta)$.
This proves the required compatibility and hence the assertion.
\end{proof}

\begin{lemma}\label{lem_comp}
Fix $\delta\in[0,1)$, and let
$x,y,z\in\mb{P}^N(\ov{\mathbf{k}})$. Then the following hold.
\begin{enumerate}
\renewcommand{\labelenumi}{(\theenumi)}
\item
For $r_1,r_2\in\mb{Z}_{\geq 0}$, if
$(x,y)\in T^{(r_1)}(\delta)$ and
$(y,z)\in T^{(r_2)}(\delta)$, then
\[
(x,z)\in T^{(r_1+r_2)}(\delta).
\]

\item
For integers $0\leq r_1\leq r_2$, if
$(y,x)\in T^{(r_1)}(\delta)$ and
$(y,z)\in T^{(r_2)}(\delta)$, then
\[
(x,z)\in T^{(r_2-r_1)}(\delta).
\]
\end{enumerate}
\end{lemma}

\begin{proof}
We first choose a common coordinate chart.
For a pair $(u,v)\in T^{(r)}(\delta)$, take homogeneous
coordinates with $\max_j|u_j|=\max_j|v_j|=1$. Choose $i$ such that $(u,v)\in T_{i}^{(r)}(\delta)$. Then $|u_i|=|v_i|=1$, and
\[
\left|
\frac{v_j}{v_i}
-
\left(\frac{u_j}{u_i}\right)^{q^r}
\right|
\leq\delta<1
\]
for every $j$. Consequently, $|u_j|=1$ if and only if $|v_j|=1$. Thus, in either case of the lemma, the same coordinate indices
attain the maximal absolute value for $x,y,z$. We may therefore choose homogeneous coordinates such that all coordinates have absolute value at most $1$ and
$x_i=y_i=z_i=1$ for a common index $i$. 

In case~(1), for every $j\neq i$,
\[
\begin{aligned}
|z_j-x_j^{q^{r_1+r_2}}|
&\leq
\max\left\{
|z_j-y_j^{q^{r_2}}|,
|y_j^{q^{r_2}}-x_j^{q^{r_1+r_2}}|
\right\}\\
&\leq
\max\left\{
|z_j-y_j^{q^{r_2}}|,
|y_j-x_j^{q^{r_1}}|
\right\}
\leq\delta.
\end{aligned}
\]
Hence $(x,z)\in T_i^{(r_1+r_2)}(\delta)$.

In case~(2), for every $j\neq i$,
\[
\begin{aligned}
|z_j-x_j^{q^{r_2-r_1}}|
&\leq
\max\left\{
|z_j-y_j^{q^{r_2}}|,
|y_j^{q^{r_2}}-x_j^{q^{r_2-r_1}}|
\right\}\\
&\leq
\max\left\{
|z_j-y_j^{q^{r_2}}|,
|y_j^{q^{r_1}}-x_j|
\right\}
\leq\delta.
\end{aligned}
\]
Thus $(x,z)\in T_i^{(r_2-r_1)}(\delta)$.
\end{proof}

Consider the standard model
\[
\mc{P}:=
\mb{P}^N_{\mathbf{k}^{\circ}}
\times_{\mathbf{k}^{\circ}}
\mb{P}^N_{\mathbf{k}^{\circ}}
\longrightarrow\Spe\mathbf{k}^{\circ}.
\]
Its special fiber is
$\mb{P}^N_{\wt{\mathbf{k}}}
\times_{\wt{\mathbf{k}}}
\mb{P}^N_{\wt{\mathbf{k}}}$.
As recalled in Section~\ref{sec_pre}, this model determines
a reduction map
\[
\red:
(\mb{P}^N_{\mathbf{k}}\times\mb{P}^N_{\mathbf{k}})^{\an}
\longrightarrow
\mb{P}^N_{\wt{\mathbf{k}}}
\times_{\wt{\mathbf{k}}}
\mb{P}^N_{\wt{\mathbf{k}}}.
\]
Let $\Phi_q$ be the Frobenius defined by 
\[
\Phi_q:\mb{P}^N_{\wt{\mathbf{k}}}
\longrightarrow\mb{P}^N_{\wt{\mathbf{k}}},
\qquad
[x_0:\dots:x_N]\longmapsto[x_0^q:\dots:x_N^q].
\]
For $r\geq 0$, write $\Phi_q^r$ for its $r$-th iterate
and $\Gamma_{\Phi_q^r}$ for the graph of $\Phi_q^r$.
In particular, $\Gamma_{\Phi_q^0}$ is the diagonal.

\begin{lemma}\label{lem_red}
For every $r\in\mb{Z}_{\geq 0}$ and $\delta\in[0,1)$, $\red\bigl(T^{(r)}(\delta)\bigr)
\subseteq\Gamma_{\Phi_q^r}$.
\end{lemma}

\begin{proof}
Let $\xi\in T_i^{(r)}(\delta)$.
Since $\xi\in E_i$, its reduction lies in the affine chart
\[
\wt{U_i}:=
\Spe\wt{\mathbf{k}}
[s_{j,i},t_{j,i}\mid 0\leq j\leq N,\ j\neq i]
\]
of the special fiber.
For every $j\neq i$, the inequality
\[
|t_{j,i}-s_{j,i}^{q^r}|_\xi\leq\delta<1
\]
implies that the reduction of $t_{j,i}-s_{j,i}^{q^r}$
vanishes at $\red(\xi)$.
These equations define $\Gamma_{\Phi_q^r}$ on $\wt{U_i}$.
Hence $\red(\xi)\in\Gamma_{\Phi_q^r}$.
\end{proof}

\begin{corollary}\label{cor_shilovtube}
Let $V\subset\mb{P}^N_{\mathbf{k}}\times\mb{P}^N_{\mathbf{k}}$
be a subvariety, let $\mc{V}$ be its schematic closure in
$\mc{P}$, and let $\wt{V}$ be the special fiber of $\mc{V}$.
Fix $\delta\in[0,1)$ and $r\in\mb{Z}_{\geq 0}$.
Assume that, for every generic point $\eta$ of an irreducible
component of $\wt{V}$, there exists a point $\rho_\eta\in V^{\an}\cap T^{(r)}(\delta)$ with $\red(\rho_\eta)=\eta$. Then $\Supp(\wt{V})\subseteq\Gamma_{\Phi_q^r}$.
\end{corollary}

\begin{proof}
By Lemma~\ref{lem_red}, every such generic point $\eta$
belongs to $\Gamma_{\Phi_q^r}$.
Since $\Gamma_{\Phi_q^r}$ is closed, it contains the closure
of each $\eta$ and hence the support of $\wt{V}$. The asserted inclusion follows after passing to the reduced induced scheme structure.
\end{proof}

\subsubsection*{Tubes for general projective varieties}

Let $X$ be a projective variety over $\mathbf{k}$, and fix
a closed embedding $X\hookrightarrow\mb{P}^N_{\mathbf{k}}$.
It induces a closed embedding
\[
(X\times X)^{\an}
\hookrightarrow
(\mb{P}^N_{\mathbf{k}}\times\mb{P}^N_{\mathbf{k}})^{\an}.
\]
For $r\in\mb{Z}_{\geq 0}$ and $\delta\in[0,1)$, define
\[
T_X^{(r)}(\delta):=
T^{(r)}(\delta)\cap(X\times X)^{\an}.
\]
By Lemma~\ref{lem_tube}, this is a compact closed subset
of $(X\times X)^{\an}$.

Let $\mc{X}$ be the schematic closure of $X$ in
$\mb{P}^N_{\mathbf{k}^{\circ}}$, and let $\wt{X}$
be its special fiber. Then
\[
\mc{X}\times_{\mathbf{k}^{\circ}}\mc{X}
\longrightarrow\Spe\mathbf{k}^{\circ}
\]
is a model of $X\times X$, with special fiber
$\wt{X}\times_{\wt{\mathbf{k}}}\wt{X}$.

Assume that the closed immersion
$\wt{X}\hookrightarrow\mb{P}^N_{\wt{\mathbf{k}}}$
is defined over a finite subfield
$\mb{F}_q\subset\wt{\mathbf{k}}$.
Then $\Phi_q$ restricts to an endomorphism $\Phi_{q,\wt{X}}:\wt{X}\longrightarrow\wt{X}$. Its graph satisfies
\[
\Gamma_{\Phi_{q,\wt{X}}}
=
\Gamma_{\Phi_q}
\cap
\bigl(\wt{X}\times_{\wt{\mathbf{k}}}\wt{X}\bigr).
\]

\subsection{Construction of bi-finite correspondences}

Throughout this subsection, let $K$ be a number field and let
$X$ be a normal projective variety over $K$ of dimension $n\geq 1$.
Fix a finite place $v$ of $K$, and write $k_v$ for its residue
field. Set $p:=\mr{char}(k_v)$ and $q:=|k_v|$.

Let $\ov{L}$ be a nef adelic line bundle on $X$ with $L$ ample.
Assume that $v$ is a model place for $\ov{L}$. More precisely, after replacing $\ov{L}$ by a positive tensor
power, fix an arithmetic model $(\mc{X},\mc{L})$ of $(X,L)$ over $\Spe O_K$ and an open neighborhood $\mc{V}\subset\Spe O_K$ of $v$ such that the metrics of
$\ov{L}$ over $\mc{V}$ are induced by
$(\mc{X}_{\mc{V}},\mc{L}_{\mc{V}})$ and
$\mc{L}|_{\mc{X}_v}$ is ample.
Fix a closed embedding $\mc{X}\hookrightarrow\mb{P}^N_{O_K}$,
and let $\wt{X}:=\mc{X}_v$.
Write $\pi_i:X\times X\to X$ for the two projections.

Fix an embedding $\ov{K}\hookrightarrow\mb{C}_v$, and let $w$
be the resulting extension of $v$ to $\ov{K}$. Let $D_w\subset \Gal(\ov{K}/K)$ be the
decomposition group at $w$. Choose $\sigma\in D_w$ lifting
the Frobenius automorphism
\[
\phi_q:\ov{k_v}\longrightarrow\ov{k_v},
\qquad a\longmapsto a^q.
\]
We use the tubes defined in Section~\ref{sec_tube} for the
chosen projective embedding. All Zariski closures and irreducible components below are
taken over $K$, unless otherwise specified.

\begin{lemma}\label{lem_deltachoice}
Let $(x_m)_{m\geq 1}$ be a sequence in $X(\ov{K})$ that is
almost unramified at $v$. Then there exists $\delta\in(0,1)$
such that, for every $m\geq 1$,
\[
\sup_{\substack{z\in K(x_m)\\ |z|_w<1}}|z|_w\leq\delta.
\]
\end{lemma}

\begin{proof}
Let $w_m:=w|_{K(x_m)}$.
By assumption, there is an integer $e\geq 1$ such that
$e(w_m/v)\leq e$ for every $m$.
Choose a uniformizer $\varpi$ of $K_v$ and set
\[
\delta:=|\varpi|_v^{1/e}\in(0,1).
\]
The absolute value on $K(x_m)_{w_m}$ is understood to extend
that on $K_v$. If $\varpi_m$ is a uniformizer of this field,
then
\[
|\varpi_m|_w=|\varpi|_v^{1/e(w_m/v)}.
\]
Hence every $z\in K(x_m)$ with $|z|_w<1$ satisfies
\[
|z|_w\leq|\varpi_m|_w
=|\varpi|_v^{1/e(w_m/v)}
\leq|\varpi|_v^{1/e}
=\delta.
\]
\end{proof}

\begin{theorem}\label{thm_constructZ}
With the notation and model hypotheses above, assume that
there exists a generic sequence $(x_m)_{m\geq 1}$ in
$X(\ov{K})$ that is abelian, almost unramified at $v$, and
satisfies
\[
\lim_{m\to\infty}h_{\ov{L}}(x_m)=0.
\]
Then there exists a correspondence $\mf{C}_X\subset X\times X$
with the following properties:
\begin{enumerate}
\renewcommand{\labelenumi}{(\theenumi)}
\item
Every irreducible component of $\mf{C}_X$ has dimension $n$.
\item
If $\wt{\mf{C}_X}$ denotes the special fiber at $v$ of the
schematic closure of $\mf{C}_X$ in
$\mc{X}\times_{O_K}\mc{X}$, then
\[
\Supp(\wt{\mf{C}_X})
=
\Supp(\Gamma_{\Phi_{q,\wt{X}}}).
\]
\item
The projections $\pi_i|_{\mf{C}_X}:\mf{C}_X\to X$ are finite
and surjective for $i=1,2$. In particular, $\mf{C}_X$ is
bi-finite.
\end{enumerate}
\end{theorem}

\begin{proof}
Set
\[
B:=\{(x_m,\sigma(x_m))\mid m\geq 1\},
\]
and let $V\subset X\times X$ be the reduced Zariski closure
of $\Gal(\ov{K}/K)\cdot B$.

Every positive-dimensional irreducible component $W$ of $V$
dominates both factors. Indeed, the points of $\Gal(\ov{K}/K)\cdot B$
lying on $W$ are Zariski dense in $W$, so infinitely many
indices $m$ occur. If either projection of $W$ were contained
in a proper closed subset $D\subsetneq X$, then, since $D$
is defined over $K$, infinitely many $x_m$ would belong to
$D$. This contradicts genericity.
In particular, $\dim W\geq n$.

Let $\mf{C}_X$ be the union, with reduced structure, of the
positive-dimensional irreducible components of $V$.
Choose $\delta\in(0,1)$ as in Lemma~\ref{lem_deltachoice}.

\medskip
\noindent\textbf{Claim 1.}
$\Gal(\ov{K}/K)\cdot B\subset T_X^{(1)}(\delta)$.

\smallskip

\noindent\textit{Proof of Claim 1.} For $\tau\in \Gal(\ov{K}/K)$, abelianity of $K(x_m)/K$ gives
\[
\tau(x_m,\sigma(x_m))
=
(\tau(x_m),\sigma(\tau(x_m))).
\]
Since $K(x_m)/K$ is Galois, $K(\tau(x_m))=K(x_m)$. Choose homogeneous coordinates $\tau(x_m)=[z_0:\dots:z_N]$ with $z_j\in K(x_m)$, $|z_j|_w\leq 1$ for all $j$, and $z_i=1$
for some $i$.
As $\sigma\in D_w$, it preserves $|\cdot|_w$, so
$(\tau(x_m),\sigma(\tau(x_m)))\in E_i$.
Moreover, $\overline{\sigma(z_j)}=\overline{z_j}^{\,q}$ in $\ov{k_v}$, and consequently
$|\sigma(z_j)-z_j^q|_w<1$.
The difference belongs to $K(x_m)$, so
Lemma~\ref{lem_deltachoice} gives $|\sigma(z_j)-z_j^q|_w\leq\delta$ for every $j\neq i$. This proves Claim~1. \bqed

\medskip
\noindent\textbf{Claim 2.}
For every irreducible component $W$ of $\mf{C}_X$, the
special fiber $\wt{W}$ of its schematic closure
$\mc{W}\subset\mc{X}\times_{O_K}\mc{X}$ satisfies
\[
\Supp(\wt{W})=\Supp(\Gamma_{\Phi_{q,\wt{X}}}).
\]
\noindent\textit{Proof of Claim 2.}  Set $\ov{H}:=\pi_1^*\ov{L}+ \pi_2^*\ov{L}\in \widehat{\Pic}(X\times X)_{\mr{nef}}$. The adelic line bundle $\ov{H}$ is nef and its underlying
line bundle is ample. Galois invariance of heights gives
\[\lim_{m\to\infty}h_{\ov{H}}\big((x_m,\sigma(x_m))\big)=\lim_{m\to\infty}h_{\ov{L}}(x_m)+\lim_{m\to\infty}h_{\ov{L}}(\sigma(x_m))=0.\]
We may therefore choose a generic sequence
$(\xi_m)_{m\geq 1}$ in $(\Gal(\ov{K}/K)\cdot B)\cap W(\ov{K})$
such that $\lim\limits_{m\to\infty}h_{\ov{H}}(\xi_m)=0$. Nefness and Zhang's inequality imply
\[
\mr{ess}(W,\ov{H}|_W)=h_{\ov{H}|_W}(W)=0.
\]

For a geometric point $\xi$, write $\xi_v$ for its image
in $W_{\mb{C}_v}^{\an}$.
By the equidistribution theorem
\cite[Theorem~3.1]{Yua08}, the measures
\[
\mu_{\xi_m,v}:=
\frac{1}{[K(\xi_m):K]}
\sum_{\xi\in G_K\cdot\xi_m}\delta_{\xi_v}
\]
converge weakly to
\[
\mu_{\ov{H}|_W,v}:=
\frac{1}{\deg_H(W)}c_1(\ov{H}|_W)_v^{\dim W}.
\]
Claim~1 and the closedness of $T_X^{(1)}(\delta)$ show that
this limiting measure is supported on
$W_{\mb{C}_v}^{\an}\cap T_X^{(1)}(\delta)$.

Let $\mc{W}$ be the Zariski closure of $W$ in $\mc{X}\times_{O_K}\mc{X}$ and $\wt{W}$ its special fiber at $v$. Each irreducible component $\eta$ of $\wt{W}$ corresponds to a Shilov point $\rho_{\eta}\in W_{\mb{C}_v}^{\an}$. Since $v$ is a model place with respect to $\ov{L}$, it is also a model place with respect to $\ov{H}|_W$ by definition.  By Lemma \ref{lem_measure}, the Chambert-Loir measure is given by
\[
c_1(\ov{H}|_W)_v^{\dim W}=\sum_{\eta}a_{\eta}\delta_{\rho_{\eta}},
\]
where $\eta$ runs over the irreducible components of $\wt{W}$ and $a_{\eta}>0$ are positive constants. In particular, this implies that $\rho_{\eta}\in T_X^{(1)}(\delta)$ for all $\eta$. By Corollary \ref{cor_shilovtube}, we obtain
\[
\wt{W}\subset \Gamma_{\Phi_q}\cap (\wt{X}\times \wt{X})=\Gamma_{\Phi_{q,\wt{X}}}.
\] 
Since $\wt{W}$ is equidimensional and
\[
\dim \wt{W}=\dim W \geq \dim X=\dim \Gamma_{\Phi_{q,\wt{X}}},
\]
it follows that $\wt{W}$ has a unique irreducible component, namely $\eta=\Gamma_{\Phi_{q,\wt{X}}}$. Consequently, $\Supp(\wt{\mf{C}_X})$ is equal to $\Gamma_{\Phi_{q,\wt{X}}}$. (Note that $\wt{\mf{C}_X}$ may be non-reduced, but it has only one irreducible component.)\bqed

\medskip

Now it remains to prove (1) and (3). Let $W$ be an irreducible component of $\mf{C}_X$ and $\pi_{i,W}: W\to X$ be projections to $i$-th factor for $i=1,2$. Denote by $\Pi_i: \mc{X}\times_{O_K}\mc{X}\to \mc{X}$ the projection onto the $i$-th factor. Since $\Supp(\wt{W})=\Gamma_{\Phi_{q,\wt{X}}}$, the morphism $\Pi_i|_{\mc{W}}:\mc{W}\to \mc{X}$ is finite on the special fiber $\wt{W}$ over $v$. By upper semicontinuity of fiber dimension, it follows that $\Pi_i|_{\mc{W}}$ is quasi-finite on the generic fiber. In other words, $\pi_i|_W: W\to X$ is quasi-finite, and hence finite by properness. Therefore $\dim W=n$. Applying this to every component of $\mf{C}_X$ proves the theorem.
\end{proof}

\begin{corollary}\label{cor_Ztube}
With the notation as above, the Berkovich analytification
\[
(\mf{C}_X)_{\mb{C}_v}^{\an}\subset (X\times X)_{\mb{C}_v}^{\an}
\]
is contained in $T_X^{(1)}(\delta)$ for some $\delta\in (0,1)$.
\end{corollary}
\begin{proof}
For each $i$, set
\[
F_i:=(\mf{C}_X)_{\mb{C}_v}^{\an}\cap E_i.
\]
We claim that the sets $F_i$ cover
$(\mf{C}_X)_{\mb{C}_v}^{\an}$. Each $F_i$ is closed. Since rigid points are dense, it suffices
to show that every rigid point belongs to some $F_i$.

Let $w$ denote the valuation on $\mb{C}_v$ extending $v$. Let $W$ be an irreducible component of $(\mf{C}_X)_{\mb{C}_v}$, and let $\xi=(y,z)\in W(\mb{C}_v)$ be a rigid point.
Write $y=[y_0:\dots:y_N]$ and $z=[z_0:\dots:z_N]$ with
\[
\max_{0\leq j\leq N}|y_j|_w
=
\max_{0\leq j\leq N}|z_j|_w
=1.
\]
By Theorem~\ref{thm_constructZ}, every point of
$(\mf{C}_X)_{\mb{C}_v}^{\an}$ reduces to the Frobenius graph.
In particular, $\red(\xi)\in\Gamma_{\Phi_q}(\ov{k_v})$. So
\[
[\overline{z_0}:\dots:\overline{z_N}]
=
[\overline{y_0}^{\,q}:\dots:\overline{y_N}^{\,q}].
\]
Thus there exists $i\in\{0,\dots,N\}$ such that
$|y_i|_w=|z_i|_w=1$. Hence $\xi\in E_i$, and therefore $\xi\in F_i$. This proves the claim.

Using the notation of Section~\ref{sec_tube}, since the reduction of $\mathfrak{C}_X$ is the graph of Frobenius, we have
\[
|t_{j,i}-s_{j,i}^q|_\xi<1
\]
for every $\xi\in F_i$ and $j\neq i$. By compactness of $F_i$, each continuous function $\xi\mapsto |t_{j,i}-s_{j,i}^q|_\xi$ attains a maximum strictly less than $1$ whenever $F_i$ is nonempty. We may therefore choose $\delta\in(0,1)$ such that $(\mf{C}_X)_{\mb{C}_v}^{\an}\subset T_X^{(1)}(\delta)$, as required.
\end{proof}

\subsubsection*{Construction of an equivalence relation $\mf{R}_X$}

We now use $\mf{C}_X$ to construct an equivalence relation
$\mf{R}_X\subset X\times X$ in the sense of
Definition~\ref{def_relation}.
Set $\mf{C}_X^1:=\mf{C}_X$ and
$\mf{C}_X^{-1}:=\trans\mf{C}_X$, and define
\[
\Omega:=
\left\{
\mf{C}_X^{\varepsilon_m}\circ\cdots\circ
\mf{C}_X^{\varepsilon_1}
\;\middle|\;
\begin{array}{l}
m\geq 0,\quad \varepsilon_i\in\{1,-1\},\quad
\displaystyle\sum_{i=1}^m\varepsilon_i=0,\\
\displaystyle\sum_{i=1}^j\varepsilon_i\leq 0
\quad(1\leq j\leq m)
\end{array}
\right\}
\subset\mc{C}(X,X).
\]
For $m=0$, the correspondence is the diagonal $\Delta_X$.

\begin{theorem}\label{thm_relation}
Let
\[
\mf{R}_X:=
\overline{\bigcup_{\gamma\in\Omega}\Supp(\gamma)}
^{\mathrm{Zar}}
\subset X\times X,
\]
with its reduced induced structure.
Then $\mf{R}_X$ is a bi-finite correspondence in
$\mc{C}(X,X)$ satisfying:
\begin{enumerate}
\renewcommand{\labelenumi}{(\theenumi)}
\item $\Delta_X\subseteq\mf{R}_X$;
\item $\mf{R}_X=\trans\mf{R}_X$;
\item
$\Supp(\mf{R}_X\circ\mf{R}_X)=\Supp(\mf{R}_X)$;
\item
$\Supp(\mf{C}_X\circ\mf{R}_X\circ\trans\mf{C}_X)
\subseteq\Supp(\mf{R}_X)$.
\end{enumerate}
\end{theorem}

We first establish the following auxiliary results.

\begin{lemma}\label{lem_gammaess}
Let $\gamma\in\Omega$, and let $W$ be an irreducible
component of $\gamma$. Then
\[
\mr{ess}(W,\pi_1^*\ov{L}+\pi_2^*\ov{L})
=
h_{\pi_1^*\ov{L}+\pi_2^*\ov{L}}(W)=0.
\]
\end{lemma}

The proof of Lemma~\ref{lem_gammaess} is given at the end
of this subsection.

\begin{proposition}\label{prop_gamma}
There exists $\delta\in(0,1)$ such that, for every
$\gamma\in\Omega$ and every irreducible component $W$
of $\gamma$, $W_{\mb{C}_v}^{\an}\subset T_X^{(0)}(\delta)$.
\end{proposition}

\begin{proof}
Choose $\delta\in(0,1)$ as in Corollary~\ref{cor_Ztube}.
Write
\[
\gamma=\mf{C}_X^{\varepsilon_m}\circ\cdots\circ
\mf{C}_X^{\varepsilon_1}.
\]
Since rigid points are dense and the tube is closed,
it suffices to prove the assertion for rigid points. We argue by induction on $m$.
The case $m=0$ is immediate, since $\gamma=\Delta_X$.

Suppose first that
$\sum\limits_{j=1}^i\varepsilon_j=0$ for some $0<i<m$.
Then $\gamma=\gamma_2\circ\gamma_1$, where
\[
\gamma_1:=
\mf{C}_X^{\varepsilon_i}\circ\cdots\circ
\mf{C}_X^{\varepsilon_1},
\qquad
\gamma_2:=
\mf{C}_X^{\varepsilon_m}\circ\cdots\circ
\mf{C}_X^{\varepsilon_{i+1}}
\]
both belong to $\Omega$.
For $(z_1,z_3)\in W(\mb{C}_v)$, choose $z_2$ such that
$(z_1,z_2)\in\gamma_1$ and $(z_2,z_3)\in\gamma_2$.
Both pairs lie in $T_X^{(0)}(\delta)$ by induction, so
Lemma~\ref{lem_comp}(1) gives
$(z_1,z_3)\in T_X^{(0)}(\delta)$.

It remains to consider the case
$\sum\limits_{j=1}^i\varepsilon_j<0$ for every $0<i<m$.
Then $\varepsilon_1=-1$, $\varepsilon_m=1$, and
\[
\gamma=\mf{C}_X\circ\gamma'\circ\trans\mf{C}_X,
\qquad
\gamma':=
\mf{C}_X^{\varepsilon_{m-1}}\circ\cdots\circ
\mf{C}_X^{\varepsilon_2}\in\Omega.
\]
For $(z_1,z_4)\in W(\mb{C}_v)$, choose $z_2,z_3$ such that
\[
(z_2,z_1),(z_3,z_4)\in\mf{C}_X,
\qquad
(z_2,z_3)\in\gamma'.
\]
Corollary~\ref{cor_Ztube} and the induction hypothesis give
\[
(z_2,z_1),(z_3,z_4)\in T_X^{(1)}(\delta),
\qquad
(z_2,z_3)\in T_X^{(0)}(\delta).
\]
Lemma~\ref{lem_comp}(1) first yields
$(z_2,z_4)\in T_X^{(1)}(\delta)$.
Applying Lemma~\ref{lem_comp}(2) to $(z_2,z_1)$ and
$(z_2,z_4)$ then gives $(z_1,z_4)\in T_X^{(0)}(\delta)$.
\end{proof}

\begin{proof}[Proof of Theorem~\ref{thm_relation}]
Set $\ov{H}:=\pi_1^*\ov{L}+\pi_2^*\ov{L}$ and
\[
T:=\bigcup_{\gamma\in\Omega}\Supp(\gamma)(\ov{K}).
\]
Let $W$ be an irreducible component of $\mf{R}_X$. For each $\gamma\in\Omega$, consider the irreducible components of $\gamma$ contained in $W$. The union of all these components
is Zariski dense in $W$. By Lemma~\ref{lem_gammaess}, each such component has
essential minimum zero. We may therefore choose a generic sequence
$(\xi_m)_{m\geq 1}$ in $A\cap W(\ov{K})$ with
$h_{\ov{H}}(\xi_m)\to 0$. Nefness and Zhang's inequality give $h_{\ov{H}}(W)=0$.

By the equidistribution theorem \cite[Theorem 3.1]{Yua08}, the measures
\[
\mu_{\xi_m,v}:=\frac{1}{\deg(\xi_m)}\sum_{\xi\in \Gal(\ov{K}/K)\cdot \xi_m}\delta_{\xi}
\]
on $W_{\mb{C}_v}^{\an}$ converge weakly to the normalized Chambert–Loir measure
\[
\mu_{\ov{H}|_W,v}:=\frac{1}{\deg_H(W)}\, c_1(\ov{H}|_W)_v^{\dim W}.
\]
By Proposition \ref{prop_gamma}, the support of $\mu_{\xi_m,v}$ is contained in the compact closed subset $W_{\mb{C}_v}^{\an}\cap T_X^{(0)}(\delta)$ for some $\delta\in (0,1)$ and all $m\geq 1$. Therefore, the Chambert–Loir measure $c_1(\ov{H}|_W)_v^{\dim W}$ is also supported on $W_{\mb{C}_v}^{\an}\cap T_X^{(0)}(\delta)$.

Let $\mc{W}$ be the Zariski closure of $W$ in $\mc{X}\times_{\Spe O_K} \mc{X}$ and set $\wt{W}$ to be the special fiber of $\mc{W}$ at $v$. Since $v$ is a model place with respect to $\ov{L}$, it is also a model place with respect to $\ov{H}|_W$. By Lemma \ref{lem_measure}, the Chambert–Loir measure can be written as
\[
c_1(\ov{H}|_W)_v^{\dim W}=\sum_{\eta} a_{\eta}\,\delta_{\rho_{\eta}},
\]
where $\eta$ runs over the irreducible components of the special fiber $\wt{W}$, and $a_{\eta}>0$. This implies that $\rho_{\eta}\in T_X^{(0)}(\delta)$ for all $\eta$. By Corollary \ref{cor_shilovtube}, we deduce
\[
\wt{W}\subset \Delta_{\mb{P}^N_{k_v}\times \mb{P}^N_{k_v}} \cap (\wt{X}\times \wt{X})=\Delta_{\wt{X}}.
\]
Since $\wt{W}$ is equidimensional and
\[
\dim \wt{W}=\dim W \geq \dim X=\dim \Delta_{\wt{X}},
\]
it follows that $\dim W=n$. Thus every component of $\mf{R}_X$ has dimension $n$.
Any component of a correspondence in $\Omega$ contained
in $W$ must consequently equal $W$. Moreover, $\Supp(\wt{\mf{R}_X})=\Delta_{\wt{X}}$.

Finally, note that $\Omega$ satisfies the following properties:
\begin{enumerate}
\renewcommand{\labelenumi}{(\theenumi)}
    \item $\Delta_X\in \Omega$;
    \item $\gamma\in \Omega$ if and only if $\trans\gamma\in \Omega$;
    \item if $\gamma_1,\gamma_2\in \Omega$, then $\gamma_2\circ \gamma_1\in \Omega$;
    \item if $\gamma\in \Omega$, then $\mf{C}_X\circ \gamma\circ \trans \mf{C}_X\in \Omega$.
\end{enumerate}
It follows that $\mf{R}_X$ also satisfies all the desired properties.     
\end{proof}

\subsubsection*{Proof of technical lemmas}
We begin with a height comparison obtained from
Lemma~\ref{lem_htcomp1}.

\begin{lemma}\label{lem_htcompZ}
Under the same assumptions as Theorem \ref{thm_constructZ}, there exists a constant $C>0$ such that for any $\epsilon>0$, there exists an open subscheme $U_{\epsilon}\subset X$ such that for all $(y,z)\in \big((U\times U)\cap \mf{C}_X\big)(\ov{K})$, we have
\[
h_{\ov{L}}(y)\leq C\,h_{\ov{L}}(z)+\epsilon,\qquad
h_{\ov{L}}(z)\leq C\,h_{\ov{L}}(y)+\epsilon.
\]
\end{lemma}
\begin{proof}
Let $W$ be an irreducible component of $\mf{C}_X$. Set $\ov{L}_1:=\pi_1^*\ov{L}|_W$ and $\ov{L}_2:=\pi_2^*\ov{L}|_W$. Then $\ov{L}_1$ and $\ov{L}_2$ are nef adelic line bundles, and $L_1$ and $L_2$ are ample.

There exists a subsequence $(y_m)_{m\geq 1}$ of $(x_m)_{m\geq 1}$ such that
\[
\left\{\big(y_m,\sigma(y_m)\big)\mid m\geq 1\right\}
\]
is Zariski dense in $W(\ov{K})$, and
\[
\lim_{m\to\infty} h_{\ov{L}_1+\ov{L}_2}\big((y_m,\sigma(y_m))\big)=0.
\]
It follows that the essential minimum $\mr{ess}(W,\ov{L}_1+\ov{L}_2)=0$, and hence $h_{\ov{L}_1+\ov{L}_2}(W)=0$ by Zhang's inequality.

Define
\[
C_W:=n\cdot\max\left\{\frac{L_1^{n-1}\cdot L_2}{L_1^n},\frac{L_2^{n-1}\cdot L_1}{L_2^n}\right\}+1.
\]
By Lemma \ref{lem_htcomp1}, for any $\epsilon>0$, there exists an open subscheme $U_{W,\epsilon}\subset W$ such that
\[
h_{\ov{L}_1}(x)\leq C_W\, h_{\ov{L}_2}(x)+\epsilon,\qquad
h_{\ov{L}_2}(x)\leq C_W\,h_{\ov{L}_1}(x)+\epsilon,\quad \forall x\in U_{W,\epsilon}(\ov{K}).
\]
We choose $C:=\max_W C_W$ and an open subscheme $U_{\epsilon}\subset X$ such that
\[
(U_{\epsilon}\times U_{\epsilon})\cap W \subset U_{W,\epsilon}
\]
for every irreducible component $W\subset \mf{C}_X$. Then for any $(y,z)\in \big((U_{\epsilon}\times U_{\epsilon})\cap W\big)(\ov{K})$, we have
\[
h_{\ov{L}}(y)=h_{\ov{L}_1}\big((y,z)\big)\leq C_W\, h_{\ov{L}_2}\big((y,z)\big)+\epsilon
\leq C\, h_{\ov{L}}(z)+\epsilon,
\]
and similarly $h_{\ov{L}}(z)\leq C\, h_{\ov{L}}(y)+\epsilon$.    
\end{proof}
\vspace{1em}

\begin{proof}[Proof of Lemma \ref{lem_gammaess}]
Suppose that $\gamma=\mf{C}_X^{\varepsilon_1}\circ \cdots\circ \mf{C}_X^{\varepsilon_m}$. For any point $x=(y,z)\in W(\ov{K})$, there exist points $z_1,\dots, z_{m+1}\in X(\ov{K})$ such that $z_1=y$, $z_{m+1}=z$, and $(z_i,z_{i+1})\in \mf{C}_X^{\varepsilon_i}(\ov{K})$ for $1\leq i\leq m$.

Let $C>1$ be the constant given by Lemma \ref{lem_htcompZ}. By Lemma \ref{lem_htcompZ}, for any $c>0$, there exists an open subscheme $U\subset X$ such that for all $1\leq i\leq m$, if
\[
(z_i,z_{i+1})\in \big((U\times U)\cap \mf{C}_X^{\varepsilon_i}\big)(\ov{K}),
\]
then $h_{\ov{L}}(z_{i+1})\leq C\, h_{\ov{L}}(z_i)+c$.

Choose an open subset $U_m\subset U$ such that
\[
\pi_2\big(\pi_1^{-1}(U_m)\cap \mf{C}_X^{\varepsilon_m}\big)\subset U.
\]
Then $z_m\in U_m$ implies $z_{m+1}\in U$. Similarly, choose an open subset $U_{m-1}\subset U_m$ such that
\[
\pi_2\big(\pi_1^{-1}(U_{m-1})\cap \mf{C}_X^{\varepsilon_{m-1}}\big)\subset U_m.
\]
Proceeding inductively, we obtain a sequence of open subsets
\[
U_1\subset U_2\subset \cdots \subset U_m\subset U_{m+1}:=U,
\]
such that for each $1\leq i\leq m$, $z_i\in U_i$ implies $z_{i+1}\in U_{i+1}$. In particular, if $y=z_1\in U_1(\ov{K})$, then
\[
\begin{aligned}
h_{\ov{L}}(z)
&\leq C\, h_{\ov{L}}(z_m)+c\\
&\leq C^2\, h_{\ov{L}}(z_{m-1})+(C+1)c\\
&\leq \cdots \\
&\leq C^m\, h_{\ov{L}}(z_1)+(C^{m-1}+C^{m-2}+\cdots+C+1)c\\
&\leq C^m\, h_{\ov{L}}(y)+mC^m c.
\end{aligned}
\]

Since $h_{\pi_1^*\ov{L}+\pi_2^*\ov{L}}(y,z)=h_{\ov{L}}(y)+h_{\ov{L}}(z)$ and $\mr{ess}(X,\ov{L})=0$, it follows that
\[
\mr{ess}(W,\pi_1^*\ov{L}+\pi_2^*\ov{L})\leq (1+C^m)\mr{ess}(X,\ov{L})+mC^m c = mC^m c.
\]
Letting $c\to 0$, we obtain $\mr{ess}(W,\pi_1^*\ov{L}+\pi_2^*\ov{L})=0$, and $h_{\pi_1^*\ov{L}+\pi_2^*\ov{L}}(W)=0$ follows immediately.    
\end{proof}

\medskip

The same argument gives a height comparison on $\mf{R}_X$. It will be used in the next subsection.
\begin{lemma}\label{lem_htcompR}
Under the same assumptions as Theorem \ref{thm_constructZ}, there exists a constant $C>0$ such that for any $\epsilon>0$, there exists an open subscheme $U_{\epsilon}\subset X$ such that for all $(y,z)\in \big((U_{\epsilon}\times U_{\epsilon})\cap \mf{R}_X\big)(\ov{K})$, we have
\[
h_{\ov{L}}(y)\leq C\,h_{\ov{L}}(z)+\epsilon,\qquad
h_{\ov{L}}(z)\leq C\,h_{\ov{L}}(y)+\epsilon.
\]
\end{lemma}

\begin{proof}
The proof is identical to that of Lemma \ref{lem_htcompZ}. Let $\ov{H}:=\pi_1^*\ov{L}+\pi_2^*\ov{L}$, 
Note that $\mf{R}_X$ is a finite union of $\gamma\in \Omega$. For every irreducible component $W$ of $\mf{R}_X$, we have $\mr{ess}(W,\ov{H})=h_{\ov{H}}(W)=0$ by Lemma \ref{lem_gammaess}. The rest of the proof follows from the same argument as Lemma \ref{lem_htcompZ}.
\end{proof}

\subsection{Proof of Theorem \ref{thm_main1}}
The main goal of this subsection is to prove the following theorem.

\begin{theorem}[Theorem \ref{thm_main1}]\label{thm_main}
Let $K$ be a number field, and let $X$ be a normal projective variety over $K$. Let $\overline{L}$ be a nef adelic line bundle on $X$ with $L$ ample. Assume that there exist a finite set $S$ of places of $K$ and a generic sequence $(x_m)_{m \geq 1}$ in $X(\overline{K})$ such that $(x_m)_{m\geq 1}$ is abelian, almost unramified at every $v \notin S$, and satisfies $\lim\limits_{m \to \infty} h_{\overline{L}}(x_m) = 0$.

Then there exist a normal projective variety $Y$ over $K$, a finite surjective morphism $\pi : X \to Y$, a polarized endomorphism $f : Y \to Y$, and an $f$-admissible adelic line bundle $\overline{M}_f \in \widehat{\mathrm{Pic}}(Y)_{\mathbb{Q}}$ such that
\[
\overline{L} \leq \pi^* \overline{M}_f.
\]
Moreover, the endomorphism $f$ is an exceptional map.
\end{theorem}

We first fix some notation and conventions. Replacing $L$ by a suitable multiple, there exists an arithmetic model $(\mathcal X,\mathcal L)$ of $(X,L)$ over $\Spe O_K$. Enlarging $S$ if necessary, we may assume that every finite place $v\notin S$ is a model place for $\ov{L}$. Equivalently, writing $\mathcal V=\Spe O_{K,S}$, the adelic line bundle $\ov{L}$ is induced over $\mathcal V$ by $(\mathcal X_{\mathcal V},\mathcal L_{\mathcal V})$, and
$\mathcal L|_{\mathcal X_v}$ is ample for every $v\notin S$.

For the remainder of this subsection, we fix a finite place
$v\notin S$. By the assumption of Theorem \ref{thm_main}, there exist a generic sequence $(x_m)_{m \geq 1}$ in $X(\overline{K})$ such that $(x_m)_{m\geq 1}$ is abelian, almost unramified at $v$, and satisfies $\lim\limits_{m \to \infty} h_{\overline{L}}(x_m) = 0$. Fix an embedding $\ov{K}\hookrightarrow \mb{C}_v$, which induces a valuation $w$ of $\ov{K}$ extending $v$. Let $\phi_q\in \Gal(\overline{k_v}/k_v)$ be the Frobenius automorphism, and choose a Frobenius element $\sigma\in D_w\subset \Gal(\overline K/K)$ lifting $\phi_q$, where $q=|k_v|$.

\subsubsection*{Construction of $Y$ and $f$} 

By Theorem \ref{thm_constructZ} and Theorem \ref{thm_relation}, we obtain bi-finite correspondences $\mf{C}_X^{(v)}$ and $\mf{R}_X^{(v)}$ in $X\times X$. Let $Y:=X/\mf{R}_X^{(v)}$ be the geometric quotient. Then $Y$ is a normal projective variety over $K$ and the quotient map $\pi: X\to Y$ is finite surjective. Moreover, for any $x,y\in X(\ov{K})$, $\pi(x)=\pi(y)$ if and only if $(x,y)\in \mf{R}_X^{(v)}(\ov{K})$. By our construction, $\Supp (\mf{C}_X^{(v)}\circ \mf{R}_X^{(v)}\circ \trans\mf{C}_X^{(v)})=\Supp \mf{R}_X^{(v)}$. It follows from Lemma \ref{lem_descent} that $\mf{C}_X^{(v)}$ descends to a surjective endomorphism $f: Y\to Y$. To complete the proof, it remains to verify the properties of $f$. 

\begin{lemma}\label{lem_fpolarized}
With notations as above, the morphism $f$ is polarized and $\deg(f)=q^n$.
\end{lemma}
\begin{proof}
Choose a relatively ample line bundle $\mc{H}$ on $\mc{X}$ with generic fiber $H\in \Pic(X)$. Let $W\subset \mf{C}_X^{(v)}$ be an irreducible component and denote $\pi_{1,W},\pi_{2,W}:W\to X$ the projections. Let $\mc{W}$ be the Zarsiki closure of $W$ in $\mc{X}\times_{O_K}\mc{X}$ and $\wt{W}$ be the special fiber of $\mc{W}$ at $v$.

By Theorem \ref{thm_constructZ}, $\Supp(\wt{W})=\Gamma_{\Phi_{q,\wt{X}}}$. It follows that
\[\big((\pi_{2,\mc{W}})^*\mc{H}-q(\pi_{1,\mc{W}})^*\mc{H}\big)|_{\wt{W}}\equiv 0\]
in $N^1(\wt{W})$. This implies $(\pi_{2,W})^*H\equiv q^i(\pi_{1,W})^*H$ in $N^1(W)$ by \cite[Theorem 10.2]{Ful98}. In other words, the correspondence $\mf{C}_X$ is numerically $q$-polarized. By Lemma \ref{lem_descent}, the morphism $f$ induced by $\mf{C}_X^{(v)}$ is $q$-polarized, and consequently $\deg(f)=q^n$.
\end{proof}

\begin{lemma}\label{lem_prep}
With notations as above, for all but finitely many $m\geq 1$, $f\big(\pi(x_m)\big)=\sigma\big(\pi(x_m)\big)$. In particular, $\pi(x_m)\in\mr{Per}(f)$ for all but finitely many $m\geq 1$.    
\end{lemma}
\begin{proof}
 By construction of $\mf{C}_X^{(v)}$, we have
\[
\big(x_m,\sigma(x_m)\big)\in \mf{C}_X^{(v)}(\ov{K})
\]
for all but finitely many $m$. Hence
\[
f(\pi(x_m))=\pi(\sigma(x_m))=\sigma(\pi(x_m)).
\]

Suppose that $\sigma|_{K(x_m)}\in \Gal(K(x_m)/K)$ has order $r$, then
\[
\pi(x_m)
=\pi(\sigma^r(x_m))
=\sigma^{r-1}\big(\pi(\sigma(x_m))\big)
=\sigma^{r-1}\big(f(\pi(x_m))\big)
=\cdots
= f^r(\pi(x_m)).
\]
Thus $\pi(x_m)\in \mr{Per}(f)$ for all but finitely many $m\geq 1$.
\end{proof}

\subsubsection*{Construction of the $f$-admissible adelic line bundle $\ov{M}_f$} 

We begin with a lemma that will play a key role in the construction of $\overline{M}_f$.

\begin{lemma}\label{lem_htdp}
With notations as above, let $\langle \ov{L}\rangle_{X/Y}$ be the Deligne pairing of $\ov{L}$ with respect to $\pi:X\to Y$, then
\[\lim_{m\to\infty}h_{\langle \ov{L}\rangle_{X/Y}}\big(\pi(x_m)\big)=0.\]
\end{lemma}
\begin{proof}
 By Lemma \ref{lem_htcompR}, there exists a constant $C>0$, such that for any $\epsilon>0$, there exists an open subscheme $U_{\epsilon}\subset X$ such that 
 \[h_{\ov{L}}(z)\leq C\,h_{\ov{L}}(y)+\epsilon,\quad \forall (y,z)\in \big((U_{\epsilon}\times U_{\epsilon})\cap \mf{R}_X\big)(\ov{K}).\]
Consider the proper closed subset $V_{\epsilon}:=\pi^{-1}\big(\pi(X\setminus U_{\epsilon})\big)\supset X\setminus U_{\epsilon}$. For every $x\notin V_{\epsilon}(\ov{K})$ and $z\in \pi^{-1}\big(\pi(x)\big)$, we have $z\in U_{\epsilon}(\ov{K})$.

Since $(x_m)_{m\geq 1}$ is generic, $x_m\notin V_{\epsilon}(\ov{K})$ for $m$ sufficiently large. For $z\in \pi^{-1}\big(\pi(x_m)\big)$, we have $z\in U_{\epsilon}(\ov{K})$ and $\pi(x_m)=\pi(z)$ which implies $(x_m,z)\in \mf{R}_X^{(v)}(\ov{K})$. Then $(x_m,z)\in\big((U_{\epsilon}\times U_{\epsilon})\cap  \mf{R}_X^{(v)}\big)(\ov{K})$. In particular,
 \[h_{\ov{L}}(z)\leq C\,h_{\ov{L}}(x_m)+\epsilon.\]
 By Lemma \ref{lem_dphtcomp},
 \[h_{\langle \ov{L}\rangle_{X/Y}}\big(\pi(x_m)\big)\leq \deg(\pi)\max_{z\in \pi^{-1}(\pi(x_m))}h_{\ov{L}}(z)\leq C\,\deg(\pi)h_{\ov{L}}(x_m)+\deg(\pi)\epsilon.\]
 Let $m\to\infty$ and $\epsilon\to 0$, we obtain $\lim\limits_{m\to\infty}h_{\langle \ov{L}\rangle_{X/Y}}\big(\pi(x_m)\big)=0$.
\end{proof}

\begin{lemma}\label{lem_lbcomp}
With notations as above, there exists a $f$-admissible adelic line bundle $\ov{M}_f\in \widehat{\Pic}(Y)_{\mr{nef},\mb{Q}}$ such that $\ov{L}\leq \pi^*\ov{M}_f$.    
\end{lemma}
\begin{proof}
 Choose $\ov{M}:=\langle \ov{L}\rangle_{X/Y}\in \widehat{\Pic}(Y)_{\mr{nef}}$. Then $M=\langle L\rangle_{X/Y}\in \mr{Pic}(Y)$ is an ample line bundle. By Theorem \ref{thm_admiss}, there exists a $f$-admissible adelic line bundle $\ov{M}_f\in \widehat{\Pic}(Y)_{\mb{Q},\mr{nef}}$ extending $M$. 

We claim $\ov{M}_f-\ov{M}\in \psi^*\widehat{\Pic}(K)$, where $\psi: Y\to K$ is the structural morphism. By Lemma \ref{lem_htdp},
\[\lim_{m\to\infty}h_{\ov{M}+\ov{M}_f}(\pi(x_m))=0.\]
By Zhang's inequality $h_{\ov{M}+\ov{M}_f}(Y)=0$ and then $\ov{M}^i\cdot\ov{M}_f^{n+1-i}=0$ for all $0\leq i\leq n+1$. In particular, $(\ov{M}-\ov{M}_f)^2\cdot (\ov{M}+\ov{M}_f)^{n-1}=0$.

Choose an adelic line bundle $\ov{N}\in \widehat{\Pic}(K)$ with $\widehat{\deg}(\ov{N})=1$. Since $\ov{M}_f$ and $\ov{M}$ has the same underlying line bundle, we have
\[(\ov{M}-\ov{M}_f)^2\cdot (\ov{M}+\ov{M}_f+\psi^*\ov{N})^{n-1}=(\ov{M}-\ov{M}_f)^2\cdot (\ov{M}+\ov{M}_f)^{n-1}=0.    \]

Applying \cite[Theorem 1.3]{YZ17}, $\ov{M}_f-\ov{M}\in \psi^*\widehat{\Pic}(K)$. Suppose $\ov{M}_f-\ov{M}=\psi^*\ov{E}$ for some $\ov{E}\in \widehat{\Pic}(K)$, then $h_{\ov{M}_f-\ov{M}}(x_m)=\widehat{\deg}(\ov{E})$. Let $m\to\infty$, we have $\widehat{\deg}(\ov{E})=0$. Thus $\ov{M}_f-\ov{M}$ is nef. Combine this with Lemma \ref{lem_dpnef}, we obtain
\[\pi^*\ov{M}_f-\ov{L}=\pi^*(\ov{M}_f-\ov{M})+(\pi^*\ov{M}-\ov{L})\]
is nef.
\end{proof}

\subsubsection*{Exceptionality of $f$}

\begin{lemma}\label{lem_fexp}
With notations as above, the polarized endomorphism $f:Y\to Y$ is an exceptional map.
\end{lemma}
\begin{proof}
By the proof of Lemma \ref{lem_lbcomp}, the underlying line bundle $M$ of $\overline{M}_f$ is ample. Choose another place $v' \notin S$ such that $v'$ is a model place for $\overline{M}_f$, and such that $\gcd(q,q')=1$, where $q' := |k_{v'}|$. Let $w'$ be a valuation of $\overline{K}$ extending $v'$, and let $\sigma' \in D_{w'} \subset \mathrm{Gal}(\overline{K}/K)$ be a Frobenius element lifting $\phi_{q'}\in \Gal(\ov{k_{v'}}/k_{v'})$.

Consider the set
\[
B := \mathrm{Gal}(\overline{K}/K)\cdot \left\{ \pi(x_m) \;\middle|\; \pi(x_m) \in \mathrm{Per}(f),\ m \ge 1 \right\}\subset Y(\overline{K}) .
\]
By Lemma \ref{lem_prep}, the set $B$ is infinite and satisfies $f(B)=B$. Reindexing $B$ as a sequence $(y_m)_{m \ge 1}$, we obtain a generic sequence such that $\lim\limits_{m\to\infty}h_{\overline{M}_f}(y_m)=0$. Since $(x_m)_{m\geq 1}$ is abelian and almost unramifed at $v'$, $(y_m)_{m\geq1}$ is also abelian and almost unramified at $v'$.

Set
\[
V' := \overline{\left\{ (y_m,\sigma'(y_m)) \mid m \ge 1 \right\}}^{\mathrm{Zar}} \subset Y \times Y,
\]
and let $\mf{C}_Y^{(v')}$ be the union of all positive-dimensional irreducible components of $V'$.  Then $\mf{C}_Y^{(v')}$ is a bi-finite correspondence of dimension $n$ by Theorem~\ref{thm_constructZ}. By Theorem~\ref{thm_relation}, it induces a bi-finite correspondence $\mf{R}_Y^{(v')} \subset Y \times Y$. Then we obtain a normal projective variety $Y'=Y/\mf{R}_Y^{(v')}$ defined over $K$, a finite surjective morphism $\pi' : Y \to Y'$, and a surjective endomorphism $g : Y' \to Y'$. By Lemma \ref{lem_fpolarized}, $g$ is $q'$-polarized and  $\deg(g) = (q')^n$. 

Since $f(B)=B$, we have $(f \times f)(\mf{C}_Y^{(v')}) = \mf{C}_Y^{(v')}$ and $(f \times f)(\mf{R}_Y^{(v')}) = \mf{R}_Y^{(v')}$. By Lemma~\ref{lem_descent} and Remark \ref{rmk_bifin_corr}, the morphism $f:Y\to Y$ descends to a polarized endomorphism $f_{Y'} : Y' \to Y'$ satisfying
\[
\pi' \circ f = f_{Y'}\circ \pi', \qquad f_{Y'} \circ g = g \circ f_{Y'}.
\]
Moreover, $\deg(f_{Y'}) = \deg(f) = q^n$. Since $\gcd(q,q')=1$, it follows from Theorem \ref{thm_commuting} that $f_{Y'}$ is exceptional. Then $f$ is also exceptional by Theorem \ref{thm_semiconj}.
\end{proof}

Theorem \ref{thm_main} now follows from Lemma \ref{lem_fpolarized}, Lemma \ref{lem_prep}, Lemma \ref{lem_lbcomp}, and Lemma \ref{lem_fexp}.

\section{Abelian backward orbits}\label{sec_backward_orbit}
The main goal of this section is to prove Theorem \ref{thm_vaeq1}. We prove Theorem \ref{thm_Kexp1} in Section \ref{sec_step4} and complete the proof of Theorem \ref{thm_vaeq1} in Section \ref{sec_proofvaeq}. In Section \ref{sec_1dim}, we discuss the one-dimensional case and Conjecture \ref{conj_P1}. Throughout this section, $K$ is a number field and $X$ is a normal projective variety over $K$.

\subsection{Complex multiplication and virtual abelianity}\label{sec_step4}
Let $f:X\to X$ be a polarized endomorphism of $X$ and $\alpha\in X(K)$ be a non-exceptional point for $f$. We study the relation between complex multiplication and the virtual abelianity of $K_{\infty}(f,\alpha)/K$. Our main result in this subsection is the following generalization of \cite[Theorem D]{FOZ24}.

\begin{theorem}[Theorem \ref{thm_Kexp1}]\label{thm_Kexp}
Let $K$ be a number field, let $X$ be a normal projective variety
over $K$, and let $f:X\to X$ be a $K$-exceptional map.
Suppose that $\alpha\in X(K)$ is preperiodic and non-exceptional
for $f$. Then $K_{\infty}(f,\alpha)/K$ is virtually abelian
if and only if $f$ is an exceptional map with CM.
\end{theorem}

We first establish the corresponding statements for AT-varieties.
The implication from complex multiplication to virtual abelianity
is proved using Frobenius lifts and the Chebotarev density theorem.
The converse relies on Zarhin's result \cite[Theorem~1]{Zar87}:
for a $K$-simple abelian variety $A$, the set
$A(K^{\ab})_{\mr{tors}}$ is infinite if and only if $A$ has
complex multiplication over $K$. We then apply these results to exceptional maps at the end
of the subsection.

We begin with the following lemma, which will be used repeatedly throughout this subsection.
\begin{lemma}\label{lem_va}
Let $L/K$ be a Galois extension of fields. Then $L/K$ is virtually abelian if and only if there exists a finite Galois extension $K'/K$ such that $LK'/K'$ is abelian. 
\end{lemma}

\begin{proof}
This is a standard consequence of the fundamental theorem of infinite Galois theory; see \cite[Chapter~7]{Mil22}.
\end{proof}

We next formulate the Chebotarev criterion used in the forward
implication. To state it, we first specify the integral models on
which the reductions of endomorphisms will be considered.

\begin{definition}
Let $(Y,A,\pi)$ be an AT-variety over $K$ and $S$ be a finite set of finite places of $K$. A \emph{good model} of $(Y,A,\pi)$ over $\Spe O_{K,S}$ consists of:
\begin{itemize}
    \item an abelian scheme $\mc{A}$ over $\Spe O_{K,S}$;
    \item a torus scheme $\mathcal T\to\Spe O_{K,S}$;
    \item a torsor $\Pi:\mathcal Y\to\mathcal A$ under $\mathcal T\times_{\Spe O_K}\mathcal A$, together with identifications of their generic fibers with $A$, $T$, and $Y$, compatible with the projection and torus action.
\end{itemize}    
\end{definition}

\begin{lemma}\label{lem_Chebo}
Let $(Y,A,\pi)$ be an AT-variety over a number field $K$, and let
$\psi:Y\to Y$ be a $d$-lift. Suppose that
$\alpha\in Y(K)$ is fixed by $\psi$.
Assume that there exist a finite set $S$ of finite places of $K$
and a good model $(\mc{Y},\mc{A},\Pi)$ of $(Y,A,\pi)$ over
$\Spe O_{K,S}$ such that, for every $v\notin S$, there is an endomorphism $\phi^{(v)}:Y\to Y$ defined over $K$ satisfying:
\begin{enumerate}
\renewcommand{\labelenumi}{(\theenumi)}
\item $\phi^{(v)}$ extends to an $O_{K_v}$-endomorphism of
$\mc{Y}_{O_{K_v}}$ whose special fiber is the Frobenius
$\Phi_{q_v,\mc{Y}_v}$, where $q_v:=|k_v|$ and $k_v$ is the
residue field at $v$;
\item $\phi^{(v)}\circ\psi=\psi\circ\phi^{(v)}$;
\item $\phi^{(v)}(\alpha)=\alpha$.
\end{enumerate}
Then $K_{\infty}(\psi,\alpha)/K$ is an abelian extension.
\end{lemma}
\begin{proof}
Fix $n\geq 1$, and set
\[
B_n:=\psi^{-n}(\alpha), \qquad
K_n:=K(B_n), \qquad
G_n:=\operatorname{Gal}(K_n/K).
\]
Since $B_n$ is finite and Galois-stable, $K_n/K$ is a finite
Galois extension, and $G_n$ acts faithfully on $B_n$.
For every $v\notin S$ and $\beta\in B_n$, we have
\[
\psi^n\bigl(\phi^{(v)}(\beta)\bigr)
=\phi^{(v)}\bigl(\psi^n(\beta)\bigr)
=\phi^{(v)}(\alpha)
=\alpha.
\]
Thus $\phi^{(v)}$ preserves $B_n$.
There is a finite set $S_n\supseteq S$ of finite places of $K$
such that, for every $v\notin S_n$, the extension $K_n/K$ is
unramified at $v$, and all points of $B_n$ admit pairwise
distinct reductions at every place of $K_n$ above $v$. Fix $v\notin S_n$ and a place $w$ of $K_n$ above $v$.
Let $\sigma_w\in G_n$ be the arithmetic Frobenius element at
$w$. Write
\[
\operatorname{red}_w:B_n\longrightarrow\mc{Y}_v(k_w)
\]
for the reduction map, where $k_w$ is the residue field at $w$. By the choice of $S_n$, this map is well-defined and injective. Both $\phi^{(v)}(\beta)$ and $\sigma_w(\beta)$ belong to $B_n$, so the injectivity of reduction gives $\phi^{(v)}(\beta)=\sigma_w(\beta)$ for every $\beta\in B_n$. Since $\phi^{(v)}$ is defined over $K$, for every $\tau\in G_n$
and $\beta\in B_n$ we obtain
\[
\tau\bigl(\sigma_w(\beta)\bigr)
=\tau\bigl(\phi^{(v)}(\beta)\bigr)
=\phi^{(v)}\bigl(\tau(\beta)\bigr)
=\sigma_w\bigl(\tau(\beta)\bigr).
\]
The faithfulness of the action of $G_n$ on $B_n$ therefore
implies that $\sigma_w\in Z(G_n)$.
By the Chebotarev density theorem, every conjugacy class of
$G_n$ occurs as a Frobenius class at a place outside $S_n$.
Since all these Frobenius elements are central, $G_n=Z(G_n)$.
Hence $K_n/K$ is abelian.
Finally, since $\psi(\alpha)=\alpha$, we have
$B_n\subseteq B_{n+1}$ and hence $K_n\subseteq K_{n+1}$.
Therefore
\[
\operatorname{Gal}\bigl(K_{\infty}(\psi,\alpha)/K\bigr)
\simeq \varprojlim_n G_n
\]
is abelian.
\end{proof}

We now construct the Frobenius lifts needed for
Lemma~\ref{lem_Chebo}, starting with the CM abelian base.
\begin{lemma}\label{lem_CMfrob}
Let $K$ be a number field, and let $A$ be an abelian variety
with complex multiplication over $K$. Then there exist a finite extension
$K'/K$ and a finite set $S$ of finite places of $K'$ such that:
\begin{enumerate}
\renewcommand{\labelenumi}{(\theenumi)}
\item $A_{K'}$ has good reduction at every $v\notin S$;
\item for every $v\notin S$, there exists
$\phi^{(v)}\in\mr{End}_{K'}(A_{K'})$ whose reduction at $v$
is the Frobenius $\Phi_{q_v,A_v}$, where $q_v=|k_v|$ and
$k_v$ is the residue field at $v$.
Moreover, $\phi^{(v)}$ commutes with every element of
$\mr{End}_{K'}(A_{K'})$.
\end{enumerate}
\end{lemma}

\begin{proof}
Since $A$ has CM over $K$, we may choose an étale subalgebra
$E\subset\mr{End}^0_K(A)$ of dimension $2\dim A$.
By \cite[Remark~3.9]{Mil06}, there exist a CM abelian variety
$B$ equipped with an action of the ring of integers $O_E$
and an isogeny $u:B\to A$ over $\ov K$.

Choose a positive integer $m$ annihilating $\ker u$, and
a finite extension $K'/K$ over which $B$, $u$, the
$O_E$-action, all conjugates of $E$, and all points of $B[m]$
are defined.
For simplicity of notation, replace $K$ by $K'$ and
$A$ by $A_{K'}$.
Choose a finite set $S$ of finite places of $K$ such that,
for every $v\notin S$, we have $v\nmid m$, both $A$ and $B$
have good reduction at $v$, and the rational prime below $v$
is unramified in $E$.
Let $\mc A$ and $\mc B$ be abelian schemes over
$\Spe O_{K,S}$ with generic fibers $A$ and $B$, respectively.

Fix $v\notin S$.
By \cite[Theorem~8.1(a), Remark~8.6(a)]{Mil06}, there exists
an endomorphism $\phi_B^{(v)}$ of $B$ whose reduction is
the Frobenius $\Phi_{q_v,\mc B_v}$.
Since $B[m]\subseteq B(K)$, every point of $B[m]$ has
$k_v$-rational reduction and is therefore fixed by Frobenius $\Phi_{q_v,\mc{B}_v}$
after reduction. Moreover, by \cite[Proposition~6.10]{Mil06}, the reduction map
\[
\mr{red}_{v,\mc B}:B[m]\longrightarrow\mc B_v(k_v)
\]
is injective.
It follows that $\phi_B^{(v)}$ fixes $B[m]$ pointwise.
In particular, it preserves $\ker u$ and hence descends to an endomorphism $\phi^{(v)}$ of $A=B/\ker u$.

Reducing the identity
\[
\phi^{(v)}\circ u=u\circ\phi_B^{(v)}
\]
at $v$, and using the compatibility of Frobenius with the
reduction of $u$, we conclude that the reduction of
$\phi^{(v)}$ is the Frobenius $\Phi_{q_v,\mc A_v}$.

Finally, let $h\in\mr{End}_K(A)$. Its reduction is defined over $k_v$ and therefore commutes
with $\Phi_{q_v,\mc A_v}$.
The injectivity of specialization on homomorphisms then gives $\phi^{(v)}\circ h=h\circ\phi^{(v)}$. This completes the proof.
\end{proof}

To lift these endomorphisms further to the torus torsor, we need
to control their pullbacks on the line bundles defining it.
The next lemma supplies the required isomorphisms.

\begin{lemma}\label{lem_Frob_linebundle}
Let $A$ be an abelian variety over a number field $K$, and let $\psi\in\operatorname{End}_K(A)$ be a $d$-lift, with $d\geq 2$. Let $S$ be a finite set of finite places of $K$ satisfying
\begin{enumerate}
\renewcommand{\labelenumi}{(\theenumi)}
    \item $A$ has good reduction at all $v\notin S$;
    \item for all $v\notin S$, there exists an endomorphism $\phi^{(v)}\in\operatorname{End}_K(A)$ commuting with $\psi$ such that the reduction of $\phi^{(v)}$ is the Frobenius $\Phi_{q_v,A_v}$, where $q_v:=|k_v|$ and $k_v$ is the residue field at $v$.
\end{enumerate}
Then, up to enlarging $S$, we have $(\phi^{(v)})^*L\simeq L^{\otimes q_v}$ for every $v\notin S$ and every $L\in\operatorname{Pic}(A)$ satisfying $\psi^*L\simeq L^{\otimes d}$.
\end{lemma}

\begin{proof}
The endomorphism $\psi^\vee-[d]$ of $A^\vee$ is an isogeny. Indeed, after choosing an embedding $K\hookrightarrow\mathbb C$,
the eigenvalues of the analytic representation of $\psi^\vee$
all have absolute value $\sqrt d$.
Since $d\geq 2$, none of them equals $d$, so
$\psi^\vee-[d]$ is an isogeny.

Choose a positive integer $N$ annihilating
$\ker(\psi^\vee-[d])$, and enlarge $S$ to include all places
dividing $N$.
Fix $v\notin S$ and $L\in\operatorname{Pic}(A)$ satisfying
$\psi^*L\simeq L^{\otimes d}$, and set
\[
M:=(\phi^{(v)})^*L\otimes L^{\otimes(-q_v)}.
\]
Since $\phi^{(v)}$ reduces to the Frobenius $\Phi_{q_v,A_v}$, the reduction of $M$ is trivial.
Consequently, the homomorphism
\[
\lambda_M:A\longrightarrow A^\vee,
\qquad
x\longmapsto t_x^*M\otimes M^{-1},
\]
has zero reduction.
The injectivity of specialization on homomorphisms of abelian
varieties implies that $\lambda_M=0$, and hence
$M\in\operatorname{Pic}^0(A)$.

Since $\phi^{(v)}$ commutes with $\psi$, we have
\[
\psi^*M
\simeq
(\phi^{(v)})^*(\psi^*L)
\otimes(\psi^*L)^{\otimes(-q_v)}
\simeq M^{\otimes d}.
\]
Viewing $M$ as a point of $A^\vee$, we obtain
\[
M\in\ker(\psi^\vee-[d])\subseteq A^\vee[N].
\]
Since $v\nmid N$, reduction is injective on $A^\vee[N]$.
The triviality of the reduction of $M$ therefore implies
$M\simeq\mathcal O_A$. Thus $(\phi^{(v)})^*L\simeq L^{\otimes q_v}$ as required.
\end{proof}

We can now combine the preceding lemmas to prove the implication
from complex multiplication to virtual abelianity for AT-varieties.

\begin{proposition}\label{prop_VA_CM1}
Let $(Y,A,\pi)$ be an AT-variety over a number field $K$. Let $\psi:Y\to Y$ be a $d$-lift, and let $\alpha\in Y(K)$ be a preperiodic point for $\psi$.
Suppose $A$ has complex multiplication over $\ov K$, then $K_{\infty}(\psi,\alpha)/K$ is virtually abelian.
\end{proposition}

\begin{proof}
Throughout the proof, we may replace $K$ by a finite extension,
since virtual abelianity is invariant under finite extensions
of the base field. 

\textbf{Step 1.} We first do some reductions. After replacing $K$ by a finite extension, we may assume that $A$ has complex multiplication over $K$. Choose $i\geq 0$ such that $\beta:=\psi^i(\alpha)$ is periodic of period $j$. Then
\[
K\bigl(\psi^{-\infty}(\alpha)\bigr)
\subseteq
K\bigl(\psi^{-\infty}(\beta)\bigr)
=
K\bigl((\psi^j)^{-\infty}(\beta)\bigr).
\]
Thus, replacing $\psi$, $d$, and $\alpha$ by $\psi^j$, $d^j$,
and $\beta$, respectively, we may assume that $\psi(\alpha)=\alpha$.

After replacing $K$ by a finite extension, we may assume that
$A$ has complex multiplication over $K$ and that the associated
torus $T$ is split. By Lemma~\ref{lem_CMfrob}, after a further
finite extension, there is a finite set $S$ of finite places
of $K$ such that $A$ has good reduction outside $S$ and, for
every $v\notin S$, there is an endomorphism $\phi_A^{(v)}\in\operatorname{End}_K(A)$ whose reduction at $v$ is the Frobenius
$\Phi_{q_v,\mc A_v}$, where $q_v:=|k_v|$.
Moreover, $\phi_A^{(v)}$ commutes with every element of
$\operatorname{End}_K(A)$.

Fix an isomorphism $T\simeq\mathbb G_{m,K}^r$, and let
$(A,L_1,\dots,L_r)$ be the tuple corresponding to $Y\to A$.
Since $\psi$ is $d$-polarized, we have $\psi_A^*L_i\simeq L_i^{\otimes d}$ for $1\leq i\leq r$. After enlarging $S$, choose an abelian scheme $\mc A$ over
$\Spe O_{K,S}$ and extensions $\mc L_i$ of $L_i$ such that
\[
\mc Y
=
\mc L_1^\times
\times_{\mc A}\cdots
\times_{\mc A}\mc L_r^\times
\]
defines a good model $(\mc Y,\mc A,\Pi)$ of $(Y,A,\pi)$.
We may also assume that $\psi$ extend over $O_{K,S}$, and that
$\alpha$ extends to a section $\widetilde\alpha\in\mc Y(O_{K,S})$.

\medskip

\textbf{Step 2.} Applying Lemma~\ref{lem_Frob_linebundle} to $\psi_A$ and
enlarging $S$ if necessary, we obtain $(\phi_A^{(v)})^*L_i\simeq L_i^{\otimes q_v}$ for $1\leq i\leq r$ and every $v\notin S$. Set $a:=\pi(\alpha)$. Since $\psi_A(a)=a$ and
$\psi_A-[1]$ is an isogeny, we may choose an integer $N\geq 1$
such that
\[
a\in\ker(\psi_A-[1])\subseteq A[N].
\]
Enlarge $S$ to include all places dividing $N$.
For every $v\notin S$, the points $\phi_A^{(v)}(a)$ and $a$
have the same reduction, since $a\in A(K)$ and
$\phi_A^{(v)}$ reduces to Frobenius. Both points lie in $A[N]$,
so injectivity of reduction on $A[N]$ gives $\phi_A^{(v)}(a)=a$.

By Lemma~\ref{lem_ATlift}, each $\phi_A^{(v)}$ lifts to a
$K$-morphism
\[
\phi^{(v)}:Y\longrightarrow Y
\]
inducing $[q_v]_T$ on $T$.
Since $\phi^{(v)}(\alpha)$ and $\alpha$ lie in the same fiber,
composing $\phi^{(v)}$ with a translation by an element of
$T(K)$ allows us to assume that $\phi^{(v)}(\alpha)=\alpha$.

\medskip

\textbf{Step 3.} We verify that $\phi^{(v)}$ extends over $\Spe O_{K_v}$. Denote $R:=O_{K_v}$ and write $\phi_{\mc{A}_R}^{(v)}$ for the
extension of $\phi_A^{(v)}$ to $\mc{A}_R$. By \cite[Tag~0BD7]{stacks-project}, the restriction map
\[
\operatorname{Pic}(\mc A_R)
\longrightarrow
\operatorname{Pic}(A_{K_v})
\]
is injective. The isomorphisms $(\phi_A^{(v)})^*L_i\simeq L_i^{\otimes q_v}$
therefore give isomorphisms
\[
\Theta_i:(\phi_{\mc{A}_R}^{(v)})^*\mc L_i
\xrightarrow{\sim}\mc L_i^{\otimes q_v}
\]
over $\mc A_R$. Set $\widetilde a:=\Pi(\widetilde\alpha)$. The section $\wt{\alpha}$ determines a generator $e_i$ of the free rank-one $R$-module $\widetilde a^*\mc L_i$ for each $i$. Since $\phi_{\mc{A}_R}^{(v)}$ fixes $\widetilde a$, pulling back $\Theta_i$ along $\widetilde a$ gives an isomorphism
\[\wt{a}^*\Theta_i: \wt{a}^*\mc{L}_i=\wt{a}^*(\phi_{\mc{A}_R}^{(v)})^*\mc L_i
\xrightarrow{\sim} \wt{a}^* \mc{L}_i^{\otimes q_v}.\]
Thus $(\widetilde a^*\Theta_i)(e_i)=u_i e_i^{\otimes q_v}$ for some $u_i\in R^\times$. Replacing $\Theta_i$ by $u_i^{-1}\Theta_i$, we may assume that $(\widetilde a^*\Theta_i)((e_i)=e_i^{\otimes q_v}$.

These normalized isomorphisms define an $R$-morphism
$\mc Y_R\to\mc Y_R$ lies over $\phi_{\mc{A}_R}^{(v)}$, inducing
$[q_v]$ on the torus, and fixing $\widetilde\alpha$.
Its generic fiber agrees with $\phi^{(v)}_{K_v}$ by
Lemma~\ref{lem_ATcomp}, since both have the same base map
and torus map and fix $\alpha$. Thus $\phi^{(v)}$ extends over $O_{K_v}$ and admits a
reduction $\phi^{(v)}_{\mc Y_v}$ on $\mc Y_v$.

\medskip

\textbf{Step 4.} Since $\phi_A^{(v)}$ commutes with $\psi_A$, the maps
$\phi^{(v)}\circ\psi$ and $\psi\circ\phi^{(v)}$ have the same
base map and the same torus map. Both fix $\alpha$, so
Lemma~\ref{lem_ATcomp} gives
\[
\phi^{(v)}\circ\psi=\psi\circ\phi^{(v)}.
\]
Similarly, $\phi^{(v)}_{\mc{Y}_v}$ and
$\Phi_{q_v,\mc Y_v}$ have the same base map and the same
torus map, and both fix
$\overline\alpha\in\mc Y_v(k_v)$.
The same uniqueness argument therefore gives $\phi^{(v)}_{\mc{Y}_v}=\Phi_{q_v,\mc Y_v}$. Thus, by Lemma~\ref{lem_Chebo}, $K(\psi^{-\infty}(\alpha))/K$ is abelian, which completes the proof.
\end{proof}

For the converse, we construct torsion points on the abelian base defined over $K^{\ab}$ and apply Zarhin's theorem.

\begin{proposition}\label{prop_VA_CM2}
Let $(Y,A,\pi)$ be an AT-variety over a number field $K$. Let $\psi:Y\to Y$ be a $d$-lift, and let $\alpha\in Y(K)$ be a preperiodic point for $\psi$.
Suppose $K_{\infty}(\psi,\alpha)/K$ is virtually abelian, then $A$ has complex multiplication over $\ov K$.    
\end{proposition}

\begin{proof}
After replacing $K$ by a finite extension, we may assume that
\[
K(\psi^{-\infty}(\alpha))\subseteq K^{\ab}.
\]
Put $a:=\pi(\alpha)$. For every $n\geq 1$, the map $\psi^n$
is surjective on each torus fiber onto the corresponding
target fiber. Therefore
\[
C_n:=\psi_A^{-n}(a)
=
\pi\bigl(\psi^{-n}(\alpha)\bigr)
\subseteq A(K^{\ab}).
\]
Since $C_n$ is a coset of $\ker(\psi_A^n)$, we have
\[
\ker(\psi_A^n)
=
\{u-v:u,v\in C_n\}
\subseteq A(K^{\ab})_{\mr{tors}}.
\]
Consequently,
\[
D:=\bigcup_{n\geq 1}\ker(\psi_A^n)
\subseteq A(K^{\ab})_{\mr{tors}}.
\]

Let $Z$ be the Zariski closure of $D$ in $A_{\ov K}$.
Since $\psi_A^{-1}(D)=D$ and $\psi_A$ is \'etale, hence open,
we have $\psi_A^{-1}(Z)=Z$. Applying Lemma~\ref{lem_invar_empty} to the polarized
\'etale endomorphism $\psi_A$, we obtain $Z=A_{\ov K}$.
Thus $A(K^{\ab})_{\mr{tors}}$ is Zariski dense in $A$.

Choose a $K$-isogeny
\[
A\longrightarrow A_1\times\cdots\times A_m,
\]
where each $A_i$ is $K$-simple.
The image of $A(K^{\ab})_{\mr{tors}}$ in each $A_i$ is
Zariski dense, so $A_i(K^{\ab})_{\mr{tors}}$ is infinite.
By Zarhin's theorem \cite{Zar87}, each $A_i$ has complex
multiplication over $K$.
Since being of CM type is preserved under products and
isogenies, $A$ has complex multiplication over $\ov K$.
\end{proof}

Propositions~\ref{prop_VA_CM1} and~\ref{prop_VA_CM2} establish
the desired equivalence for AT-varieties. We now prove Theorem~\ref{thm_Kexp}.

\begin{proof}[Proof of Theorem \ref{thm_Kexp}]
Since $f$ is a $K$-exceptional map, there exist an AT-variety
$(Y,A,\pi)$ over $K$, a finite surjective morphism $h:Y\to X^{\circ}\subset X$, and a
$d$-lift $\psi:Y\to Y$ such that 
\[
h\circ\psi=f^l\circ h
\]
for some $l\geq 1$. Replacing $f$ by $f^l$, we may assume that
$h\circ\psi=f\circ h$. Moreover, both $X^{\circ}$ and $X\setminus X^{\circ}$
are totally invariant under $f$. If $\alpha\notin X^{\circ}(K)$, then $f^{-\infty}(\alpha)\subseteq X\setminus X^{\circ}$, contradicting the non-exceptionality of $\alpha$. Hence $\alpha\in X^{\circ}(K)$.

Put $B:=h^{-1}(\alpha)$, which is a nonempty finite set. Every $\beta\in B$ is preperiodic for $\psi$. After replacing $K$ by a finite extension, we may assume that
$B\subseteq Y(K)$.

Suppose first that $f$ is exceptional with CM, meaning that
$A$ has complex multiplication over $\overline K$.
After a further finite extension of $K$, we may assume that $A$ has complex multiplication over $K$. By Proposition~\ref{prop_VA_CM1}, the extension $K\bigl(\psi^{-\infty}(\beta)\bigr)/K$ is virtually abelian for every $\beta\in B$.
Since $B$ is finite, after replacing $K$ by a further finite
extension, we may assume that
\[
K\bigl(\psi^{-\infty}(\beta)\bigr)\subseteq K^{\mathrm{ab}}
\]
for every $\beta\in B$. Their finite compositum $E:=K\bigl(\psi^{-\infty}(B)\bigr)$ is therefore abelian over $K$. Since
\[
h^{-1}\bigl(f^{-\infty}(\alpha)\bigr)=\bigcup_{\beta\in B}\psi^{-\infty}(\beta),
\]
it follows that $K(f^{-\infty}(\alpha))/K$ is abelian.

Conversely, suppose that $K(f^{-\infty}(\alpha))/K$ is virtually abelian.
Choose a point $\beta\in B\subseteq Y(K)$ that is non-exceptional
for $\psi$. Then Proposition \ref{prop_VA_CM2} implies that $A$ has complex
multiplication over $\overline K$. Thus $f$ is an exceptional map with CM.
\end{proof}

\subsection{Proof of Theorem \ref{thm_vaeq1}}\label{sec_proofvaeq}
In this subsection, we prove Theorem \ref{thm_vaeq1}. We fix some settings first. Let $X$ be a normal projective variety over a number field $K$ of dimension $n$, and let $f: X \to X$ be a polarized endomorphism such that $f^*L \simeq L^{\otimes d}$ for some ample line bundle $L$ and integer $d \geq 2$.

Replacing $L$ by a suitable multiple, we may choose an arithmetic model $(\mc{X},\mathcal{L})$ of $(X, L)$ over $\Spe  O_K$ .Choose a finite set $S$ of finite places of $K$ such that, over the open affine subscheme $\mathcal{V}:=\Spe O_{K,S}\subseteq\Spe O_K$, the morphism $f$ extends to $F_{\mathcal{V}}:
\mc{X}_{\mathcal{V}}\longrightarrow\mc{X}_{\mathcal{V}}$, and the isomorphism $f^*L\simeq L^{\otimes d}$ extends to an isomorphism $F_{\mathcal{V}}^*\mathcal{L}_{\mathcal{V}}
\simeq \mathcal{L}_{\mathcal{V}}^{\otimes d}$. For each finite place $v\notin S$, let
$f_v:\mc{X}_v\to\mc{X}_v$ denote the induced endomorphism
of the special fiber.

Define the \emph{critical locus}
\[
\Crit(F_{\mathcal{V}})
:=
\left\{
x \in \mc{X}_{\mathcal{V}}
\ \middle|\
F_{\mathcal{V}}: \mc{X}_{\mathcal{V}}\to  \mc{X}_{\mathcal{V}} \text{  or  } \pi_{\mathcal{V}} : \mc{X}_{\mathcal{V}} \to \mathcal{V}
\text{ is not smooth at } x
\right\},
\]
which is a proper closed subset of $\mc{X}_{\mathcal{V}}$.

Let $\Crit(f) \subset X$ denote the generic fiber of $\Crit(F_{\mathcal{V}})$, and let $\Crit(f_v)$ denote its special fiber over $v \notin S$. For any closed point $x \in X \setminus f(\Crit(f))$ (resp. $x \in \mc{X}_v \setminus f_v(\Crit(f_v))$), the fiber $f^{-1}(x)$ (resp. $f_v^{-1}(x)$) consists of $d^n$ distinct points.

\bigskip

Recall that for each finite place $v\notin S$, there is a reduction map
\[
\mathrm{red}_{v,\mc{X}} : X(\overline{K}) \to \mc{X}_v(\overline{k}_v),
\]
where $k_v$ is the residue field at $v$. We will use the following lemma.

\begin{lemma}\label{lem_unrami}
Let $x \in X(\overline{K})$ and $y=f(x)$. Let $v \notin S$ be a finite place. If $\mathrm{red}_{v,\mc{X}}(x) \notin \Crit(f_v)$,
then the extension $K(x)/K(y)$ is unramified at every place $w$ of $K(y)$ lying above $v$.
\end{lemma}

\begin{proof}
Let $w_1$ be a place of $K(y)$ above $v$, and let $w_2$ be a
place of $K(x)$ above $w_1$. Set
\[
\mathcal X':=\mathcal X_{\mathcal V}
  \times_{\mathcal V}\operatorname{Spec}O_{K_v},
\]
and let $F':\mathcal X'\to\mathcal X'$ be the induced morphism.
By properness, $x$ and $y$ extend uniquely to morphisms
\[
\widetilde x:\operatorname{Spec}O_{K(x)_{w_2}}\to\mathcal X',
\qquad
\widetilde y:\operatorname{Spec}O_{K(y)_{w_1}}\to\mathcal X',
\]
compatible with $F'$.
By hypothesis, $F'$ is étale at the reduction
of $x$ at $w_2$. Consider the base change diagram
\[\begin{tikzcd}
\mathcal{Z} \arrow[d] \arrow[r] & { \operatorname{Spec}O_{K(y)_{w_1}},} \arrow[d, "\widetilde y"] \\
\mathcal{X}' \arrow[r, "F'"]    & \mathcal{X}'.                                             
\end{tikzcd}\]
Let $z$ be the image of the closed point under the induced
morphism $\operatorname{Spec}O_{K(x)_{w_2}}\longrightarrow\mathcal Z$. 

We claim that $\mO_{\mc{Z},z}=O_{K(x)_{w_2}}$. In fact, by hypothesis and \cite[Proposition 17.8.2]{EGAIV4}, $\mathcal Z\to\operatorname{Spec}O_{K(y)_{w_1}}$ is étale at $z$. Since $O_{K(y)_{w_1}}$ is henselian, \cite[Tag~04GH]{stacks-project} implies that $\mathcal O_{\mathcal Z,z}$ is finite étale over $O_{K(y)_{w_1}}$. Then $\mathcal O_{\mathcal Z,z}$ is a DVR, whose fraction field is $K(x)_{w_2}$. Consequently, $\mathcal O_{\mathcal Z,z}=O_{K(x)_{w_2}}$. Thus $O_{K(x)_{w_2}}$ is finite étale over $O_{K(y)_{w_1}}$, so $w_2$ is unramified over $w_1$.
\end{proof}

\begin{lemma}\label{lem_unramifiedset}
With the notation as above, assume that $\alpha\in X(K)$ is a non-exceptional point and $K_{\infty}(f,\alpha)/K$ is an abelian extension. Then for all but finitely many finite places $v$ of $K$, there exists a set $\mc{P}^{(v)}\subset  f^{-\infty}(\alpha)$ which is abelian, almost unramified at $v$, and Zariski dense in $X$. Moreover,  $f\big(\mc{P}^{(v)}\big)\subset \mc{P}^{(v)}\cup \{f(\alpha)\}$.
\end{lemma}
\begin{proof}
Choose a closed subset $W \subset X$ of codimension $1$ such that $\Crit(f) \subset W$, and let $\mc{W} \subset \mc{X}$ be the Zariski closure of $W$. Set $D := \deg_L(W) = W \cdot L^{n - 1}$. Choose an integer $N \geq 1$ such that $2 d^{n-N} D < d - 1$.

Since $\alpha$ is non-exceptional, the backward orbit $f^{-\infty}(\alpha)$ is Zariski dense in $X$. We may therefore choose $\alpha' \in f^{-\infty}(\alpha)$ such that $\alpha' \notin \bigcup\limits_{i=1}^N f^i(W)$. Enlarging $S$ if necessary, for all $v \notin S$,
\[
\Crit(f_v)\subset \mc{W}_v,\ \text{ and}\ \ \ \mr{red}_v(\alpha') \notin \bigcup_{i=1}^N f_v^i(\mc{W}_v).
\]
Fix a finite place $v\notin S$,  write $\overline{\beta} := \mathrm{red}_{v,\mc{X}}(\beta)$ for $\beta \in X(\overline{K})$ for simplicity. Set
$Z := f_v(\mc{W}_v) \subsetneq \mc{X}_v$. For each $i \geq 0$, define
\[
\mathcal{P}_i :=
\left\{
\beta \in f^{-i}(\alpha')
\ \middle|\ 
\overline{\beta}, f_v(\overline{\beta}), \dots, f_v^i(\overline{\beta}) \notin Z
\right\},
\]
and
\[\mc{P}^{(v)}:=\{f(\alpha'),f^2(\alpha'),\dots, \alpha\}\cup\bigcup_{i\geq 1}\mc{P}_i .\]
Since $\mc{P}^{(v)}\subset f^{-\infty}(\alpha)$, $K(\beta)/K$ is abelian for all $\beta\in \mc{P}^{(v)}$. Moreover, if $\beta\neq \alpha$, $f(\beta)\in \mc{P}^{(v)}$. 

\medskip

\noindent\textbf{Claim 1.} The set $\mc{P}^{(v)}$ is almost unramified at $v$.

\noindent\textit{Proof of Claim 1.} Choose an integer $M>0$ such that the ramification index $e(v'/v)\leq M$ for all place $v'$ of $K(\alpha')$ above $v$. Let $\beta\in \mc{P}^{(v)}$ be a point.

If $\beta\notin \mc{P}_i$ for all $i\geq 0$, then $\beta=f^j(\alpha')$ for some $j>0$, which implies $K(\beta)\subset K(\alpha')$. It follows that $e(w/v)\leq M$ for all place $w$ of $K(\beta)$ above $v$.

If $\beta\in \mc{P}_i$ for some $i\geq 0$. Let $w$ be a place of $K(\beta)$ above $v$. For $0 \leq j \leq i-1$, the condition $f_v^{j+1}(\overline{\beta}) \notin Z = f_v(\mc{W}_v)$ implies that $f_v^j(\overline{\beta}) \notin \Crit(f_v)$. Let $w_j$ denote the restriction of $w$ to $K(f^j(\beta))$. Then each extension
\[
K(f^j(\beta))_{w_j} / K(f^{j+1}(\beta))_{w_{j+1}}
\]
is unramified by Lemma \ref{lem_unrami}, and hence $K(\beta)_{w_0} / K(\alpha')_{w_i}$ is unramified. Thus $e(w/v)=e(w_0/w_i)e(w_i/v)\leq M$. \bqed

\bigskip

\noindent \textbf{Claim 2.} For all $i\geq 0$, $|\mc{P}_i|\geq \frac{1}{2}d^{ni}$.

\medskip

\noindent\textit{Proof of Claim 2.} Let $\widetilde{\mathcal{P}}_i$ be the image of $\mc{P}_i$ under reduction. Then 
\[1 \leq |\widetilde{\mathcal{P}}_i| \leq |\mathcal{P}_i| \leq d^{ni}.\]
For $0 \leq i \leq N-1$, by the choice of $v$, we have $f_v^{-i}(\overline{\alpha'}) \cap Z = \varnothing$, and hence $f_v^{-i}(\overline{\alpha'})$ consists of $d^{ni}$ distinct points. In particular,
\[
|\mathcal{P}_{i}| = |\widetilde{\mathcal{P}}_{i}| = d^{ni},\qquad 0\leq i\leq N-1.
\]
Now let $i \geq N$. We estimate the number of points lying in $Z$. First,
\[
\deg_{\mathcal{L}_v}(Z)
\leq d^{n-1} \deg_{\mathcal{L}_v}(\mc{W}_v)
= d^{n-1} D.
\]
Hence
\[
\deg\big(f_v^i|_Z:Z\to f_v^i(Z)\big)\leq d^{(n-1)i} \deg_{\mathcal{L}_v}(Z)\leq d^{(i+1)(n-1)} D.
\]
It follows that among the points in $f_v^{-i}(\overline{\alpha'})$, at most $d^{(i+1)(n-1)}D$ lie in $Z$. On the other hand, for each $\overline{\beta} \in \widetilde{\mathcal{P}}_{i-1}$, since $\overline{\beta} \notin Z\supset f_v\big(\Crit(f_v)\big)$, the fiber $f_v^{-1}(\overline{\beta})$ consists of $d^n$ distinct points. Therefore,
\[
|\widetilde{\mathcal{P}}_i|\geq d^n |\widetilde{\mathcal{P}}_{i-1}|- d^{(i+1)(n-1)} D.
\]
It follows that
\[
\frac{|\widetilde{\mathcal{P}}_i|}{d^{ni}}
\geq \frac{|\widetilde{\mathcal{P}}_{N-1}|}{d^{n(N-1)}}- \sum_{j=N}^i d^{n-1-j} D
> 1 - \frac{d^{n-1-N}D}{1 - \frac{1}{d}}> \frac{1}{2},
\]
where the last inequality follows from the choice of $N$. Thus $|\mc{P}_i|\geq |\wt{\mc{P}}_i|\geq \frac{1}{2}d^{ni}$ for all $i\geq 0$. \bqed

\medskip Finally, we show that $\mathcal{P}^{(v)}$ is Zariski dense in $X$. Assume that $\mathcal{P}^{(v)}$ is contained in a proper closed subset $Y \subset X$. Enlarging $Y$ if necessary, we may assume $\mathrm{codim}(Y,X)=1$. Then
\[
|\mc{P}_i|\leq |f^{-i}(\alpha') \cap Y|
\leq \deg\big(f^i|_Y:Y\to f^i(Y)\big)
\leq d^{i(n-1)} \deg_L(Y).
\]
For $i$ sufficiently large, 
\[|\mathcal{P}_i|>\tfrac{1}{2} d^{ni}>d^{(n-1)i} \deg_L(Y).\] 
Contradiction. This completes the proof.
\end{proof}

The following theorem follows from the same strategy as Theorem \ref{thm_main}.
\begin{theorem}\label{thm_backorb}
Let $K$ be a number field, $X$ be a normal projective variety over $K$, and $f:X \to X$ be a polarized endomorphism. Assume $\alpha\in X(K)$ is a non-exceptional point and $K_{\infty}(f,\alpha)/K$ is an abelian extension. Then $f$ is exceptional and $\alpha\in \mr{Prep}(f)$.
\end{theorem}
\begin{proof} 
We first prove the following claim.

\noindent\textbf{Claim 1.} For all but finitely many finite places $v$ of $K$, there exists a bi-finite correspondence $\mf{C}_X^{(v)}\subset X\times X$ satisfying:
\begin{enumerate}
\renewcommand{\labelenumi}{(\theenumi)}
    \item $(f\times f)(\mf{C}_X^{(v)})=\mf{C}_X^{(v)}$;
    \item the reduction of $\mf{C}_X^{(v)}$ at $v$ has support $\Gamma_{\Phi_{q,\mc{X}_v}}\subset \mc{X}_v\times_{k_v}\mc{X}_v$.
\end{enumerate}
\noindent\textit{Proof of Claim 1.} There exists an invariant adelic line bundle $\overline{L}_f \in \widehat{\mathrm{Pic}}(X)_{\mathrm{nef}}$ such that $f^*\overline{L}_f \simeq d\,\overline{L}_f$. Enlarging $S$ if necessary, we may assume that every finite place $v\notin S$ is a model place for $\ov{L}$. By assumption and Lemma \ref{lem_unramifiedset}, we may choose $v\notin S$ with a set $\mc{P}^{(v)}\subset  f^{-\infty}(\alpha)$ which is abelian, almost unramified at $v$, and Zariski dense in $X$. Moreover, $f(\beta)\in \mc{P}^{(v)}$ for any $\beta\in \mc{P}^{(v)}\setminus \{\alpha\}$.

Fix an embedding $\overline{K} \hookrightarrow \mathbb{C}_v$, inducing a valuation $w$ on $\overline{K}$ extending $v$. Choose a Frobenius element $\sigma\in D_w\subset \Gal(\ov{K}/K)$. Consider the Zariski closure
\[
\mf{C}_X^{(v)} := \overline{\left\{ (\beta, \sigma(\beta)) \in (X \times X)(\overline{K}) \mid \beta \in \mathcal{P}^{(v)} \right\}}^{\mathrm{Zar}}
\subset X \times X.
\]
Since $f(\mathcal{P}^{(v)}) \subset \mathcal{P}^{(v)} \cup \{f(\alpha)\}$, we have $(f \times f)(\mf{C}_X^{(v)}) = \mf{C}_X^{(v)}$. As $\mathcal{P}^{(v)}$ is Zariski dense in $X$, every irreducible component of $\mf{C}_X^{(v)}$ dominates both factors.

Choose a sequence $(x_m)_{m \ge 1}$ in $\mathcal{P}^{(v)}$ such that the sequence $\big( (x_m, \sigma(x_m)) \big)_{m\geq1}$ is generic in $\mf{C}_X^{(v)}$. Then $(x_m)_{m\geq1}$ is also generic in $X$. We have 
\[\mf{C}_X^{(v)}=\Zar{\{(x_m,\sigma(x_m))\mid m\geq 1\}}.\]
By the choice of $\mc{P}^{(v)}$, the sequence $(x_m)_{m\geq1}$ is abelian, almost unramified at $v$, and satisfies $\lim\limits_{m \to \infty} h_{\overline{L}_f}(x_m) = 0$. By the proof of Theorem \ref{thm_constructZ}, $\mf{C}_X^{(v)}\subset X\times X$ is the desired bi-finite correspondence.\bqed

Now, applying Theorem \ref{thm_relation} to $\mathfrak{C}_X^{(v)}$, we obtain an equivalence relation $\mf{R}_X^{(v)}\subset X\times X$. Moreover, since $(f \times f)(\mf{C}_X^{(v)}) = \mf{C}_X^{(v)}$, the construction of $\mathfrak{R}_X^{(v)}$ implies that $(f \times f)(\mf{R}_X^{(v)}) = \mf{R}_X^{(v)}$.

By the same argument as in the proof of Theorem \ref{thm_main}, we obtain a normal projective variety $Y$ over $K$ of dimension $\dim X$, a finite surjective morphism $\pi: X \to Y$, and a polarized endomorphism $g: Y \to Y$ such that $\pi(x_m) \in \mathrm{Per}(g)$ for all but finitely many $m$. Moreover, we have $\deg(g) = q^{n}$.

By Lemma \ref{lem_descent} and Remark \ref{rmk_bifin_corr}, $f$ descends to a polarized endomorphism $f_Y : Y \to Y$ satisfying
\[
\pi \circ f = f_Y \circ \pi, 
\qquad
g \circ f_Y = f_Y \circ g.
\]
Moreover, we have $\deg(f_Y) = \deg(f) = d^{n}$. By the choice of the finite place $v$, we have $\gcd(d,q)=1$. It follows from Theorem \ref{thm_commuting} that $f_Y$ is exceptional. Hence $f$ is also exceptional by Theorem \ref{thm_semiconj}. Moreover, since $g(\mathrm{Prep}(f_Y)) \subset \mathrm{Prep}(f_Y)$, we deduce from \cite[Theorem 1.6]{YZ17} that
\[
\mathrm{Prep}(f_Y) = \mathrm{Prep}(g).
\]
Note that $ \mathrm{Prep}(f)=\pi^{-1}\big(\mathrm{Prep}(f_Y)\big)$ and $\pi(x_m) \in \mathrm{Prep}(g)=\mathrm{Prep}(f_Y)$ for all but finitely many $m$, it follows that $x_m \in \mathrm{Prep}(f)$ for all but finitely many $m$. As each $x_m$ lies in the backward orbit $f^{-\infty}(\alpha)$, we conclude that $\alpha \in \mathrm{Prep}(f)$.  
\end{proof}

Now, we are able to prove the main theorem of this paper.

\begin{theorem}[Theorem \ref{thm_vaeq1}]\label{thm_vaeq}
Let $K$ be a number field, $X$ be a normal projective variety over $K$, and $f: X\to X$ be a polarized endomorphism. Assume $\alpha\in X(K)$ is a non-exceptional point, then $K_{\infty}(f,\alpha)/K$ is virtually abelian if and only if $f$ is an exceptional map with CM and $\alpha\in \mr{Prep}(f)$.
\end{theorem}
\begin{proof}
Assume $K_{\infty}(f,\alpha)/K$ is virtually abelian. Lemma \ref{lem_va} implies that there exists a finite extension $K'$ of $K$ such that $K'_{\infty}(f_{K'},\alpha)/K'$ is abelian. By Theorem \ref{thm_backorb}, $f_{K'}$ is an exceptional map and $\alpha\in \mr{Prep}(f)$. There exists a finite extension $K''$ of $K'$ such that $f_{K''}$ is a $K''$-exceptional map. Then $f_{K''}$ is an exceptional map with CM by Theorem \ref{thm_Kexp}.

Conversely, assume $f$ is an exceptional map with CM and $\alpha\in \mr{Prep}(f)$. Then we may choose a finite extension $K'$ of $K$ such that $f_{K'}$ is a $K'$-exceptional map with CM. By Theorem \ref{thm_Kexp}, $K'_{\infty}(f_{K'},\alpha)/K'$ is virtually abelian. So $K_{\infty}(f,\alpha)$ is virtually abelian.
\end{proof}

\subsection*{One-dimensional case}\label{sec_1dim}
Let $K$ be a number field and let $f:\mathbb{P}^1_K \to \mathbb{P}^1_K$ be a rational map of degree $d \geq 2$. For $\alpha \in \mathbb{P}^1(K)$, define $K_n^{\ab}(f,\alpha)$ to be the extension of $K$ generated by all points in $f^{-n}(\alpha)\cap \mathbb{P}^1(K^{\ab})$. Then we have a tower
\[
K = K_0^{\ab}(f,\alpha) \subset K_1^{\ab}(f,\alpha) \subset \cdots,
\]
and set
\[
K_{\infty}^{\ab}(f,\alpha) := \bigcup_{n=1}^{\infty} K_n^{\ab}(f,\alpha).
\]
If $\alpha$ is a non-exceptional point and $K_{\infty}(f,\alpha)/K$ is abelian, then $K_{\infty}^{\ab}(f,\alpha)/K$ is an infinite extension.

Recall that in the proof of Theorem \ref{thm_backorb}, a key step is applying Lemma \ref{lem_unramifiedset}, which constructs a set $\mc{P}^{(v)}\subset f^{-\infty}(\alpha)$ that is abelian, almost unramified at $v$, and Zariski dense in $X$. In the one-dimensional case, \cite[Theorem A]{FOZ24} provides an alternative approach to obtain such a result. Combining this with the classification of exceptional maps on $\mathbb{P}^1$, we obtain the following result, which answers Conjecture \ref{conj_P1}.

\begin{theorem}\label{thm_onedim}
Let $K$ be a number field and let $f \in K(x)$ be a rational map of degree $d \geq 2$, and let $\alpha \in \mathbb{P}^1(K)$. Assume that $f$ is not a Lattès map. Then $K_{\infty}^{\ab}(f,\alpha)/K$ is an infinite extension if and only if the pair $(f,\alpha)$ is $K^{\ab}$-conjugate to one of the following:
\begin{enumerate}
\renewcommand{\labelenumi}{(\theenumi)}
    \item $(x^{\pm d}, \zeta)$, where $\zeta$ is a root of unity in $\overline{K}$;
    \item $(\pm T_d(x),  \zeta + \zeta^{-1})$, where $\zeta$ is a root of unity in $\overline{K}$.
\end{enumerate}
\end{theorem}

\begin{proof}
If $(f,\alpha)$ is $K^{\ab}$-conjugate to $(x^{\pm d},\zeta)$ or $(\pm T_d(x), \zeta + \zeta^{-1})$, where $\zeta$ is a root of unity in $\overline{K}$, then it is clear that $f^{-\infty}(\alpha)\subset \mb{P}^1(K^{\ab})$ and $\alpha$ is not an exceptional point for $f$. Therefore $K_{\infty}^{\ab}(f,\alpha)/K$ is an infinite extension. It remains to prove the converse.

Assume that $K_{\infty}^{\ab}(f,\alpha)/K$ is infinite. By \cite[Theorem A]{FOZ24}, this extension ramifies at only finitely many places of $K$. It follows that for all but finitely many finite places $v$ of $K$, the set
\[
\mathcal{P}^{(v)} := \left\{ \beta \in f^{-\infty}(\alpha) \;\middle|\; K(\beta)/K \text{ is abelian and unramified at } v \right\}
\]
is infinite. In particular, $\mathcal{P}^{(v)}$ is Zariski dense in $\mathbb{P}^1_K$. By the same argument of the proof of Theorem \ref{thm_backorb}, it follows that $f$ is exceptional and $\alpha \in \mathrm{Prep}(f)$.

Since $f$ is not a Lattès map, it is $\ov{K}$-conjugate to either $x^{\pm d}$ or $\pm T_d(x)$. Let $\varphi \in \mathrm{PGL}_2(\ov{K})$ be such that $\varphi \circ f\circ\varphi^{-1} = g$, where $g$ is one of these maps. Choose $\beta \in f^{-\infty}(\alpha)$ such that $\varphi(\beta) \neq 0,\infty$. Then $\varphi(\beta) \in \mathrm{Prep}(g)$, and hence
\[
\varphi(\beta) =
\begin{cases}
\zeta, & \text{if } g(x)=x^{\pm d},\\
\zeta + \zeta^{-1}, & \text{if } g(x)=\pm T_d(x),
\end{cases}
\]
for some root of unity $\zeta \in \overline{K}$. In particular, $\varphi(\beta) \in \mathbb{P}^1(K^{\ab})$ for all $\beta \in f^{-\infty}(\alpha)$.

Choose three distinct points $\beta_1,\beta_2,\beta_3 \in f^{-\infty}(\alpha)\cap  \mathbb{P}^1(K^{\ab})$, and set $\gamma_i := \varphi(\beta_i)$. There exists $\sigma \in \mathrm{PGL}_2(K^{\ab})$ such that $\sigma(\beta_i) = \gamma_i$ for $i=1,2,3$. On the other hand, $\varphi$ is the unique element of $\mathrm{PGL}_2(\ov{K})$ satisfying $\varphi(\beta_i)=\gamma_i$ for $i=1,2,3$. Hence $\sigma = \varphi$, and therefore $\varphi \in \mathrm{PGL}_2(K^{\ab})$.

Consequently, $(f,\alpha)$ is $K^{\ab}$-conjugate to $(x^{\pm d},\zeta)$ or $(\pm T_d(x), \zeta+\zeta^{-1})$, where $\zeta$ is a root of unity in $\overline{K}$.
\end{proof}

\bibliography{reference}
\bibliographystyle{alpha}

\noindent \small{Address: \textit{Institute for Theoretical Sciences, Westlake University, Hangzhou 310030, China}}

\noindent \small{Email: \texttt{jizhuchao@westlake.edu.cn}}

\noindent \small{Address: \textit{School of Mathematical Sciences, Peking University, Beijing 100871, China}}

\noindent \small{Email: \texttt{soyo999@pku.edu.cn}}

\noindent \small{Address: \textit{Beijing International Center for Mathematical Research, Peking University, Beijing 100871, China}}

\noindent \small{Email: \texttt{xiejunyi@bicmr.pku.edu.cn}}
\end{document}